\documentclass[11pt,final]{amsart}

\usepackage{amsmath,amssymb,enumerate}

\usepackage[T1]{fontenc}
\usepackage[all,cmtip]{xy}

\usepackage{eurosym}
\usepackage[german,english]{babel}
\usepackage{amstext}
\usepackage{amsmath}
\usepackage{amsfonts}
\usepackage{latexsym}
\usepackage{ifthen}
\usepackage{mathrsfs}

\usepackage{tikz}
\usetikzlibrary{arrows.meta,calc}

\usepackage{xcolor}

\usepackage{hyperref}
\hypersetup{
	colorlinks=true,
	linkcolor=red,
	citecolor=blue}

\xyoption{all}
\newcommand{\Ysing}{{S}}
\newcommand{\ev}{{\operatorname{ev}}}

\newcommand{\Gr}{{\operatorname{Gr}}}
\newcommand{\mer}{{\operatorname{mer}}}
\newcommand{\red}{{\operatorname{red}}}

\newcommand{\std}{{\operatorname{std}}}

\newcommand{\vol}{{\operatorname{vol}}}
\newcommand{\res}{{\rm res}}

\newcommand{\pr}{{\rm pr}}

\newcommand{\GL}{{\rm GL}}

\newcommand{\tensor}{{ \,\otimes\, }}
\newcommand{\gp}{{\mathrm{gp}}}

\newcommand{\tOm}{{\widetilde{\Omega}}}
\newcommand{\ratl}{{{\dashrightarrow}}}
\newcommand{\into}{{{\hookrightarrow}}}
\newcommand{\isom}{{{\, \cong \,}}}

\newcommand{\Coh}{{\rm Coh}}
\newcommand{\reg}{{\rm reg}}

\newcommand{\Spec}{{\rm Spec}}
\newcommand{\Art}{{\rm Art}}
\newcommand{\Set}{{\rm Set}}
\newcommand{\ft}{{\rm ft}}
\newcommand{\lt}{{\rm lt}}

\newcommand{\id}{{\rm id}}

\newcommand{\E}{{\mathcal E}}

\newcommand{\rk}{{\rm rk}}

\newcommand{\sK}{{\mathcal K}}

\newcommand{\sHom}{{\sH om}}

\DeclareMathOperator{\nd}{nd}
\DeclareMathOperator{\mult}{mult}
\DeclareMathOperator{\Supp}{Supp}

\theoremstyle{plain}

\newtheorem{theorem}{Theorem}[section]
\newtheorem{proposition}[theorem]{Proposition}
\newtheorem{lemma}[theorem]{Lemma}
\newtheorem{corollary}[theorem]{Corollary}

\theoremstyle{definition}

\newtheorem{definition}[theorem]{Definition}
\newtheorem{construction}[theorem]{Construction}
\newtheorem{remark}[theorem]{Remark}
\newtheorem{remarks}[theorem]{Remarks}
\newtheorem{example}[theorem]{Example}
\newtheorem{notation}[theorem]{Notation}
\newtheorem{setup}[theorem]{Setup}

\theoremstyle{remark}

\newenvironment{remark*}{{\bf Remark.}}{}
\newenvironment{remarks*}{{\bf Remarks.}}{}
\newenvironment{example*}{{\bf Example.}}{}
\newenvironment{assumption*}{{\bf Assumption.}}{}

\newcommand{\R}{\ensuremath{\mathbb{R}}}
\newcommand{\Q}{\ensuremath{\mathbb{Q}}}
\newcommand{\Z}{\ensuremath{\mathbb{Z}}}
\newcommand{\C}{\ensuremath{\mathbb{C}}}

\newcommand{\N}{\ensuremath{\mathbb{N}}}
\newcommand{\PP}{\ensuremath{\mathbb{P}}}

\newcommand{\merom}[3]{\ensuremath{#1\colon #2 \dashrightarrow #3}}

\newcommand{\holom}[3]{\ensuremath{#1\colon #2  \rightarrow #3}}
\newcommand{\fibre}[2]{\ensuremath{#1^{-1} (#2)}}

\newcommand{\LinAut}{\operatorname{LinAut}}

\newcommand{\scrC}{{\mathscr C}}
\newcommand{\scrK}{{\mathscr K}}
\newcommand{\scrL}{{\mathscr L}}

\newcommand\sA{{\mathcal A}}

\newcommand\sFxt{{\mathcal Ext}}
\newcommand\sF{{\mathcal F}}
\newcommand\sG{{\mathcal G}}
\newcommand\sH{{\mathcal H}}

\newcommand\sL{{\mathcal L}}

\newcommand\sO{{\mathcal O}}

\newcommand\sX{{\mathcal X}}
\newcommand\sY{{\mathcal Y}}
\newcommand\sC{{\mathcal C}}

\newcommand\Hom{{\rm Hom}}
\newcommand\Ext{{\rm Ext}}

\newcommand\sD{{\mathcal D}}

\DeclareMathOperator*{\pic}{Pic}

\DeclareMathOperator*{\supp}{Supp}
\DeclareMathOperator*{\Aut}{Aut}
\DeclareMathOperator*{\Bim}{Bim}

\DeclareMathOperator*{\Rat}{RatCurves^n}

\newcommand{\Chow}[1]{\ensuremath{\mbox{\rm Chow}(#1)}}

\newcommand{\cond}{\ensuremath{\Xi}}
\newcommand{\nonn}{\ensuremath{\Theta}}

\newcommand{\Curv}{\ensuremath{\Gamma}}
\newcommand{\nons}{\ensuremath{{\rm nons}}}
\newcommand{\sing}{\ensuremath{{\rm sing}}}

	\author{Andreas H\"oring}
	\address{Andreas H\"oring, Universit\'e C\^ote d4Azur, CNRS, LJAD, France}
	\email{andreas.hoering@univ-cotedazur.fr}
	
	\author{Vladimir Lazi\'c}
	\address{Vladimir Lazi\'c, Fachrichtung Mathematik, Campus, Geb\"aude E2.4, Universit\"at des Saarlandes, 66123 Saarbr\"ucken, Germany}
	\email{lazic@math.uni-sb.de}

	\author{Christian Lehn}
	\address{Christian Lehn, Fakult\"at f\"ur Mathematik, Ruhr-Universit\"at Bochum, Universit\"atsstraße 150, Postfach IB 45, 44801 Bochum, Germany}
	\email{christian.lehn@rub.de}

\title{A semiampleness result for hyperk\"ahler fourfolds and smoothings of Lagrangian subvarieties}
\date{\today} 

\begin{document}

\begin{abstract} We give a geometric proof of the following statement: let $X$ be a smooth hyperk\"ahler fourfold
such that the Picard group is generated by a line bundle $L$ that is nef, non-trivial, and not big. If the Kodaira dimension of $L$ 
is non-negative, some positive multiple defines a Lagrangian fibration. The proof requires a number of tools which are of independent interest, including an analogue of Voisin's unobstructedness result and a smoothing result for singular Lagrangian subvarieties.
\end{abstract}


% Primary:

% 14J42 		Holomorphic symplectic varieties, hyper-Kähler varieties
% 32J25 		Transcendental methods of algebraic geometry (complex-analytic aspects)

% Secondary:

% 14D06 		Fibrations, degenerations in algebraic geometry
% 14D15 		Formal methods and deformations in algebraic geometry
% 32G10 		Deformations of submanifolds and subspaces
% 32J27   	Compact Kähler manifolds: generalizations, classification
% 53C26   	Hyper-Kähler and quaternionic Kähler geometry, "special'' geometry

\maketitle

	%Table of Contents
	\begingroup
		\hypersetup{linkcolor=black}
		\setcounter{tocdepth}{1}
		\tableofcontents 
	\endgroup

%----------------------------------------------------
%----------------------------------------------------
\section{Introduction}
%----------------------------------------------------
%----------------------------------------------------

Hyperk\"ahler manifolds (also known as irreducible holomorphic symplectic manifolds) have been studied by many geometers over the last decades. The properties of line bundles play a central role in this study, these are governed by the Beauville--Bogomolov form. A line bundle $L$ on a hyperk\"ahler manifold $X$ is said to be \emph{parabolic} if its Beauville--Bogomolov square
$q_X(L)$ is equal to zero. In this paper, we show that for very general non-projective hyperk\"ahler fourfolds the geometry
of parabolic divisors can be described completely. As a consequence we obtain:

\begin{theorem} \label{theorem main semiampleness}
Let $X$ be a hyperkähler fourfold such that $\pic{X} \isom \Z L$ where $L$ is a nef parabolic line bundle with $\kappa(X, L) \geq 0$. Then $L$ is semiample, i.e., some positive multiple defines a Lagrangian fibration on $X$.
\end{theorem}

The idea of our proof is quite natural: up to replacing $L$ by some multiple, we can assume that there exists an effective
divisor $D \in |L|$. We expect to have a Lagrangian fibration 
$$
\holom{f:=\varphi_{|mL|}}{X}{B}
$$ 
onto a surface, so we should have
a fibration
$$
\holom{\phi:=f|_D}{D}{U}
$$
onto a curve $U \subset B$ that is defined by some multiple of the restriction $L|_D$. If the curve $U$ is not contained in the discriminant locus of the Lagrangian fibration $f$ the general $\phi$-fibre is an abelian surface. The deformations of Lagrangian abelian subvarieties are unobstructed by a result of Voisin~\cite{Voi92}. Building on earlier work~\cite{GLR13}, this leads to the existence of the Lagrangian fibration. If the curve $U$ is contained in the discriminant locus, Hwang--Oguiso~\cite{HO09,HO11} still gives a precise description of the normalisation of general singular fibres.
Our strategy is to reconstruct and refine this information without assuming the existence of a Lagrangian fibration, but instead use that we work on a non-projective hyperk\"ahler manifold. We first show (see Propositions~\ref{prop:normalisation0} and~\ref{prop:normalisation}) that the normalisation $\nu_D: D^\nu \rightarrow D$ admits a fibration
$$
\holom{\varphi}{D^\nu}{\Curv}
$$
onto a curve $\Curv$ such that the general fibre $F$ is either an abelian surface or a $\PP^1$-bundle over an elliptic curve. In the abelian case, it is not difficult to see that $\nu_D(F) \isom F$ is a Lagrangian submanifold of $X$, so Voisin's result applies. In the case of a $\PP^1$-bundle, the divisor $D$ is always singular along the image $\nu_D(F) \subset D$, and in Section~\ref{section:uniruled} we describe the singularities
of the divisor $D$  near the surface $\nu_D(F)$. It turns out that they are either normal crossing singularities of multiplicity two
or singularities of Kodaira type $II, III, IV$ (cf. Definition~\ref{defn:kodaira-type}). The second case seems to be especially difficult since these singularities are not slc, so general extension theorems do not apply. However, they admit rather special log-resolutions, which we call plt resolutions. In Section~\ref{section:extending}, we show how to combine these resolutions with the Demailly--Hacon--P\v aun extension theorem~\cite{DHP13} and algebraic approximation to show Theorem~\ref{theorem main semiampleness} in this case.

This leaves us with the case where $D$ has normal crossing singularities of multiplicity two near the surface $\nu_D(F)$. In Section~\ref{section:descending}, we show that the fibration $\holom{\varphi}{D^\nu}{\Curv}$ descends along the normalisation to give the expected fibration
$
\holom{\phi}{D}{U}.
$
The general fibre of this fibration is a Lagrangian surface $Y\subset X$ with normal crossing singularities of multiplicity two. Our goal is then to smooth this subvariety to a Lagrangian torus. This is achieved in Section~\ref{section unobstructedness}; see Theorem~\ref{theorem normal crossing cover corank one smoothing}. One of the ingredients is an analogue of Voisin's unobstructedness result for these singular Lagrangian subvarieties proven in Theorem~\ref{theorem normal crossing cover unobstructedness}. Others ingredients are a detailed structure theory of ``corank one'' Lagrangian normal crossing subvarieties with trivial canonical sheaf established in 	Section~\ref{section snc fibers} and the connection between Lagrangian normal crossing subvarieties and logarithmic structures we lay out in Section~\ref{section lagrangian normal crossing deformations}.

This part of our argument actually works in arbitrary dimension. More precisely, our second main result is:

\begin{theorem}\label{theorem main fibration}
Let $X$ be a compact hyperkähler manifold of dimension $2n$, and let $Y \subset X$ be a Lagrangian normal crossing subvariety satisfying the following properties.
\begin{enumerate}
	\item $h^0(Y,\tOm_Y)=n-1$.
	\item $\omega_Y$ is torsion.
	\item $-K_{Y^\nu}$ is nef, where $\nu\colon Y^\nu \to Y$ is the normalisation.
\end{enumerate}
Then there is a Lagrangian fibration $f\colon X\to B$ such that $Y$ is the reduction of a fibre, i.e. $Y=f^{-1}(t)_\red$ for some $t\in B$.
\end{theorem}

Here, $\tOm_Y$ denotes the sheaf of torsion free differentials (Definition~\ref{definition torsion free differentials}). The term ``corank one'' comes from the theory of degenerations of abelian varieties and denotes degenerations where the toric rank is equal to one --- for us, it is simply subsumed in condition (1) of the theorem, which in our context means that locally trivial deformations of $Y$ inside $X$ sweep out a divisor.

As mentioned before, the main ingredient in the proof of Theorem~\ref{theorem main fibration} is a smoothability result. In many instances, such results have been obtained by means of logarithmic deformation theory; see e.g.~\cite{KN94,FFR21,CLM23}. Log structures also play a role in our approach, since they allow to globalise the relation between the normal bundle of a normal crossing Lagrangian subvariety and its log differentials. Somewhat surprisingly though, we do not need to resort to logarithmic deformation theory, which makes the proof much lighter. From a conceptual viewpoint, this is due to the existence of the ambient symplectiv variety.
With the deformation theoretic groundwork, the proof of the existence of a Lagrangian fibration is then completed using~\cite{GLR13,HW13,Mat14}. 

Both, Theorem~\ref{theorem main semiampleness} and Theorem~\ref{theorem main fibration} are proven in Section~\ref{section proof main result} by assembling the results obtained and tools developed in the preceeding sections. We believe that many of them are of independent interest.

Coming back to Theorem~\ref{theorem main semiampleness}, its assumptions on $X$ have to be thought of as the situation we plan to reduce the general case to rather than as a restriction. The remaining task to obtain an unconditional semiampleness result in dimension four is to prove nonvanishing for a very general deformation. Concretely, one deduces without too much difficulty the following consequence of Theorem~\ref{theorem main semiampleness}.

\begin{corollary} \label{corollary main}
Let $X$ be a hyperkähler fourfold and let $L$ be a nef parabolic line bundle on $X$. Suppose that for some general deformation $(X',L')$ of $(X,L)$ we have $\kappa(X', L') \geq 0$. Then $L$ is semiample.
\end{corollary}

{\bf Related work.} 
This work was announced in several talks in the spring 2026. When most of the writing
was done, preprints by OpenAI~\cite{OpenAI} and Engel--Mauri~\cite{EM26} on the Hyperk\"ahler SYZ conjecture were posted. A first superficial comparison suggests that our arguments appear to be substantially different from theirs. 

{\bf Acknowledgements.} 
We would like to thank Klaus Hulek and Christian Sevenheck for valuable discussions.

All three authors were supported by the ANR-DFG project ``Positivity on K-trivial varieties'', ANR-23-CE40-0026 and DFG Project-ID 530132094. H\"oring was partially supported  by the France 2030 investment plan managed by the ANR, as part of the Initiative of Excellence of Universit\'e C\^ote d'Azur, reference ANR-15-IDEX-01. Lazi\'c was supported by the DFG Project-ID 286237555 - TRR 195. Lehn was supported by the DFG Project-ID 550535392.

%----------------------------------------------------
%----------------------------------------------------
\section{Preliminaries}\label{section preliminaries}
%----------------------------------------------------
%----------------------------------------------------

%----------------------------------------------------
\subsection*{Notation and conventions}
%----------------------------------------------------
A complex space is always assumed to be separated, finite-dimensional and to have finitely many irreducible components. A complex variety is a reduced complex space that is irreducible (except in the case of a normal crossing variety, cf. Definition~\ref{defn:nc+triple}). A manifold is a complex variety that is smooth. A curve, respectively a surface, is a reduced \emph{compact} complex space of pure dimension one, respectively two.

A dimension of a reducible compact complex space is defined as the maximal dimension of its irreducible components.

If $X$ is a complex space, a \emph{$\Q$-line bundle} on $X$ is a formal $\Q$-linear combination of line bundles on $X$.

On a normal complex variety $X$, the canonical sheaf $\omega_X$ is a reflexive sheaf which is defined as the pushforward of the determinant bundle of the sheaf of holomorphic $1$-forms on the regular part of $X$. The variety $X$ is \emph{$\Q$-factorial} if every prime Weil divisor $D \subseteq X$ is $\Q$-Cartier and if there exists a positive integer $m$ such that the reflexive power $(\omega_X^{\otimes m})^{**}$ is a locally free sheaf. In that case, $\omega_X$ and every $\Q$-divisor on $X$ define $\Q$-line bundles. We omit the double dual when the reflexive sheaves involved are $\Q$-line bundles.

As in~\cite[\S II.4]{Nak04}, we use the symbol $K_X$ in a formal way, i.e.\ we write $\omega_X\isom\sO_X(K_X)$ even if there might not exist a canonical Weil divisor with that property. To avoid confusion with Kähler forms, we use the notation $K_X$ instead of $\omega_X$. This allows to apply additive notation: in particular, on a $\Q$-factorial pair $(X,\Delta)$, the adjoint canonical $\Q$-line bundle $\omega_X\otimes\sO_X(\Delta)$ corresponds to the formal class $K_X+\Delta$, as is standard in birational geometry. For an integer $m$ we have \mbox{$\sO_X(mK_X)\isom\omega_X^{\otimes m}$}.

We denote the Kodaira dimension of an $\R$-Cartier divisor $D$ (respectively $\Q$-line bundle $\sL$) on a normal complex variety $X$ by $\kappa(X, D)$ (respectively $\kappa(X,\sL)$); see~\cite[\S II.3]{Nak04} for details.

We use the standard notions of singularities in the MMP as explained in~\cite{KM98, Kol13}. We use $\equiv$ for numerical equivalence of divisors and $\sim_\Q$ for $\Q$-linear equivalence of divisors.

Given reflexive sheaves $\sF, \sG$ on a normal complex space, we denote $\sF^{\boxtimes m} := (\sF^{\otimes m})^{**}$ and $\sF \boxtimes \sG := (\sF \otimes \sG)^{**}$ for the reflexive tensor products.

If $D=\sum d_i D_i$ is an $\R$-divisor on a complex space $X$, we denote by $\lfloor D\rfloor:=\sum\lfloor d_i\rfloor D_i$ the round-down of $D$.

For a complex space $Y$, the tangent sheaf is defined as $T_Y:=\Omega_Y^*$ and the $T^1$-sheaf is $T^1_Y:=\sFxt^1(\Omega_Y,\sO_Y)$; see~\cite[Appendix~B]{Ser06}. Also, if $Y$ is a closed subspace of another complex space $X$ and $I \subset \sO_X$ denotes the ideal sheaf of $Y$, the normal sheaf of $Y$ in $X$ is defined as $N_{Y/X}:=\sHom_{\sO_Y}(I/I^2,\sO_Y)$.

\begin{definition} \label{defn:almost-holomorphic}
Let $M$ be a reduced equidimensional compact complex space. A rational map $\merom{\gamma}{M}{N}$ to a complex variety $N$ is \emph{almost holomorphic} if there exists a Zariski dense open subset $M_\circ \subseteq M$ such that the map $\gamma|_{M_\circ}$ is holomorphic and proper. If $G \subseteq M$ is a subvariety of $M$ such that $G\cap M_\circ\neq\emptyset$, then we define
$$ \gamma(G) := \overline{\gamma|_{M_\circ}(G\cap M_\circ)}. $$
\end{definition}

\begin{remark} \label{shrinking-Zariski}
Let $B$ an algebraic variety (resp. irreducible and reduced complex space). We say that we shrink
$B$ (in the Zariski topology) if we replace $B$ by $B_\circ$ where $B \setminus B_\circ \subsetneq B$ is an algebraic set (resp. an analytic subset). 
\end{remark}

\subsection{Gorenstein spaces}

A complex space is \emph{Gorenstein} if it is Cohen--Macaulay and if the dualising sheaf is locally free. For the reader's convenience we recall the normalisation formula for Gorenstein spaces.

\begin{proposition}\label{normalisation-formula}
Let $X$ be a reduced Gorenstein complex space with normalisation $\holom{\nu}{X^\nu}{X}$. Then there exists an effective Weil divisor $\cond$, called the \emph{conductor}, such that $\nu$ is an isomorphism on $X^\nu \setminus \cond$, and we have
$$ K_{X^\nu} \sim \nu^* K_X -\cond. $$
\end{proposition}

\begin{proof}
This is shown in the algebraic setting in~\cite[\S 2]{Rei94}, but the arguments work as soon as dualising sheaves exist. Since $X$ is Cohen--Macaulay, the locus where~$\nu$ is not an isomorphism is exactly given by the conductor. 
\end{proof}

\begin{lemma}\label{lem:intersectionmult}
Let $X$ be a nonsingular variety, let $V$ and $W$ be closed subvarieties of $X$ which intersect properly, and let $Z$ be an irreducible component of the scheme-theoretic intersection of $V$ and $W$. Assume that the multiplicity of $Z$ in the scheme-theoretic intersection of $V$ and $W$ is $1$. Then $V$ and $W$ are smooth at a general point of $Z$ and they intersect transversally at a general point of $Z$.
\end{lemma}

\begin{proof}
The proof is a combination of~\cite[Tag 0B1I]{Stacks} and~\cite[Proposition 8.2, Remark 8.2 and \S20.4]{Ful98}.
\end{proof}

\begin{lemma}\label{lemma:Nagata}
Let $X$ be a Cohen--Macaulay variety and let $D$ be a subvariety of codimension $1$ on $X$ such that $X$ is not normal at a general point of $D$. Then the multiplicity of $X$ at a general point of $D$ is at least $2$.
\end{lemma}

\begin{proof}
The variety $X$ is singular at a general point of $D$ by assumption, hence the lemma follows from a theorem of Nagata; see for instance~\cite[Example 4.3.5(d)]{Ful98}.
\end{proof}

Recall also the following statement:

\begin{lemma} \label{lemma:birational-curves}
Let $\merom{\mu}{X'}{X}$ be a bimeromorphic map between normal compact Kähler spaces. Then for every point $x \in X$ the fibre $\fibre{\mu}{x}$ is chain connected by projective curves.
\end{lemma}
\begin{proof}
By Chow's lemma~\cite[Corollary~2]{Hir75}, the bimeromorphic map is dominated by a projective morphism, so every irreducible component of a fibre is dominated by a projective scheme.
\end{proof}

\subsection{Line bundles and $(1,1)$-classes}

\begin{definition}\label{definitionnumericaldimension}
Let $X$ be a compact Kähler manifold and let $\alpha \in N^1(X)$ be a nef class. The numerical dimension of $\alpha$ is defined by 
$$ \nd (\alpha) := \max \{k \in \N \mid \alpha^k\neq 0 \text{ in } H^{2k}(X,\R)\}. $$
\end{definition}

\begin{remark} \label{remarknumericaldimension} 
A nef class $\alpha$ is big if and only if $\int_X \alpha^{\dim X} >0$~\cite[Theorem~0.5]{DP04} which is of course equivalent to $\nd(\alpha)=\dim X$. 

By~\cite[Proposition~6.21]{Dem12} the cohomology class $\alpha^{\nd (\alpha)}$ can be represented by a non-zero closed positive $(\nd (\alpha),\nd (\alpha))$-current $T$. Therefore, 
\[
\alpha^{\nd (\alpha)} \cdot \omega_X^{\dim X - \nd (\alpha)} > 0
\]
for any Kähler class $\omega_X$ on $X$.
\end{remark}

\begin{definition}
Let $X$ be a compact Kähler space and let $L$ be a line bundle on~$X$. We say that $L$ is \emph{cohomologically trivial} if $c_1(L)=0$ in $H^2(X, \R)$.
\end{definition}

\begin{remark} \label{remark:num-trivial}
On a projective scheme, a line bundle is cohomologically trivial if and only if $L \cdot C=0$ for every curve $C$ on $X$. This is no longer true on a compact Kähler manifold. Compact K\"ahler surfaces of algebraic dimension one give a counterexample.
\end{remark}

\begin{definition}
Let $X$ be a compact Kähler space with a surjective morphism 
$$
\holom{\gamma}{X}{X'}
$$ 
onto a compact Kähler space $X'$. A line bundle $L$ on $X$ is \emph{$\gamma$-cohomologically trivial} if the restriction of $L$ to every fibre is cohomologically trivial.
\end{definition}

\begin{remark} \label{remark:rel-num-trivial}
If the fibres of $\gamma$ are projective,  the line bundle $L$ is $\gamma$\nobreakdash-cohomo\-logically trivial if and only if $c_1(L)$ has degree zero on every curve contracted by $\gamma$.
\end{remark}

\begin{lemma} \label{lemma:invariance-degree-zero}
Let $X$ be a normal compact Kähler space with klt singularities and let $\gamma\colon X\to X'$ be a $K_X$-negative elementary contraction. Let $\alpha$ be a $(1,1)$-class on $X$ that has degree zero on the contracted extremal ray, and let $\alpha'$ be a $(1,1)$-class on $X'$ such that $\alpha = \gamma^* \alpha'$~\cite[Lemma 3.3]{HP16}. Then $\alpha$ is nef if and only if $\alpha'$ is nef, and in this case $\nd(\alpha)=\nd(\alpha')$.

Moreover, we have $\alpha \cdot C=0$ for every curve $C \subseteq X$ if and only if $\alpha' \cdot C'=0$ for every curve $C' \subseteq X'$.
\end{lemma}

\begin{proof}
The first statement follows from~\cite[Lemma 2.38]{DHP24}.

Now we show the second statement. It is clear by the projection formula that $\alpha \cdot C=$ for every curve $C \subseteq X$ if $\alpha' \cdot C'=0$ for every curve $C' \subseteq X'$. Vice versa, assume that the property holds on $X$, and let $C' \subseteq X'$ be an irreducible curve. Since $C'$ is compact, it must be projective. The elementary contraction $\gamma$ is a projective morphism, so $\fibre{\gamma}{C'}$ is a projective scheme. In particular, it contains a curve $C \subseteq \fibre{\gamma}{C'} \subseteq X$ such that $\gamma(C)=C'$. Thus, we can again use the projection formula.
\end{proof}

As a consequence, we obtain:

\begin{lemma} \label{lemma:invariance-degree-zero-MMP}
Let $X$ be a normal compact Kähler space with klt singularities and let $\merom{\gamma}{X}{X'}$ be a sequence of steps in a $K_X$-MMP. Let $\sL$ be a line bundle on $X$ that has degree zero on every curve on $X$. Then there exists a line bundle $\sL'$ on $X'$ with the following properties:
\begin{enumerate}[\normalfont (a)]
\item $\sL'$ has degree zero on every curve on $X'$,
\item if $(p,q)\colon W\to X\times X'$ is a resolution of indeterminacies of $\gamma$, then $p^*\sL\isom q^*\sL'$,
\item $\kappa(X,\sL)=\kappa(X',\sL')$ and $\nd(\sL)=\nd(\sL')$,
\item $\sL$ is semiample if and only if $\sL'$ is semiample.
\end{enumerate}
\end{lemma}

\begin{proof}
We may assume that the MMP is a composition of divisorial contractions or flips $\gamma_i\colon X_i\dashrightarrow X_{i+1}$ for $i=1,\dots,k-1$, where $X_1:=X$ and $X_k:=X'$. Thus, we can prove the lemma one step of the MMP at the time. Therefore, we may assume that $\gamma$ itself is a divisorial contraction or a flip.

Assume that $\gamma$ is a flip; the case of a divisorial contraction and of a contraction of fibre type is analogous. Let $\pi\colon X\to Z$ be the corresponding flipping contraction and let $\pi'\colon X'\to Z$ be the corresponding flipped contraction. Since $\sL$ has degree zero on every curve flipping curve, by the Cone theorem~\cite[Theorem~4.12(3)(b)]{Nak87} there exists a line bundle $\mathcal M$ on $Z$ such that $\sL\isom\pi^*\mathcal M$. Set $\sL':=(\pi')^*\mathcal M$. Then (b), (c) and (d) are clear, whereas part (a) follows by applying Lemma~\ref{lemma:invariance-degree-zero} to the map $\pi$ and then to the map $\pi'$.
\end{proof}

\begin{lemma} \label{lemma-unique-divisor}
Let $X$ be a compact Kähler manifold such that $\pic(X)\otimes\Q \isom \Q \sL$, where $\sL$ is a line bundle on $X$. If $\kappa(X, \sL)=0$, then $X$ contains a unique prime divisor~$D$, and there exists a rational number $r>0$ such that $\sL\isom_\Q\sO_X(rD)$.
\end{lemma}

\begin{proof}
Since $\kappa(X, \sL)=0$, there exists a prime divisor $D$ on $X$, and clearly the last statement of the lemma holds since $\pic(X)\otimes\Q \isom \Q \sL$.

Let $D'$ be another prime divisor on $X$. Then as above, there exists a rational number $r'>0$ such that $\sL\isom_\Q\sO_X(r'D')$. If $D\neq D'$, then $\kappa(X, L)>0$, a contradiction which finishes the proof.
\end{proof}

Finally, let us recall the following result.

\begin{theorem} \label{thm:3fold-abundance}
Let $(X, \Delta)$ be a normal $\Q$-factorial compact K\"ahler threefold with at most klt singularities such that $K_X+\Delta$ is nef. Then $K_X+\Delta$ is semiample.
\end{theorem}

For terminal threefolds, this is~\cite[Theorem~1.1]{CHP16}; the gap in the proof was repaired and extended to pairs in~\cite[Theorem~1.2]{DO24},~\cite{GP24}.

\subsection{Birational geometry}

\begin{theorem}~\cite[Theorem~3.1]{Ue75}
Let $X$ be an irreducible and reduced compact complex space. Then there exists a meromorphic map 
$$ \merom{\alpha}{X}{A(X)} $$
onto a normal projective variety $A(X)$ such that we have an isomorphism $\C(X) \isom \C\big(A(X)\big)$ between the field of meromorphic functions. 

The map is unique up to bimeromorphic equivalence of meromorphic maps.
We call $\alpha$ the \emph{algebraic reduction of $X$} and $\dim A(X)$ the \emph{algebraic dimension of $X$}. We say that $X$ is \emph{Moishezon} if $\dim X = \dim A(X)$.
\end{theorem}

\begin{remark} \label{remark:alg-red} By definition, cf.~\cite[Definition~3.3]{Ue75}, the algebraic reduction of a singular complex space is given by a modification $X' \rightarrow X$ such that $X'$ is a smooth compact complex space
and $X' \rightarrow A(X)$ is holomorphic with connected fibres.
\end{remark}

For lack of reference let us prove the following well-known statement:

\begin{lemma} \label{lemma:alg-dimension}
Let $X$ and $Y$ be irreducible and reduced compact complex spaces, and let
$\merom{f}{X}{Y}$ be a generically finite meromorphic map. Then $a(X)=a(Y)$. 
\end{lemma}

\begin{proof}
By Remark~\ref{remark:alg-red}, we can assume without loss of generality that
$X$ and $Y$ are smooth, and that $f$ and the algebraic reductions of $X$ and $Y$ are holomorphic. It is clear that $a(X) \geq a(Y)$, so let us show the other inclusion: let $D \subset X$ be  the pull-back of a general very ample divisor from $A(X)$, then $\kappa(X,D)=a(X)$.
Since $f$ is generically finite the image $f(D) \subset Y$ is a prime divisor. Note that
$f(D)$ is Cartier since $Y$ is smooth, and $D \leq f^* f(D)$. Since
$\kappa(Y, f(D)) \leq a(Y)$ and $\kappa(Y, f(D))= \kappa(X, f^*f(D))$ by~\cite[Theorem~5.13]{Ue75},
we obtain
$$
a(X)=\kappa(X, D) \leq \kappa(X, f^*f(D))  \leq a(Y).
$$
\end{proof}

\begin{theorem}(Campana connectedness,~\cite[Th\'eor\`eme~3]{Cam81})\label{theorem campana connectedness}
Let $X$ be a normal compact Kähler variety that is chain connected by projective curves, i.e.\ two general points can be connected by a chain of projective curves. Then $X$ is Moishezon.
\end{theorem}

\begin{corollary} \label{corollary:no-multisection}
Let $X$ be a normal compact Kähler space admitting a fibration $\holom{\psi}{X}{T}$ onto a curve $T$. If the general $\psi$-fibre is Moishezon and $\psi$ admits a multisection, then $X$ is Moishezon.
\end{corollary}

\begin{remark} \label{remark:alg-reduction}
Let $X$ be a normal compact Kähler space. If the general fibre of the algebraic reduction $\alpha$ is Moishezon, it follows from Corollary~\ref{corollary:no-multisection} and Lemma~\ref{lemma:birational-curves} that $\alpha$ is an almost holomorphic map. If moreover $a(X)=1$, the almost holomorphic map is holomorphic. As a special case, this shows that for a normal compact Kähler surface the algebraic reduction is always given by a morphism.
\end{remark}

\begin{definition} \label{defn:moishezon-morphism}
Let $\holom{\varphi}{X}{Y}$ be a surjective morphism between irreducible and reduced compact complex spaces. We say that $\varphi$ is a \emph{Moishezon morphism} if there exists a line bundle $L$ on $X$ whose restriction to a general fibre of $\varphi$ is big.
\end{definition}

\begin{remarks} \label{remarks:moishezon-morphism}
The property of being a Moishezon morphism is stronger than having general fibres that are Moishezon. 

If $\holom{\varphi}{X}{Y}$ is Moishezon morphism, some positive multiple of $L$ induces a bimeromorphic
map $X \dashrightarrow X'$ over $Y$ such that $X' \rightarrow Y$ is a projective morphism. Thus, if $Y$ is projective, then $X'$ is projective and $X$ is Moishezon.
\end{remarks}

\begin{proposition} \label{proposition:MRC}~\cite{HP15}
Let $X$ be a normal compact Kähler threefold that is not Moishezon. If $X$ is uniruled, the MRC fibration is an almost holomorphic fibration $\merom{\psi}{X}{S}$ onto a surface that is {\em not} Moishezon.
\end{proposition}

\begin{lemma} \label{lemma:notmoishezon}
Let $X$ be a irreducible and reduced compact complex Kähler space of dimension $d$, and let $\sL$ be a pseudoeffective line bundle on $X$ that is not cohomologically trivial. If $\sL \cdot C=0$ for every curve $C$ on $X$, then $X$ is not Moishezon.
\end{lemma}

\begin{proof}
Assume that $X$ is Moishezon. Then there exists a resolution $\holom{\mu}{X'}{X}$, where $X'$ is a projective manifold. Let $C'$ be the complete intersection of $d-1$ ample divisors on $X'$. Then $\mu^* \sL \cdot C'=\sL\cdot \mu_*C'=0$ by the projection formula and by the assumptions of the lemma. Then $\mu^* \sL\equiv 0$ by~\cite[Lemma 3.7(b)]{HLL25}, hence $\sL\equiv0$, a contradiction.
\end{proof}

\begin{proposition} \label{prop:KS-pseff}
Let $S$ be a normal compact Kähler surface that is not uniruled. Then $K_S \cdot \omega_S \geq 0$ for every Kähler class $\omega_S$ and equality holds if and only if $S$ has canonical singularities and $K_S \sim_\Q 0$.
\end{proposition}

\begin{remark*}
We do not assume that $K_S$ is $\Q$-Cartier, the intersection product $K_S \cdot \omega_S$ is defined as $K_{S'} \cdot \mu^* \omega_S$ where $\holom{\mu}{S'}{S}$ is a resolution of singularities. It is standard to check that this is independent of the resolution.
\end{remark*}

\begin{proof}
Let $\holom{\mu}{S'}{S}$ be the minimal resolution, so $K_{S'}$ is $\mu$-nef. Since $S'$ is not uniruled we know by nonvanishing on surfaces the canonical class $K_{S'}$ is represented by an effective $\Q$-divisor $E$. Thus, $K_S \cdot \omega_S = K_{S'} \cdot \mu^* \omega_S \geq 0$ and equality holds if and only if the support of $E$ is contained in the $\mu$-exceptional locus. 

Assume that $K_S \cdot \omega_S=0$. Since $\mu$ is the minimal resolution the divisor $E \sim_\Q K_{S'}$ is $\mu$-nef. By the negativity lemma we obtain that $-E$ is effective, so in conclusion $E=0$. Thus, we see that $K_{S'} \sim_\Q 0$ and the minimal resolution $\mu$ is a crepant morphism. By definition this means that $S$ has canonical singularities.
\end{proof}

\begin{proposition} \label{prop:KS-trivial}
Let $S$ be a normal projective surface that does not contain any rational curves. Assume that $K_S \cdot \omega_S = 0$ for some Kähler class $\omega_S$ on $S$. Then $S$ is an abelian or bielliptic surface.
\end{proposition}

\begin{remark*}
The statement is not true for compact Kähler surfaces: let $T$ be a two-dimensional compact torus that does not contain any curve, and let $S:= T/\langle \pm 1 \rangle$ be its singular Kummer surface. Then $S$ is not projective but satisfies all the other assumptions of the proposition.
\end{remark*}

\begin{proof}
By Proposition~\ref{prop:KS-pseff}, we already know that $S$ has canonical singularities and $K_S \sim_\Q 0$. Let $\holom{\mu}{S'}{S}$ be the minimal resolution, so $S'$ is a smooth projective surface with $K_{S'} \sim_\Q 0$. 

{\em Case 1.} If $S'$ is abelian or bielliptic every curve $C \subseteq S'$ satisfies $C^2 \geq 0$, so it cannot be contracted onto a point. Thus, $\mu$ is an isomorphism.

{\em Case 2.} If $S'$ is a projective K3 or Enriques surface, it contains infinitely many distinct rational curves~\cite[Theorem~A]{CGL22}. Since $\mu$ only contracts finitely many curves we obtain that $S$ contains a rational curve which we excluded.
\end{proof}

\begin{proposition} \label{proposition:generalised-abundance-surface}
Let $X$ be a normal compact Kähler surface with canonical singularities such that $K_X \equiv 0$. Let $\sL$ be a nef line bundle on $X$ such that $\nd(\sL)=1$. Then $\sL$ is numerically equivalent to a semiample line bundle.
\end{proposition}

\begin{proof}
By Abundance on surfaces we have that $K_X\sim_\Q0$. By passing to a cyclic cover and then taking a minimal resolution, we may assume that $X$ is a smooth K3 surface or a torus. If $X$ is K3, then $\sL$ semiample follows from the Riemann--Roch estimate $h^0(X, \sL) \geq \chi(X, \sL) = \chi(\sO_X)=2$.
If $X$ is a torus, by~\cite[Ch.2, Proposition~1.7]{BL99} we have a reduction map $X \rightarrow T$ associated to $\sL$ such that $c_1(\sL)$ is the pull-back of a polarisation from $T$. 
Since the pull-back of an ample line bundle is semiample we obtain the statement.
\end{proof}

\begin{lemma} \label{lemma:dim3}
Let $X$ be a normal compact Kähler threefold with klt singularities such that $K_{X} \equiv 0$. Let $\sL $ be a nef line bundle on $X$ such that $\kappa(X,\sL)\geq0$. Then $\sL$ is semiample.
\end{lemma}

\begin{proof}
By assumption, there exists an integral divisor $D\geq0$ and a positive integer $n$ such that $\sO_X(D)\isom\sL^{\otimes n}$. Pick a rational number $0<\varepsilon\ll1$ such that the pair $(X, \varepsilon D)$ is klt. Then by Theorem~\ref{thm:3fold-abundance} we have that $K_X +\varepsilon D $ is semiample and that $K_X\sim_\Q0$. In particular, for a sufficiently divisible positive integer $m$ we have $\sO_X\big(m(K_X+\varepsilon D)\big)\isom_\Q \sL^{\otimes mn\varepsilon}$, which yields the result.
\end{proof}

\begin{lemma} \label{lemma:factor-birational}
Let $X$ be a normal compact Kähler surface with klt singularities, and let $\holom{\gamma}{X}{X'}$ be a birational morphism onto a smooth surface $X'$. Then $\gamma$ decomposes into a sequence of $K$-negative divisorial contractions.
\end{lemma}

\begin{proof}
Since $X'$ has terminal singularities we have $K_X \sim_\Q \gamma^* K_{X'} + E$ with $E$ an effective $\Q$-divisor whose support is equal to the $\gamma$-exceptional locus. By the negativity lemma the divisor $E$ is not $\gamma$-nef, so there exists a $\gamma$-exceptional curve $C \subseteq X$ such that $K_X \cdot C = E \cdot C<0$. Thus, there exists a $K$-negative contraction over $X$ and we conclude by induction on the relative Picard number. 
\end{proof}

\begin{theorem} \label{thm:plt-extension}~\cite[Corollary~1.8]{DHP13}
Let $X$ be a normal $\Q$-factorial projective variety of dimension four, and let $\Delta = E + B$ be boundary divisor, where $E$ is a prime divisor and $\lfloor B\rfloor=0$. Assume that $(X, E+B)$ is plt and $K_X+E+B$ is nef, and that there exists an effective $\Q$-divisor $N \sim_\Q K_X+E+B$ with
$$ E \subseteq \supp N \subseteq \supp (E+B). $$
Then for all sufficiently divisible positive integers $m$, the restriction morphism
$$ H^0\big(X, m (K_X+E+B)\big) \to H^0\big(E, m (K_E+B)\big) $$
is surjective. In particular, $\kappa(X,K_X+E+B)\geq 1$. 
\end{theorem}

For the last statement note that the pair $(E, B)$ is klt by adjunction. Since $K_E+B$ is nef
we know that $H^0\big(E, m (K_E+B)\big) \neq 0$ for $m \gg 0$ by Theorem~\ref{thm:3fold-abundance}.
Thus, $h^0\big(X, m (K_X+E+B)\big)\geq 2$ for $m \gg 0$.

The following statement is standard, we prove it for lack of reference:

\begin{lemma} \label{lemma:degenerate-uniruled}
Let $\holom{\pi}{\sX}{\Delta}$ be a flat proper K\"ahler morphism from an irreducible and reduced complex space onto the unit disc $\Delta$. Assume that for $t \neq 0$ the fibre $X_t$
is uniruled. Let $S \subset X_0$ be an irreducible component of the central fibre. Then 
$S$ is uniruled. 
\end{lemma}

\begin{proof}
By countability of the irreducible components of $\Chow{\sX/\Delta}$ we know that
for $t \in \Delta$ a very general point every irreducible component of the cycle space $\Chow{X_t}$
deforms with $X_t$. Thus if we choose $K_t \subset \Chow{X_t}$ an irreducible component such that
a general point corresponds to a rational curve $\holom{f}{\PP^1}{X_t}$ whose deformations dominate $X_t$, there exists an irreducible component $K \subset \Chow{\sX/\Delta}$ that dominates $\Delta$ and the fibre of $\psi\colon  K \rightarrow \Delta$ contains $K_t$ as an irreducible component. Since
the fibration $\pi$ is a proper K\"ahler morphism the map $\psi$ is proper.
Let $\holom{q}{U}{K}$ be the universal family over $K$, and let $\holom{p}{U}{\sX}$
be the evaluation map. Since the general fibre of $q$ is a rational curve, hence irreducible,
the complex space $U$ is irreducible. Moreover $p$ is proper since $\psi \circ q = \pi \circ p$
is proper. Thus $p(U) \subset \sX$ is an irreducible analytic subvariety by Remmert's theorem. 
By construction $p(U)$ contains the very general fibre $X_t$ so $p(U)=\sX$.

Choose now a general point $x \in S \subset \sX$ such that $S$ is the unique irreducible component of $X_0$ passing through $x$. Since $p$ is surjective we can choose $k \in K$ such that
the cycle $p(\fibre{q}{k})$ passes through $x$. Note that this implies that $p(\fibre{q}{k})$ is contained in $X_0$: a general member of $K$ is a rational curve contained in a $\pi$-fibre.
Since all the fibres of $q$ have the same homology class, we have
$\pi_*(p(\fibre{q}{k}))=0$. 
Since $S$ is the unique component of $X_0$ passing through $x$ this implies that there exists an irreducible component $C \subset \fibre{q}{k}$ such that $p(C) \subset S$
and $x \in p(C)$. By~\cite[II, Proposition~2.2]{Kol96} the curve $C$ is rational. Thus we have shown that there exists a rational curve through a general point of $S$.
\end{proof}

\subsection{Differential forms on singular spaces}\label{section differential forms}

\begin{definition}\label{definition torsion free differentials}
For an equidimensional reduced complex space $Y$, we define $\tau^p_Y \subset \Omega_Y^p$ as the subsheaf of sections with support of dimension strictly smaller than $\dim Y$. Somewhat imprecisely, we refer to $\tau^p_Y$ as the \emph{torsion subsheaf} of $\Omega_Y^p$. We define the sheaf of \emph{torsion free $p$-forms} as
\[
\tOm_Y^p:=\Omega^p_Y/\tau^p_Y.
\]
For $p=1$, we also refer to it as the sheaf of \emph{torsion free Kähler differentials}.
\end{definition}

A useful property of $\tOm_Y^p$ is that if $\pi\colon Z\to Y$ is a resolution of singularities, then the canonical map $\tOm_Y^p \to \pi_*\Omega_Z^p$ (induced by $\pi^* \Omega_Y \rightarrow \Omega_Z$) is injective.

\subsection{Hyperkähler preliminaries}\label{section hyperkaehler}

Recall from~\cite{Bea83} that on a hyperkähler manifold $X$ of dimension $2n$ there exists a unique non-degenerate primitive integral symmetric bilinear form 
\[
q_X\colon H^2(X,\R)\times H^2(X,\R)\to \R
\]
of signature $(1,h^{1,1}(X)-1)$ on $H^{1,1}(X,\R)$. This form satisfies the \emph{Fujiki relation}: there exists a positive rational number $c_X$ depending only on the deformation class of $X$ such that for each $\alpha \in H^2(X,\C)$ we have
\[
q_X(\alpha,\alpha)^n = c_X \int_X\alpha^{2n}.
\]
The form $q_X$ is called the \emph{Beauville--Bogomolov--Fujiki form} or \emph{BBF form} for short.

\begin{definition}
Let $X$ be a hyperkähler manifold. We denote by $\scrC_X \subset H^{1,1}(X,\R)$ the \emph{positive cone}, i.e. the connected component of the cone of positive vectors with respect to $q_X$ that contains the Kähler cone $\scrK_X$.
\end{definition}

\begin{lemma} \label{lemma:allzero}
Let $X$ be a hyperkähler manifold of dimension $2n$, and let $L \not\cong \sO_X$ be a parabolic nef line bundle on $X$. Then for a very general deformation $X_t$ in the Hodge locus of $c_1(L) \in H^2(X, \Z)$, the rational Picard group $\pic(X_t)\tensor \Q$ is generated by the line bundle $L_t$ corresponding to $L$. Moreover, for any such deformation, up to shrinking the representative of the Kuranishi space, $L_t$ is a nef line bundle satisfying $ c_1(L_t) \cdot C = 0 $ for every curve $C \subseteq X_t$. 
\end{lemma}

\begin{proof}
It follows from the local Torelli theorem that the Hodge locus of $c_1(L)$ is a hypersurface in the Kuranishi space of $X$ and that for a very general point in this Hodge locus, the corresponding fibre of the universal family has Picard rank equal to one. Since $c_1(L) \in H^2(X_t, \Z) \cap H^{1,1}(X_t)$, it defines a line bundle $L_t$ on~$X_t$ with $c_1(L_t)=c_1(L)\neq 0$, so $L_t$ is a $\Q$-generator of the Picard group. Note that $b_2(X) \geq 4$ so that the Hodge locus has positive dimension. 

Let $X_t$ be a hyperkähler manifold corresponding to such a very general point. We show now that $c_1(L_t)$ pairs trivially with every curve in $X_t$. For this, let $C \subseteq X_t$ be a curve and consider the induced linear form on $\pic(X)\tensor \Q$ given by taking the intersection product with $C$. By integrality and nondegeneracy of the BBF form, there is a unique class $\alpha_C \in H^{2}(X,\Q)$ satisfying 
$$ q_{X_t}(\alpha_C,\beta) = C\cdot\beta \quad \text{for all }\beta \in H^2(X_t,\Q). $$
Since $C$ is an analytic cycle, we must have $\alpha_C \in H^{1,1}(X_t,\Q)=\pic(X_t)\tensor \Q$. By assumption on $X_t$, the class $\alpha_C$ must therefore be proportional to $c_1(L_t)$, say $\alpha_C=\lambda c_1(L_t)$ for some $\lambda \in \Q$. Thus, we have
$$ c_1(L_t) \cdot C = q_X\big(c_1(L_t),\alpha_C\big)= \lambda q_{X_t}(L_t) = \lambda q_X(L) = 0. $$
The last equality comes from the assumption that $L$ is parabolic.

Since $L$ is nef, we know that for $X_t$ sufficiently close to $X$, there is a Kähler class~$\omega_t$ such that $c_1(L_t) \cdot \omega_t^{2n-1} > 0$. Thus, $L_t$ is in the closure of the positive cone of $X_t$ since \mbox{$q_{X_t}(L_t)=0$} and $c_1(L_t) \cdot \omega_t^{2n-1} > 0$. As $c_1(L_t)$ has non-negative intersection with every rational curve in $X_t$, we obtain that $L_t$ is nef by~\cite[Proposition~3.2]{Huy03}.
\end{proof}

The proof above shows:

\begin{corollary} \label{cor:allzero}
Let $X$ be a hyperkähler manifold of dimension $2n$ such that $\pic(X) \otimes \Q \isom \Q \sL$, where $\sL$ is a non-trivial nef parabolic line bundle. Then $\nd(\sL)=n$, and $ \sL $ has degree zero on all curves in $X$.
\end{corollary}

%----------------------------------------------------
\subsection{Linearisations}\label{section linearisations}
%----------------------------------------------------

We briefly recall some aspects of the theory of $G$-linearisa\-tions; see~\cite[Section~1.3]{MFK94} for a detailed treatment. Let $X$ be a complex space and $L$ a line bundle on $X$. We denote by $\Aut(X,L) \subset \Aut(X)$ the subgroup of automorphisms $g$ of $X$ such that $g^*L\isom L$. 

\begin{definition}\label{definition g-linearisation}
Let $G$ be a complex Lie group acting on a complex space $X$ via a morphism
\[
\sigma \colon G \times X \to X, \qquad (g,x) \mapsto g \cdot x,
\]
and let $\pr_2 \colon G \times X \to X$ be the projection onto the second factor. Let $L$ be a line bundle on $X$. A \emph{$G$-linearisation} of $L$ is an isomorphism of line bundles on $G \times X$
\begin{equation}
\label{eq linearization isomorphism}
\Phi \colon \sigma^*L \xrightarrow{\;\sim\;} \pr_2^*L
\end{equation}
satisfying the following cocycle condition. On $G \times G \times X$, we have the equality
\begin{equation}
\label{eq linearization cocycle}
(m \times \id_{X})^*\Phi = \pr_{23}^*\Phi \circ (\id_{G} \times \sigma)^*\Phi,
\end{equation}
of isomorphisms of line bundles, where $m \colon G \times G \to G$ is the group structure and $\pr_{23} \colon G \times G \times X \to G \times X$ projects onto the second and third factors.
\end{definition}

In the situation of Definition~\ref{definition g-linearisation}, set 
\begin{equation}\label{eq linearization pointwise}
 \varphi_g := \Phi|_{\{g\} \times X} \colon g^*L \xrightarrow{\sim} L
\end{equation}
for $g \in G$. Then the $\varphi_g$ satisfy the cocycle condition:  
\begin{equation}\label{eq cocycle condition on points}
\varphi_1 = \id_L
  \quad\text{and}\quad
\varphi_{gh} = \varphi_g \circ g^*\varphi_h \qquad \textrm{ for all } g,h \in G.
\end{equation}
The rule $(g,s) \;\mapsto\;
g \cdot s := \varphi_g\bigl(g^*s\bigr)$. defines a linear representation 
\[
G \to \GL\left(H^{0}(X,L)\right).
\]
\begin{remark}\label{remark linearization discrete group}
If $G$ is a \emph{discrete} group, then the isomorphisms $\varphi_g \colon g^*L \to L$ can be specified individually. The global condition \eqref{eq linearization isomorphism} is equivalent to giving a family~$(\varphi_g)$ as above satisfying the cocycle condition \eqref{eq cocycle condition on points}. 
\end{remark}

\begin{example}
If $X$ is smooth, equipped with a $G$-action, there is a canonical $G$-linearisation on the canonical sheaf~$\omega_X$. The corresponding automorphisms $\varphi_g\colon g^*\omega_X \to \omega_X$ from \eqref{eq linearization pointwise} are given by the determinant of the dual to the differential $T_X \to g^*T_X$. 

Similarly, for a pair $(X,\Delta)$, there is a canonical $G$-linearisation on $\omega_X(\Delta)$ whenever $G$ is a group of automorphisms leaving $\Delta$ invariant. On a general pair $(X,\Delta)$, one obtains a $G$-linearisation on the reflexive sheaf $\omega_X(\Delta)$ by push forward from $X^\reg$. The same holds if automorphisms are replaced by pseudo-automorphisms.
\end{example}

\begin{lemma}\label{lemma tensor power linearization}
Let $X$ be a complex space, $L$ a line bundle on $X$, and $G \subset \Aut(X,L)$ a finite group. If $L$ admits a $G$-linearisation, then there is $n \in \Z_{>0}$ such that $L^{\otimes n}$ descends to the quotient $X/G$.
\end{lemma}
\begin{proof}
We have to show that there is $n\in \Z_{>0}$ such that for each $x \in X$ the action of the stabilisers $G_x$ on the fibre $(L^{\tensor n})_x \isom (L_x)^{\otimes n}$ is trivial. Since $G$ is finite, the action of $G_x$ on $L_x$ is by roots of unity, hence as soon as $|G|$ divides $n$, the action on $L_x^{\otimes n}$ is trivial. Then the bundle $L^{\otimes n}$ descends to the quotient by standard descent theory; adapt~\cite[Proposition~35.5.2]{Stacks} or~\cite[Theorem~10.3]{Alp13} to the analytic setup.
\end{proof}

We define the group
\begin{equation}\label{eq linearized automorphism group}
\LinAut(X,L):= \{ (\sigma,\alpha) \mid \sigma \in \Aut(X, L),\ \alpha \colon \sigma^*L \xrightarrow{\sim} L \},
\end{equation}
of linearised automorphisms with group law
\[
(\sigma,\alpha)\cdot(\tau,\beta):= (\sigma\circ\tau,\ \alpha \circ \sigma^*\beta).
\]
The projection to the first component gives a group homomorphism
\[
\pr_1 \colon \LinAut(X,L) \to \Aut(X),\quad (\sigma,\alpha) \mapsto \sigma.
\]
\begin{lemma}\label{lemma linearization as a group homomorphism}
Let $G$ be a discrete group acting on $X$, and let $L$ be a line bundle on $X$. Then the datum of a $G$-linearisation of $L$ is equivalent to the datum of a group homomorphism
\[
\rho \colon G \to \LinAut(X,L)
\]
such that $\pr_1 \circ \rho \colon G \to \Aut(X)$ is the given action of $G$ on $X$.
\end{lemma}
\begin{proof}
As $G$ is discrete, a $G$-linearisation of $L$ is given by isomorphisms
\[
\varphi_g \colon g^*L \xrightarrow{\sim} L
\]
for all $g \in G$ satisfying \eqref{eq cocycle condition on points}, cf.~Remark~\ref{remark linearization discrete group}. Thus, the map
\[
  \rho \colon G \to \LinAut(X,L),\quad
  \rho(g) := (g,\varphi_g),
\]
where we write $g$ also for the corresponding automorphism of $X$, is a group homomorphism if and only if the cocycle condition for the $\varphi_g$ are satisfied:
\[
\rho(gh)= (gh,\varphi_{gh})= (gh,\varphi_g \circ g^*\varphi_h)= (g,\varphi_g)\cdot(h,\varphi_h)= \rho(g)\rho(h).
\]
\end{proof}

\begin{corollary}\label{corollary erzeugnis mit linearisierung}
Let $L$ be a line bundle on a normal variety $X$. Given linearisations of $L$ for groups $G_1, G_2\subset \Aut(X,L)$, there is a group $G\subset \LinAut(X,L)$ and a $G$-linearisation of $L$ with the following properties.
\begin{enumerate}
	\item The groups $G_1$ and $G_2$ are in a canonical way subgroups of $G$ and together they generate $G$.
	\item The projection $\pr_1\colon  G \to \Aut(X,L)$ restricts to isomorphsims $G_i \to G_i \subset \Aut(X,L)$ for $i=1,2$.
	\item When restricted to the $G_i$, the $G$-linearisation induces the given $G_i$-linearisations on $L$ for $i=1,2$.
\end{enumerate}
\end{corollary}
\begin{proof}
By Lemma~\ref{lemma linearization as a group homomorphism}, the $G_i$-linearisations on $L$ induce morphisms $\phi_i\colon G_i \to \LinAut(X,L)$ for $i=1,2$ such that $\pr_1\circ \rho_i$ is just the inclusion in $\Aut(X,L)$. Hence, $\rho_1 ,\rho_2$ must be monomorphisms. We define
\[
G := \langle \rho_1(G_1), \rho_2(G_2) \rangle \subset \LinAut(X,L).
\]
Again by Lemma~\ref{lemma linearization as a group homomorphism}, the line bundle $L$ carries a canonical $G$-linearisation. The claimed properties hold by construction. 
\end{proof}

\subsection{Some descent lemmas}

Let $\holom{n}{\widetilde X}{X}$ be a generically finite surjective morphism between normal compact Kähler spaces, and let $P$ be a line bundle on $X$. Then we have
$$ \kappa(\widetilde X, n^* P) = \kappa(X, P), $$
i.e.\ the Iitaka dimension is invariant. This is no longer true when $X$ is not normal: let $X \subseteq \PP^2$ be a nodal cubic, and let $\holom{\nu}{X^\nu}{X}$ be the normalisation. Let $P$ be a degree zero line bundle on $X$ that is not torsion in $\pic{X} \isom \C^*$. Then $\kappa(X, P)=-\infty$, but $\kappa(\widetilde X, \nu^* P)=0$. We will prove a series of technical descent lemmas whose interest will become clear later on. Let us start by recalling:

\begin{lemma} \label{lemma:zariski}
Let $\holom{\varphi}{X}{T}$ be a fibration from a normal compact Kähler space $X$ onto a smooth projective curve $T$. Let $\sL$ be a line bundle on $X$ such that $c_1(\sL)^2=0$ and such that the restriction of $\sL$ to a general fibre of $\varphi$ is a torsion line bundle. Then $\sL \isom_\Q \varphi^* A$, where $A$ is a $\Q$-line bundle on $T$ which is unique up to $\Q$-linear equivalence.  \end{lemma}

\begin{proof}
The uniqueness statement follows since the pullback map of line bundles along a fibration is injective by the projection formula. In the remainder of the proof we show the existence of such a $\Q$-divisor $A$. By replacing $X$ by its resolution we may assume that $X$ is smooth. For each $t\in T$ we denote by $X_t$ the fibre of $\varphi$ over $t$. We will use below that for all $m\in\N$ the sheaves $\varphi_*\sL^{\otimes m}$ are either zero-sheaves or are locally free: this follows by~\cite[paragraph after Corollary 5.15]{Kob87}, since $T$ is a curve and since these sheaves are torsion-free, as the map $\varphi$ is surjective.

Let $U\subseteq T$ be an open subset such that $\sL|_{X_t}$ is a torsion line bundle for all $t\in U$. By~\cite[Theorem on p.\ 210]{GR84}, for each $m\in\N$ there exists a non-empty open subset $U_m$ of the non-critical locus in $T$ of $\varphi$ such that the function $h^0(X_t,\sL^{\otimes m}|_{X_t})$ is constant on $U_m$. Pick a point $t_0\in U\cap \bigcap_{m\in\N}U_m$, and pick a positive integer $m_0$ such that $(\sL|_{X_t})^{\otimes m_0}\isom \sO_{X_{t_0}}$, which exists since $t_0\in U$. Then the sheaf $\varphi_*\sL^{\otimes m_0}$ is locally free of rank one on $U_{m_0}$ by~\cite[Theorem on p.\ 211]{GR84}, hence everywhere on $T$ by the previous paragraph. By replacing $\sL$ by $\sL^{\otimes m_0}$, we may thus assume that $\varphi_*\sL$ is a line bundle on $T$.

Using the evaluation morphism $\varphi^* \varphi_* \sL \rightarrow \sL$ we obtain that
$$ \sL \isom \varphi^* \varphi_* \sL \otimes\sO_X(E), $$
where $E$ is an effective divisor that has support in the fibres of $\varphi$. Since \mbox{$c_1(\sL)^2=0$}, we know that $E^2=0$. Thus, the proof of Zariski's lemma \cite[III, Lemma~8.2]{BHPV04} shows that $E \sim_\Q \varphi^* \Delta$ for an effective $\Q$-divisor $\Delta$ on $T$. In that proof, the intersection form on divisors is replaced by the form
$$ Q(\alpha,\beta):=\int_X \alpha.\beta.\omega^{\dim X -2}, $$
where $\omega$ is a Kähler class on $X$.
\end{proof}

\begin{proposition} \label{prop:semiample-nonnormal}
Let $X$ be a reduced and irreducible compact Kähler space. Let $F \subseteq X$ be an effective, reduced Weil divisor which is Cartier, and whose every irreducible component meets the smooth locus of~$X$. Let $ \holom{\pi}{\widetilde X}{X} $ be a birational morphism that is an isomorphism in a neighbourhood of $F$, and denote by $\widetilde F:=\pi^*F \subseteq \widetilde X$. Assume that there exists a fibration $\holom{\psi}{\widetilde X}{\Gamma}$ onto a projective curve $\Gamma$.
$$\xymatrix{
\widetilde{X} \ar[r]^\pi \ar[d]_\psi & X \\
\Gamma & }$$
If there exists a smooth point $t\in\Gamma$ such that $\widetilde F = \psi^* t$, then $F$ is semiample with $\kappa(X,F)=1$.
\end{proposition}

\begin{proof}
The point $t$ is an ample Cartier divisor on the curve $\Gamma$, hence the divisor $\widetilde F = \psi^* t$ is semiample with
\begin{equation}\label{eq:009b}
\kappa(\widetilde X,\widetilde F)=1.
\end{equation}
We have the exact sequence
\begin{equation}\label{eq:009a}
0 \rightarrow \sO_{X} \rightarrow \pi_* \sO_{\widetilde X} \rightarrow \mathcal Q \rightarrow 0,
\end{equation}
where $\mathcal Q$ is a coherent sheaf whose support is disjoint from $F$. In particular, for each positive integer $d$ we have $\mathcal Q \otimes \sO_X(dF) \isom \mathcal Q$. Furthermore, since $\widetilde F=\pi^*F$, by the projection formula we have
$$ \pi_* \sO_{\widetilde X} \otimes \sO_X(dF) = \pi_* \sO_{\widetilde X}(d \widetilde F). $$
Therefore, twisting the sequence \eqref{eq:009a} with $\sO_X(dF)$ and passing to the associated cohomology sequence, for each positive integer $d$ we obtain the exact sequence
$$ 0 \rightarrow H^0\big(X,\sO_X(dF)\big) \rightarrow H^0\big(\widetilde X,\sO_{\widetilde X} (d\widetilde F)\big) \rightarrow H^0(X,\mathcal Q). $$
Consequently, for each positive integer $d$ we have
$$h^0\big(\widetilde X,\sO_{\widetilde X} (d\widetilde F)\big) \geq h^0\big(X,\sO_X(dF)\big) \geq h^0\big(\widetilde X,\sO_{\widetilde X} (d\widetilde F)\big) - h^0(X,\mathcal Q ).$$
The first inequality and \eqref{eq:009b} immediately give $\kappa(X,F)\leq 1$. Moreover, \eqref{eq:009b} implies that $\lim\limits_{d\to\infty}h^0\big(\widetilde X,\sO_{\widetilde X} (d\widetilde F)\big)={+}\infty$, hence the second inequality gives $\lim\limits_{d\to\infty}h^0\big(X,\sO_X (dF)\big)={+}\infty$, and consequently, $\kappa(X,F)\geq 1$. Thus, $\kappa(X,F)= 1$, which shows the second statement of the proposition.

In the remainder of the proof we show the first statement of the proposition. By the paragraph above, there exists a minimal positive integer $d$ such that $h^0\big(X, \sO_X(dF)\big) \geq 2$. We will show that the line bundle $\sO_X(dF)$ is globally generated, which will finish the proof.

By the choice of $d$, the natural injection
$$ H^0\big(X, \sO_X((d-1) F)\big) \rightarrow H^0\big(X, \sO_X(dF)\big), $$
given by the multiplication by a  section whose vanishing locus is $F$, is strict. Thus, $F$ is not contained in the base locus of the linear system $|dF|$. Hence, there exists a divisor $G \in |dF|$ whose support does not contain~$F$. Let $\sigma\in H^0\big(X, \sO_X(dF)\big)$ be a  section whose vanishing locus is $G$.

Since $\widetilde F = \psi^* t$ and $\psi$ is a fibration, we have
$$ H^0\big(X, \sO_X(dF)\big) \subseteq H^0\big(\widetilde X, \sO_{\widetilde X}(d\widetilde F)\big) = \psi^*\big(H^0\big(\Gamma, \sO_\Gamma(d t)\big)\big). $$
Therefore, there exists $s\in H^0\big(\Gamma, \sO_\Gamma(d t)\big)$ such that $\sigma=\psi^*(s)$.
In particular
if $s \in H^0(X, \sO_X(dF))$ and 
$\widetilde F \isom F = \cup F_i$ is the decomposition of $F$ into its irreducible components,
the section $s$ vanishes with the same multiplicity on every irreducible component $F_i$.

Then $G$ has multiplicity zero on one component $F_i$, by the above this implies 
that $G$ has multiplicity zero on every component $F_i$.
Therefore,
the divisors $dF$ and $G$ are elements of $|\sO_{X}(dF)|$ that have no common component. 
Since $c_1(\sO_X(F))^2=0$ it follows that
$$
\supp(dF) \cap \supp(G) = \emptyset.
$$
Thus, $\sO_X(dF)$ is globally generated.
\end{proof}

\begin{corollary} \label{corollary:alg-dim-one}
Let $X$ be an irreducible and reduced compact Kähler surface with $a(X)=1$. Then the algebraic reduction of $X$ can be represented by a holomorphic fibration $\holom{\alpha_X}{X}{\Gamma}$ onto a (not necessarily normal) projective curve, and a general fibre of $\alpha_X$ is a smooth elliptic curve. Moreover, any analytic subset $Z \subsetneq X$ is contained in finitely many fibres of $\alpha_X$.
\end{corollary}

\begin{proof}
Let $\holom{\nu}{X^\nu}{X}$ be the normalisation of $X$. We know by Remark~\ref{remark:alg-reduction} that the algebraic reduction $\holom{\alpha_{X^\nu}}{X^\nu}{\widetilde \Gamma}$ is an elliptic fibration. The set $\fibre{\nu_X}{X_{\sing}} \subseteq X$ is analytic, hence it is mapped onto a finite set by Corollary~\ref{corollary:no-multisection}. Thus, $\nu$ is an isomorphism in a neighbourhood of a general fibre $\widetilde F$ of $\alpha_{X^\nu}$. Therefore, $F:= \nu_X(\widetilde F)$ is a Cartier divisor that satisfies the conditions of Proposition~\ref{prop:semiample-nonnormal}, and some positive multiple of $F$ defines the fibration $\alpha_X$. Since $X_{\sing}$ does not dominate the base $\Gamma$, it is clear that the general fibre of $\alpha_X$ is smooth, and in fact equal to $F$.

For the last statement of the corollary, note that any multisection of $\alpha_X$ is dominated by a multisection of $\alpha_{X^\nu}$. But such a multisection does not exist by Remark~\ref{remark:alg-reduction}.
\end{proof}

\begin{lemma} \label{lemma:numerically-trivial-twist}
Let $Z$ be a compact Kähler space, let $\sG$ be a be a coherent sheaf on $Z$, and let $P$ be a line bundle on $Z$ that is cohomologically trivial. Then there exists $K \in \N$ such that
$$ h^0(Z, \sG \otimes P^{\otimes d}) \leq K \quad \text{for all }d \in \N. $$
\end{lemma}

\begin{proof}
Let $\pi\colon \PP(\sG) \rightarrow Z$ be the projectivisation of $\sG$ as in~\cite[Chapter~II]{AT82}, and let $\sL:=\sO_{\PP(\sG)}(1)$ be the tautological line bundle on $\PP(\sG)$. Since $\pi_*\sO_{\PP(\sG)}=\sO_Z$, the projection formula yields
$$ \pi_* (\sL \otimes \pi^* P^{\otimes d}) \isom \pi_* \sL \otimes P^{\otimes d} \isom \sG  \otimes P^{\otimes d}, $$
so it is enough to uniformly bound the size of the global sections of $\sL \otimes \pi^* P^{\otimes d}$.

Let $\omega$ be a Kähler form on $\PP(\sG)$. For any $d\in\N$, and consider any section $D_d \in |\sL \otimes \pi^* P^{\otimes d}|$. Then, if $n:=\dim \PP(\sG)$, we have
$$ D_d \cdot \omega^{n-1} = c_1(\sL \otimes \pi^* P^{\otimes d}) \cdot \omega^{n-1} = c_1(\sL) \cdot \omega^{n-1}, $$
hence $D_d \cdot \omega^{n-1}$ does not depend on $d$ of $D_d$. Thus, the divisors $D_d$ belong to the family of cycles of degree at most~$c_1(\sL) \cdot \omega^{n-1}$. By 
an application of
Bishop's theorem~\cite[Ch.VIII, Theorem~2.13]{CP94} this family is bounded; in particular, it has finitely many irreducible components. We set $K$ to be the maximal dimension of such an irreducible component.
\end{proof}

\begin{proposition} \label{prop:semiample+fibration}
Let $X$ be an irreducible and reduced compact Kähler space admitting a fibration
$$ \holom{\phi}{X}{\Gamma} $$
onto a projective curve $\Gamma$ such that $\phi_* \sO_X \isom \sO_\Gamma$. Let 
$$ \holom{n}{\widetilde X}{X} $$
be a birational morphism that is an isomorphism over $\fibre{\phi}{\Gamma_\circ}$ where $\Gamma_\circ \subseteq \Gamma$ is Zariski open. Let $P$ be a line bundle on $X$ that is cohomologically trivial over $\Gamma$. If $\kappa(\widetilde X, \pi^* P)>0$, then $P$ is semiample and $\kappa(X, P)=1$.
\end{proposition}

\begin{proof}
The birational map $n$ defines an exact sequence
$$ 0 \rightarrow \sO_{X} \rightarrow n_* \sO_{\widetilde X} \rightarrow Q \rightarrow 0 $$
where $Q$ is a coherent sheaf that has support on $\fibre{\phi}{\Gamma \setminus \Gamma_\circ}$. Twisting with $P^{\otimes d}$ we obtain an exact sequence
$$ 0 \rightarrow P^{\otimes d} \rightarrow n_* \sO_{\widetilde X} \otimes P^{\otimes d} \rightarrow Q \otimes P^{\otimes d} \rightarrow 0. $$
Since
$$ H^0(X,  n_* \sO_{\widetilde X} \otimes P^{\otimes d}) \isom H^0(X,  n_* (\sO_{\widetilde X} \otimes n^* P^{\otimes d})) \isom  H^0(\widetilde X,  \sO_{\widetilde X} \otimes n^* P^{\otimes d}) $$
is unbounded and $h^0(X, Q \otimes P^{\otimes d})$ is bounded by Lemma~\ref{lemma:numerically-trivial-twist} we can choose $d \in \N$ such that $h^0(X, P^{\otimes d}) \geq 1$.

Let $D \in |P^{\otimes d}|$ be an effective divisor. We claim that, up to taking a multiple, one has $mD \isom \varphi^* A$ with $A \rightarrow \Gamma$ an effective Cartier divisor. Since $\Gamma$ is a curve the Cartier divisor $A$ is ample and we obtain the statement.

{\em Proof of the claim.}
Since $P$ is $\phi$-numerically trivial, the divisor $D$ does not intersect the general $\phi$-fibres. Thus, the support of $D$ is a disjoint union of effective Cartier divisors contained in some $\phi$-fibre. We will assume without loss of generality that $\phi(D)$ is a point $p \in \Gamma$. Choose a local chart of $\Gamma$ around the point $p$ and in this chart a hyperplane section $A_p$ that intersects $\Gamma$ only in $p$. Then $\phi^* A_p$ is an effective Cartier divisor with set-theoretical support $\fibre{\phi}{p}$. We choose $l \in \N^*, m \in \N_0$ such that $l D - m \phi^* A_p$ is effective, but there is an irreducible component $G \subseteq \fibre{\phi}{p}$ that is not contained in the support of $l D - m \phi^* A_p$. Since $l D - m \phi^* A_p$ is numerically trivial on $\fibre{\phi}{p}$ and $\phi$ has connected fibres, it is now straight forward to show that $l D - m \phi^* A_p=0$. 
\end{proof}

\subsection{Groups acting on surface of algebraic dimension one}

\begin{lemma} \label{lemma-reduce-Bim}
Let $X$ be a smooth compact Kähler surface with $a(X)=1$ and let $ \holom{\alpha}{X}{C} $ be the algebraic reduction onto a projective curve $C$; see Remark~\ref{remark:alg-reduction}. Then $\alpha$ induces a group morphism
\begin{equation}\label{eq group action algebraic reduction}
\holom{\gamma}{\Bim(X)}{\Aut(C)},
\end{equation}
and the image of $\gamma$ is finite unless the minimal model of $X$ is a torus.
\end{lemma}

\begin{proof}
The surface $X$ is not projective, so $h^0(X, K_X)=h^2(X,\sO_X) \neq 0$ by Kodaira's criterion. Thus, $X$ is not uniruled and its minimal model $X_{\min}$ has nef canonical divisor. Since
$$ \Bim(X)\isom \Bim(X_{\min}) = \Aut(X_{min}),$$
and the bimeromorphic map $X \dashrightarrow X_{min}$ preserves the algebraic reduction, we can assume without loss of generality that $X$ is minimal and the bimeromorphic automorphisms are regular.

The fibration $\alpha$ has no multisection by Corollary~\ref{corollary:no-multisection}, so every $\varphi \in \Aut(X)$ preserves the fibration. Thus, the rigidity lemma yields the existence of the group morphism $\gamma$. Note that if $\Delta \subseteq C$ are the points such that $\fibre{\alpha}{t}$ is not smooth, we have $\varphi(\Delta)=\Delta$ for every $\varphi \in \Aut(X)$. Therefore, the image of $\gamma$ is contained in 
$$ \Aut(C, \Delta) = \big\{ \varphi \in \Aut C \mid \varphi(\Delta) = \Delta \big\}. $$
Let us now show the second statement. The statement is trivial if $\mbox{Aut}(C, \Delta)$ is finite:
this holds if  
$g(C) \geq 2$ or $g(C)=1$ and $\Delta \neq \emptyset$.
If $C \isom \PP^1$, we know, for instance by~\cite[Proposition~A.1]{CHL19}, that $\Delta$ has at least three elements. So again $\mbox{Aut}(C, \Delta)$ is finite.

Thus we are left with the case where $g(C)=1$ and $\Delta=\emptyset$. Then $c_2(X)=e(X)=0$ by~\cite[Proposition~III.11.4]{BHPV04} and $K_X \equiv 0$ by the canonical bundle formula. Since $X$ is not projective, the classification of surfaces shows that $X$ is a torus.
\end{proof}

Let $L$ be a line bundle on a compact complex manifold $X$. For a bimeromorphic automorphism $g$ of $X$, the line bundle $g^*L$ is defined by pulling back the line bundle to the domain of definition of $g$ and then pushing forward to $X$ and taking the bidual. We denote
$$ \Bim (X,L) := \{ g\in \Bim(X) \mid g^*L\isom L \} $$
and
$$ \Aut(X,L):=\Bim(X,L)\cap \Aut X. $$

\begin{lemma} \label{lemma:factor-action}
Let $X$ be a smooth compact Kähler surface with $a(X)=1$, and let $ \holom{\alpha}{X}{C} $ be the algebraic reduction onto a projective curve $C$. Let $L$ be a nef line bundle on $X$ whose restriction to a general fibre of $\alpha$ is torsion, and let $\gamma$ be as in Lemma~\ref{lemma-reduce-Bim}. Then some power of $L$ descends to a line bundle $L_C$ on $C$, and the image of $\Bim(X,L)$ under $\gamma$ is contained in $\Aut(C,L_C)$. If $L$ is not numerically trivial, this group is finite.
\end{lemma}

\begin{proof}
The line bundle $L$ has degree zero on the fibres of $\alpha$. As $L$ is nef, we see that $L$ has degree zero on every irreducible curve in $X$. Therefore, $L$ descends as a line bundle to the minimal model of $X$ and we can assume without loss of generality that $X$ is minimal and $\Bim(X,L) = \Aut(X,L)=:G$.

Since $L$ restricts to a torsion line bundle on a general fibre of $\alpha$, up to replacing~$L$ by a multiple, Lemma~\ref{lemma:zariski} implies that $L = \alpha^* L_C$ for a unique line bundle $L_C$ on the curve $C$. Thus, for every $g\in G$ we have $g^* L= \alpha^* \gamma(g)^* L_C$. But $g^*L\isom L = \alpha^*L_C$. Since $\alpha$ has connected fibres, pullback along it is injective on Picard groups, so $\gamma(g)^* L_C\cong L_C$, and hence we obtain $\gamma(G)\subset \Aut(C,L_C)$. 

For the finiteness statement, by Lemma~\ref{lemma-reduce-Bim} we only have to discuss the case where $X$ is a torus. Since $L \isom \alpha^* L_C$ and $L$ is not numerically trivial, we deduce that $L_C\in\pic^d(C)$ for some $d \neq 0$. Now recall that $\Aut(C)$ acts transitively on $\pic^d(C)$ and that $\Aut(C) \isom C \rtimes \Aut(C,0)$ so that the stabiliser of any $L_C \in \pic^d(C)$ is finite. This shows that $\gamma(G)$ is finite.
\end{proof}

\begin{remark}\label{rem:another proof}
By a result of Fujiki~\cite[Theorem~5.1, Case 2]{Fuj88}, the automorphism group $\Aut_0(X)$ of a two-dimensional torus with algebraic dimension one is finite. This gives an alternative, much less elementary, proof of Lemma~\ref{lemma:factor-action}.
\end{remark}

\begin{corollary} \label{cor:global-generation}
Let $X$ be a smooth compact Kähler surface with $a(X)=1$, and let $ \holom{\alpha}{X}{C} $ be the algebraic reduction onto a projective curve $C$. Let $L$ be a nef line bundle on $X$ such that $H^0(X, L) \neq 0$. Let $G$ be a complex Lie group with a morphism $G \to \Bim(X,L)$ and suppose that $L$ admits a $G$-linearisation. Then some multiple of~$L$ is globally generated by $G$-invariant sections.
\end{corollary}

\begin{proof}
We may assume that $L\not\cong \sO_X$, for this case is trivial. By Lemma~\ref{lemma:factor-action}, up to replacing $L$ by a multiple, there is a unique line bundle $L_C$ on $C$ such that $L \isom \alpha^* L_C$. Since $L$ has a $G$-linearisation and $\alpha$ is a fibration, the $G$-linearisation on $L$ descends along $\alpha$ to a $\gamma(G)$-linearisation on $L_C$ (cf. Subsection~\ref{section linearisations}). Note that this implies that the action of $G$ on
$$ H^0(X,L) \isom H^0(C,L_C) $$
factors through $\gamma(G)$. Note that an effective line bundle which is not trivial, is not numerically trivial. So again by Lemma~\ref{lemma:factor-action}, the group $\gamma(G)$ is finite. Therefore, Lemma~\ref{lemma tensor power linearization} implies that some positive power of $L_C$ descends to a line bundle $L_{C'}$ on the quotient $C':=C/\gamma(G)$. Since $L_C$ is nef, nontrivial, and has a nonvanishing section, it must be ample on $C$ and $L_{C'}$ must be ample on $C'$. Hence, some multiple of $L_{C'}$ is very ample on $C'$ and in particular globally generated. Sections of $L_{C'}$ pull back to invariant sections of multiples of $L_C$ respectively $L$ and the claim follows.
\end{proof}

%----------------------------------------------------
\subsection{Projective bundles}\label{section projective bundles}
%----------------------------------------------------

We collect some standard material about projective bundles, mainly over abelian varieties. 

Let $V$ be a vector bundle on a smooth projective variety $A$. It is well known that the canonical sheaf of a projective bundle $P=\PP_A(V)$ is given by
\begin{equation}\label{eq canonical sheaf projective bundle}
\omega_P \isom \pi^*\left(\omega_A \tensor \det(V)\right) \tensor \sO_V(-r)
\end{equation}
where $r=\rk(V)$. Also for $r=2$, the following result is standard and we omit its proof.

\begin{lemma}\label{lemma class of section of p1 bundle in pic}
Let $V$ be a rank $2$ vector bundle on a smooth variety $A$,
set $P:=\PP(V)$ and denote by $\pi\colon P\to A$ the projection. Consider a line bundle $L$ on $A$ and a surjection $V\twoheadrightarrow L$ with kernel $M$. Let $s\colon A\to P$ be the corresponding
section and $D$ its image. Then 
\[
\sO_P(D) \isom \sO_V(1) \otimes \pi^*M^{\vee}.
\]
\end{lemma}

Again for a vector bundle $V$ on $A$ of rank two, we put $P:=\PP(V)$ and denote $\pi\colon P \to A$ the projection. Suppose that $\pi$ has two disjoint sections $\Sigma_0$ and $\Sigma_\infty$. Then $V$ splits as a sum of line bundles (see~\cite[Exercise~V.2.2]{Har77}), and we can tensor $V$ by a line bundle so that we may assume
\begin{equation}\label{eq p1 bundle with two sections}
P=\PP(V) \quad \textrm{ for } \quad V = \sO_A \oplus L.
\end{equation}
We assume that $\Sigma_0, \Sigma_\infty \subset P$ correspond to the surjections $V\to \sO_A$, $V\to L$ in this order.

\begin{corollary}\label{corollary log canonical bundle}
In the above notation, assume that $\omega_A$ is trivial. Then the log canonical bundle~$\omega_P\left(\Sigma_0+\Sigma_\infty\right)$ is trivial.
\end{corollary}
\begin{proof}
By Lemma~\ref{lemma class of section of p1 bundle in pic}, we have identifications $\sO_{P}(\Sigma_0)\isom \sO_V(1)\tensor \pi^*L^*$ and $\sO_{P}(\Sigma_\infty)\isom \sO_V(1)$. From \eqref{eq canonical sheaf projective bundle}, we infer that the log canonical bundle is trivial.
\end{proof}

\begin{lemma}\label{lemma nef anticanonical}
Let $A$ be an abelian variety and let $\pi\colon P\to A$ be the bundle projection where $P$ is as in \eqref{eq p1 bundle with two sections}. Then $\omega_P^*$ is nef if and only if $L \in \pic^0(A)$. 
\end{lemma}
\begin{proof}
If $L\in\pic^0(A)$, the vector bundle $V=\sO_A \oplus L$ is nef on $A$, i.e.  $\sO_A(1)$ is nef on $P$; see e.g.~\cite[II, Theorem~6.2.12]{Laz04}. Hence, also the bundle
\[
\omega_P^* \isom \sO_P(2)\tensor \pi^*L^* 
\]
is nef. Conversely, if $\omega_P^*$ is nef, then also its restrictions to the sections $\Sigma_0, \Sigma_\infty \subset P$ are nef. These restrictions are given by
\[
\begin{aligned}
\omega_P^*|_{\Sigma_0} & \isom \sO_A^{\tensor 2}\tensor L^* \isom L^*, \\
\omega_P^*|_{\Sigma_\infty} & \isom L^{\tensor 2}\tensor L^* \isom L. \\
\end{aligned}
\]
But $L$ and $L^*$ are simultaneously nef if and only if $L \in \pic^0(A)$.
\end{proof}

\begin{lemma}\label{lemma degree zero or big}
Let $E$ be an elliptic curve and let $\pi\colon P\to E$ be the bundle projection where $P$ is as in \eqref{eq p1 bundle with two sections}. Then $\omega_P^*$ is not big if and only if $L\in\pic^0(E)$.
\end{lemma}
\begin{proof}
Up to tensoring with $L^*$, we may assume that $\deg L \leq 0$. The anticanonical bundle can be written as
\[
\omega_P^* \isom \pi^*L^* \tensor \sO_P(2) \isom \pi^*L^* \tensor \sO_P(2 \Sigma_\infty)
\]
by Lemma~\ref{lemma class of section of p1 bundle in pic}.

If we assume $\deg L < 0$, then $L^*$ and therefore $\omega_P^*$ is effective. We claim that the bundle $B:=\pi^*L^* \tensor \sO_P(\Sigma_\infty)$ is big and nef. For this purpose, we compute the intersection number
\[
B.\Sigma_\infty = - \deg L + \Sigma_\infty^2 = 0.
\]
So $B$ is nef. From $B^2 = -\deg L >0$, we see that $B$ is also big. Hence, $\omega_P^* = B(\Sigma_\infty)$ is big as well. 

If $\deg L = 0$, the same computations show that $\omega_P^*$ is not big.
\end{proof}

Combining Lemmas~\ref{lemma nef anticanonical} and~\ref{lemma degree zero or big} we obtain

\begin{corollary}\label{corollary nef iff not big}
Let $E$ be an elliptic curve and let $\pi\colon P\to E$ be the bundle projection where $P$ is as in \eqref{eq p1 bundle with two sections}. Then $\omega_P^*$ is nef if and only if it is not big if and only if $L\in\pic^0(E)$.
\end{corollary}

%----------------------------------------------------
\subsection{Normal crossing subvarieties}\label{section normal crossing subvarieties}
%----------------------------------------------------

Let us recall the definition of a normal crossing variety.

\begin{definition} \label{defn:nc+triple}
A \emph{normal crossing} (or nc) variety $Y$ is given locally by the vanishing of one equation of the form $x_1\cdot\ldots\cdot x_k=0$ where $x_1,\ldots, x_k$ is part of a local coordinate system. We say that $Y$ has \emph{simple normal crossings} (snc) if in addition all irreducible components $Y_i$ of $Y$ are smooth.
We say that a normal crossing variety has \emph{no triple intersections} or \emph{no triple points} if in all local equations as above we have $k\leq 2$. 
 A pair $(X,D)$ consisting of a complex space $X$ and a reduced effective Weil divisor $D \subset X$ is called \emph{log smooth} if $X$ is smooth and $D$ is a simple normal crossing divisor in $X$.
\end{definition}

Next we give an embedded version of the notion of a (simple) normal crossing variety.

\begin{definition}
Let $X$ be a smooth variety and $Y \subset X$ a closed subspace. We call $Y$ a \emph{normal crossing subvariety} if for any $p\in Y$ there is an open neighbourhood $U \subset X$ of $p$ and a smooth closed subvariety $V \subset U$ of dimension $\dim V = \dim Y+1$  such that $Y \cap U \subset V$ and there are coordinates $z_1,\ldots,z_{n+1}$ on $V$ such that $Y\cap U$ is defined in $V$ by the equation $z_1\ldots z_k=0$ for some $k\leq n+1$. We say that $Y$ is a \emph{simple normal crossing subvariety} if in addition each irreducible component of $Y$ is smooth.
\end{definition}

It turns out that an (s)nc subvariety is nothing else than a closed subspace that is an (s)nc variety.

\begin{lemma}\label{lemma snc sub is sncsub}
Let $X$ be a smooth variety and let $Y \subset X$ be a closed subspace. Then $Y$ is a (simple) normal crossing subvariety if and only if it is a (simple) normal crossing variety.
\end{lemma}
\begin{proof}
It suffices to prove the claim for the normal crossing case since the additional simplicity requirement is the same on both sides. Clearly, if $Y$ is a normal crossing subvariety, it is also abstractly normal crossing. Let us suppose that $Y$ is a normal crossing variety of dimension $n$. Since the claim is trivial at smooth points, let $p\in Y$ be a singular point. Then, the embedding dimension at $p$ is $n+1$. Let $I \subset \sO_X$ denote the ideal sheaf of $Y$ and let $\mathfrak m_X \subset \sO_{X,p}$ and $\mathfrak m_Y \subset \sO_{Y,p}$ be the maximal ideals. Since $\sO_{Y,p} = \sO_{X,p}/I_p$, we get an exact sequence
\[
0 \to I_p / I_p \cap \mathfrak m_X^2 \to \mathfrak m_X/\mathfrak m_X^2 \to \mathfrak m_Y/\mathfrak m_Y^2 \to 0.
\]
Then we may find local sections $f_1, \ldots, f_\ell$ in $I$, where $\ell:= \dim X - n - 1$, such that their images in $\mathfrak m_X/\mathfrak m_X^2$ generate $I_p / I_p \cap \mathfrak m_X^2$. Let $U \subset X$ be an open neighbourhood of $p$ on which $f_1,\ldots,f_\ell$ are defined and let $V \subset U$ the closed subspace with ideal sheaf generated by $f_1,\ldots,f_\ell$. By construction, the embedding dimension of $V$ at $p$ is $n+1$, and since $V$ is cut out in $U$ by $\ell$ equations, its dimension is at least $\dim X - \ell = n+1$. Thus, by shrinking $U$, we may assume that $V$ is smooth of dimension $n+1$. By construction, it contains $Y\cap U$. Again, up to shrinking $U$ we may assume that $Y\cap U = Y_1\cup\ldots \cup Y_k$ is a union of smooth irreducible components all passing through $p$. Let $z_i=0$ be a defining equations for $Y_i$ for all $i=1,\ldots,k$. Since we know that the $z_i$ are linearly independent in $\mathfrak m_Y/\mathfrak m_Y^2\isom \mathfrak m_V/\mathfrak m_V^2$, where $\mathfrak m_V \subset \sO_{V,p}$ is the maximal ideal, we can, after possibly shrinking $U$ further, extend them to a coordinate system of $V$. So $Y$ is a normal crossing subvariety of $X$.
\end{proof}

%----------------------------------------------------
\subsection{The dual intersection graph}\label{section dual graph}
%----------------------------------------------------

Let $Y$ be an simple normal crossing variety. The dual intersection complex is a cell complex, whose definition goes back at least to Mumford~\cite[II.1, Definition~5]{KKMS73}; see also~\cite[p.~502]{Sha79}. It is obtained by taking a vertex for each irreducible component, and a $p$-simplex for each irreducible component of a nonempty intersection of $p+1$ irreducible components of $Y$. It becomes more involved when passing from simple normal crossing to normal crossing; see e.g.~\cite[\S~4.1]{Thu07}. For our purposes, however, $Y$ will mostly have no triple intersections. This implies that the complex is a graph. In order to describe the intersection behaviour, we recall a construction due to Koll\'ar~\cite[\S~18]{Kol11}; see also~\cite[5.23]{Kol13}.

\begin{construction}\label{construction kollar}
Let $Y$ be a normal crossing variety, let $\nu\colon Y^\nu\to Y$ be the normalisation, denote by $\cond \subset Y^\nu$ the conductor and by $\tau\colon\cond\to\cond$ the covering involution of $\cond\to \nu(\cond)$. We consider the disjoint union $Y^\nu\coprod Y^\nu$ and denote by $\iota_1,\iota_2$ the inclusions of $Y^\nu$. Following Koll\'ar, we define
\[
\widetilde Y:= \left(Y^\nu\coprod Y^\nu\right) /\sim,
\]
where $\sim$ is the equivalence relation generated by $\iota_1(y) \sim \iota_2(\tau(y))$ for all $y\in \cond$. We refer to $\widetilde Y$ as the canonical double cover of $Y$. Note that $\widetilde Y$ has \emph{simple} normal crossings.
\end{construction}

\begin{definition}\label{definition dual graph}
Let $Y$ be a normal crossing variety without triple points, and let $\widetilde Y \to Y$ be its canonical double cover. We denote by $\Delta(\widetilde Y)$ its dual intersection graph\footnote{For us, a graph may have several edges between two vertices. They represent the different connected components of the intersection.}. The covering involution $\varphi\colon \widetilde Y \to \widetilde Y$ acts on $\Delta(\widetilde Y)$ and we define 
\[
\Delta(Y):=\Delta(\widetilde Y)/\left\langle \varphi\right\rangle.
\]
In addition, we endow $\Delta(Y)$ with a label $\lambda_e \in \{1,2\}$ associated to any edge $e$ by setting $\lambda_e:=1$ if and only if $e$ is fixed under the involution $\varphi$. We refer to $\lambda_e$ as the \emph{multiplicity} of $e$. The labelled graph $\Delta(Y)$ is called the \emph{dual intersection graph} or \emph{dual graph} of $Y$. The \emph{valency} of a vertex is the number of edges, counted with multiplicity, that is incident with the vertex.
\end{definition}

\begin{example}\label{example dual graph}
In the left graph in Figure~\ref{figure dual graphs}, the vertex has valency $3$; in the right graph all vertices have valency $2$. The left graph represents an irreducible normal crossing variety $Y$, where $\cond \subset Y^\nu$ has three irreducible components, two of which are identified with one another (giving the loop of multiplicity two) and the third one is glued to itself along a fixed point free involution (giving the loop of multiplicity one). The right graph represents a simple normal crossing variety where $\cond \cap Y_i$ has two connected components for each connected component $Y_i\subset Y^\nu$. 

This example illustrates that the valency of a vertex counts the number of irreducible components of the conductor in the irreducible component of the normalisation corresponding to that vertex. 
\end{example}

\begin{figure}[h]
\centering
\begin{tikzpicture}[
	vertex/.style={circle, draw=black, fill=black, line width=0.9pt, minimum size=5.5mm, inner sep=0pt},
	edge/.style={draw=black, line width=1.2pt}
  ]

	\node[vertex] (A) at (0,  0) {};
	
  \draw[edge] (A) edge[loop left, looseness=20] (A);
  \draw[edge] (A) edge[loop right, looseness=20] (A);
	\node (L) at (1.5,0.2) {$2$};

\begin{scope}[xshift=4.5cm] % move the whole thing to the right
  \node[vertex] (F) at (0, 0.85) {};
  \node[vertex] (G) at (0,-0.85) {};
  \draw[edge, bend left=25]  (F) to (G);
  \draw[edge, bend right=25] (F) to (G);
\end{scope}

\end{tikzpicture}
\caption{Dual graphs with multiplicities. The multiplicity is only displayed if equal to two.}
  \label{figure dual graphs}
\end{figure}
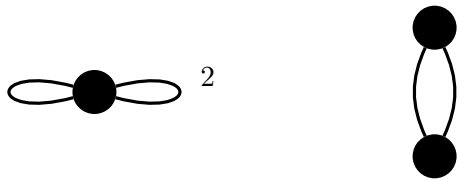

%----------------------------------------------------
\subsection{Logarithmic spaces}\label{section log spaces}
%----------------------------------------------------

We briefly recall some of the relevant notions from logarithmic geometry, adapted to complex spaces instead of algebraic varieties. We refer to~\cite{Ogu18} for a more comprehensive treatment.

\begin{definition}\label{definition log structure}
Let $X$ be a complex space. A \emph{prelog structure} on $X$ consists of a sheaf of commutative monoids $M$  together with a morphism of sheaves of monoids
\[
\alpha\colon  M \to \sO_X,
\]
where $\sO_X$ is viewed as a monoid under multiplication. A \emph{log structure} on $X$ is a prelog structure $(M, \alpha)$ such that the induced morphism
\[
\alpha^{-1}(\sO_X^\times) \to \sO_X^\times
\]
is an isomorphism. A \emph{log space} is a pair $(X, M,\alpha)$ where $X$ is a complex space and $(M,\alpha)$ is a log structure on $X$. We will usually drop the structure map $\alpha$ from the notation. 
\end{definition}

There is also a notion of morphism of log spaces, the definition being identical with the one for schemes~\cite[III, Definition~1.2.1]{Ogu18}

\begin{example}\label{example divisorial logarithmic structure}
Let $X$ be a complex space and $D\subset X$ be an effective Cartier divisor. Let $M_{(X,D)}$ be the subsheaf of monoids of $\sO_X$ defined by
\begin{equation}\label{eq definition divisorial log structure}
M_{(X,D)}=\sO_X\cap j_*\sO_{X\setminus D}^\times
\subset\sO_X,
\end{equation}
where $j\colon X\setminus D\hookrightarrow X$ is the inclusion. This defines a logarithmic structure  on $X$, which we refer to as the \emph{divisorial log structure} associated to $(X,D)$. Suppose $D\subset X$ is reduced and $(X,D)$ is log smooth, let $x \in D$, and let $z_1, \ldots, z_n$ be a local coordinate system in a neighbourhood $U$ of $x$ such that $D$ is cut out by $z_1\cdot\ldots\cdot z_k=0$ for some $k$. Then sections of $M_{(X,D)}$ over $U$ can be written uniquely in the form $uz_1^{n_1}\ldots z_k^{n_k}$ with $u\in \sO_X(U)$ and $n_i\in \N$ for $i=1,\ldots,n$. In other words, there is a canonical isomorphism of monoids 
\begin{equation}\label{eq divisorial log structure local splitting}
M_{(X,D)}(U)\isom \sO_X^\times(U) \oplus \N^k.
\end{equation}
\end{example}

\begin{example}\label{example divisorial logarithmic structure infinitesimal}
There is also an infinitesimal version of Example~\ref{example divisorial logarithmic structure}: the sheaf $M_{(X,D)}$ only depends on the second infinitesimal neighbourhood $D^{(2)}:=2D\subset X$ of $D$ in $X$. For our deformation theoretic application, the following more precise situation is relevant. Let $Y$ be a compact normal crossing variety and let $f_n\colon \sY_n \to S_n:=\Spec(\C[t]/(t^{n+1}))$ be a deformation of $Y$ over the artinian scheme $S_n$. For simplicity, let us suppose that there is some $k< n$ such that $\sY_k:=\sY_n \times_{S_n} S_k\to S_k$ and for $\sY_{k+1}\to S_{k+1}$, all irreducible components of $Y^\sing$ are being smoothed. By this we mean that the singular locus of $f_n$ is contained in $\sY_k$. Then there is a log structure $M_{(X,D)}$ on $X:=\sY_n$ consisting of regular functions whose vanishing locus is contained in the Cartier divisor $D:=\sY_k =\mathrm{div}(f_n^*t^k)\subset \sY_n$. Clearly, the description \eqref{eq definition divisorial log structure} is no longer valid as $X\setminus D=\emptyset$. 
\end{example}

\begin{definition}\label{definition embedding type log structure}
Let $Y$ be a normal crossing variety (cf. Definition~\ref{defn:nc+triple}) with singular locus $\Ysing:=Y^\sing$. A log structure $M$ on $Y$ is said to be \emph{of embedding type}\footnote{Steenbrink uses the term \emph{logarithmic embedding} in for this; see ~\cite[Definition~3.8]{Ste95}.} if for every point $x\in Y$ there exist
\begin{enumerate}
\item a (Euclidean) open neighbourhood $U \subset Y$ of $x$,
\item an open subset $V \subset \C^{n+1}$ and a closed immersion $i\colon U\hookrightarrow V$ of $U$ as a simple normal crossing divisor of $V$,\item an isomorphism of log structures $M|_U \isom i^*M_{(V,U)}$ where $M_{(V,U)}$ is the divisorial log structure from Example~\ref{example divisorial logarithmic structure}.
\end{enumerate}
\end{definition}

\begin{example}\label{example standard log point}
The \emph{standard log point} is the log structure on the one-point space $S:=\Spec(\C)$ determined by the monoid $L:=\C^\times \oplus \N$ and the map $\alpha\colon L \to \C$ given by
\[
\alpha(u,0)= u, \qquad \alpha(u,n)=0 \quad \textrm{if}\quad n>0.
\]
It is also the restriction of the divisorial log structure of the line with its origin from Example~\ref{example divisorial logarithmic structure}. More precisely, $L\isom \iota^*M_{(\C,\{0\})}$ where $\iota$ is the inclusion of the origin in the line.
\end{example}

By~\cite[Definition~3.8]{Ste95} (cf. also~\cite[Lemma~6.1.18]{Mar24}), for a normal crossing variety $Y$ with a log structure $M$ of embedding type, there is an exact sequence of sheaves of monoids
\begin{equation}\label{eq sequence divisorial log structure}
0 \to \sO_Y^\times \to M \to \nu_* \N_{{Y^\nu}} \to 0,
\end{equation}
where $\nu\colon  Y^\nu \to Y$ is the normalisation.

\begin{definition}\label{definition semistable type log structure}
A log structure $M$ of embedding type on $Y$ is said to be \emph{of semistable type} if there exists 
a morphism of sheaves of monoids
\begin{equation}\label{eq semistable type log structure condition}
\psi \colon  \N_Y \to M
\end{equation}
such that the triangle
\[
\xymatrix{
M \ar[r]&  \nu_*\N_{Y^\nu}\\
&\N_Y \ar[ul]^{\psi} \ar[u]
}
\]
is commutative where the horizontal map is the one from \eqref{eq sequence divisorial log structure}.
\end{definition}

In other words, a log structure $M$ of embedding type is of semistable type if there exists a global section of $M$ that is mapped to the global section $(1,\ldots,1)$ of $\nu_*\N_{Y^\nu}$. From this observation, the following is straight forward.

\begin{lemma}\label{lemma log structure semistable type morphism}
Let $(Y,M)$ be a logarithmic space, where is a $Y$ normal crossing variety and $M$ is of embedding type. Denote $(S,L)$ the standard log point from Example~\ref{example standard log point}. Then giving a global section $s$ of $M$ is equivalent to giving a morphism of log spaces $f\colon(Y,M) \to (S,L)$, where $f^{-1}L \to M$ sends $1\mapsto s$. In particular, if $M$ is of semistable type, there is such a morphism $f$ where $s$ maps to $(1,\ldots,1)$ in $\nu_*N_{Y^\nu}$. 
\end{lemma}

We recall a different approach due to Kawamata--Namikawa; see~\cite[\S~2]{KN94}.

\begin{definition}\label{definition log atlas kawamata namikawa}
Let $Y$ be a normal crossing variety. A \emph{log atlas} on $Y$ is a partial open covering of $Y$ together with chosen local coordinates
\[
\mathfrak U=\{U_i\, ; \, z_0^{(i)},\ldots,z_n^{(i)}\}_{i\in I}
\]
such that the following conditions are satisfied.
\begin{enumerate}
\item The union $\bigcup_i U_i$ contains the singular locus $\Ysing$ of $Y$.
\item For every $i\in I$, there is an isomorphism $\varphi_i\colon U_i \to V_i$ where $V_i$ is an open neighbourhood of $0 \in \{(x_0,\ldots,x_n)\in \C^{n+1};\ x_0\cdots x_r=0\}$ for some $r=r(i)$ such that $\varphi_i^*(x_\ell)=z_\ell^{(i)}$ for $0\leq \ell\leq r$ and such that $z_\ell^{(i)}$ for $r<\ell\leq n$ are invertible.

\item If $U_i\cap U_j\neq \emptyset$, then there exist invertible holomorphic functions $u_\ell^{(ij)}$ for $0\leq \ell\leq n$, on $U_i\cap U_{j}$ and a permutation $\sigma=\sigma(i,j)$ such that $z_{\sigma(\ell)}^{(i)}=u_\ell^{(ij)}z_\ell^{(j)}$ and 
\begin{equation}\label{eq kawamata namikawa unit condition}
u_0^{(ij)}\cdots u_{r}^{(ij)}=1 \quad\text{on } U_i\cap U_{j}.
\end{equation}
\end{enumerate}

Two log atlases $\mathfrak U$ and $\mathfrak U'$ on a normal crossing variety $Y$ are called \emph{equivalent} if their union defines a log atlas on $Y$. 
\end{definition}

In the terminology of Kawamata--Namikawa, an equivalence class of log atlases is referred to as a log structure. We avoid this terminology, since it does not coincide with the (nowadays standard) one from Definition~\ref{definition log structure}. There is however the following description in terms of logarithmic geometry.

\begin{lemma}\label{lemma kawamata namikawa is semistable type}
Let $Y$ be a normal crossing variety. Giving an equivalence class of log atlases on $Y$ is equivalent to giving a log structure $M$ of semistable type in the sense of Definition~\ref{definition semistable type log structure} together with a morphism $\psi\colon \N_Y\to M$ as in~\eqref{eq semistable type log structure condition}.
\end{lemma}
\begin{proof}
The equivalence follows in citeable form from comparing~\cite[Proposition~1.1]{KN94} with~\cite[Theorem~11.7]{Kat96}, where both notions are shown to be equivalent to the notion of $d$-semistability. We will discuss this notion in Definition~\ref{definition d-semistable nc}. The equivalence has also been noticed by Steenbrink\footnote{Steenbrink uses the term \emph{logarithmic deformation} for what we call a log structure of semistable type.}; see~\cite[Remark~3.10]{Ste95}. The claim can also be seen directly, and the condition~\eqref{eq semistable type log structure condition} corresponds precisely to \eqref{eq kawamata namikawa unit condition}.
\end{proof}

%----------------------------------------------------
\subsection{Logarithmic differential forms}\label{section log differentials}
%----------------------------------------------------

Let $f\colon (X,M)\to(S,N)$ be a morphism of log spaces. Then there is an associated sheaf $\Omega_{(X,M)/(S,N)}$ of logarithmic Kähler differentials; see~\cite[Theorem~IV.1.2.4]{Ogu18}. It has a similar universal property as the sheaf $\Omega_{X/S}$ of Kähler differentials and there is a canonical morphism $\Omega_{X/S} \to \Omega_{(X,M)/(S,N)}$, Moreover, it comes with a morphism $d\log\colon M\to\Omega_{(X,M)/(S,N)}$ satisfying
\[
d\alpha(m)=\alpha(m)d\log(m),
\qquad
d\log(f^\flat(n))=0,
\]
where $d\alpha(m)$ is the image of the corresponding section from $\Omega_{X/S}$ and $f^\flat\colon f^{-1}N\to M$ is the morphism on monoids coming from $f$.

\begin{example}\label{example divisorial log differentials}
Let $X$ be a smooth variety and let $D\subset X$ be a simple normal crossing divisor. Then there is a canonical isomorphism of $\sO_X$-modules
\[
\Omega_{(X,M_{(X,D)})} \isom \Omega_X(\log D),
\]
where $M_{(X,D)}$ is the divisorial log structure from Example~\ref{example divisorial logarithmic structure}.
\end{example}

\begin{remark}\label{remark structure morphism semistable type}
Let $(Y,M)$ be a log space where $Y$ is a normal crossing variety and $M$ is a log structure of semistable type; see Definition~\ref{definition semistable type log structure}. Recall that there exists a morphism $(Y,M)\to (S,L)$ to the standard log point and that by
Lemma~\ref{lemma log structure semistable type morphism} such a morphism corresponds to a global section of $M$ mapping to $(1,\ldots,1)$ under the map to $\nu_*\N_{Y^\nu}$ from \eqref{eq sequence divisorial log structure}. In particular, the morphism is \emph{not} part of the datum of a log structure of semistable type. However, from \eqref{eq sequence divisorial log structure} again, we see that for any two such morphisms, the corresponding global sections of $M$ differ by a global section of $\sO_Y^\times$. If $Y$ is compact and connected, then $\Gamma(Y,\sO_Y^\times)=\C^\times = \Gamma(S,\sO_S^\times)$. Hence, the ambiguity comes from pulling back a function from $S$ and thus $\Omega_{(Y,M)/(S,L)}$ is unambiguously defined: it only depends on the log structure of semistable type and not on the morphism to the standard log point.
\end{remark}

For the following result, recall that $\sK_Y$ denotes the sheaf of total quotient rings of a locally ringed space $Y$; see~\cite{Kle79} or~\cite[Definition~31.23.3]{Stacks}.

\begin{proposition}\label{proposition relative log differentials into meromorphic differentials}
Let $Y$ be a simple normal crossing variety. Let $M$ be a log structure of semistable type on $Y$, and denote by $(S,L)$ the standard log point. Then there is a canonical injective morphism 
\begin{equation}\label{eq differentials inject into normalization}
\Omega_{(Y,M)/(S,L)}\; \into \; \Omega_Y \tensor \sK_Y.
\end{equation}
\end{proposition}
\begin{proof}
On $U:=Y^\reg$, the morphism $f\colon (Y,M)\to (S,L)$ is strict in the sense that $M|_U= (f|_U)^*L$. Thus, the definition of the sheaf of logarithmic differentials implies that the natural map $\Omega_{U}\to \Omega_{(U,M|_U)/(S,L)}$ is an isomorphism. Hence, the natural map
\[
\Omega_Y\otimes_{\sO_Y}\sK_Y \xrightarrow{\isom} \Omega_{(Y,M)/(S,L)}\otimes_{\sO_Y}\sK_Y
\]
is an isomorphism. Since $\sO_Y\to \sK_Y$ is always injective and $\Omega_{(Y,M)/(S,L)}$ is locally free, we obtain a canonical injection 
\[
\Omega_{(Y,M)/(S,L)} \hookrightarrow \Omega_{(Y,M)/(S,L)}\tensor_{\sO_Y}\sK_Y \isom
\Omega_{Y}\tensor_{\sO_Y}\sK_Y,
\]
as claimed.
\end{proof}

The proposition says that relative log differentials for log structures of semistable type are canonically a subsheaf of the sheaf of meromorphic differentials. It turns out that this subsheaf does not depend on the log structure of semistable type.

\begin{corollary}\label{corollary independence of log differentials}
In the situation of Proposition~\ref{proposition relative log differentials into meromorphic differentials}, the sheaf of log differentials $\Omega_{(Y,M)/(S,L)}$ is independent, as a subsheaf of $\Omega_Y \tensor \sK_Y$, of the choice of $M$ among log structures of semistable type.
\end{corollary}
\begin{proof}
This is a direct consequence of Proposition~\ref{proposition relative log differentials into meromorphic differentials} and the definition of log structures of semistable type. The former allows to argue locally, as verifying equality of two subsheaves is a local question. The latter implies that locally, all such $M$ are isomorphic as log structures. 
\end{proof}

\begin{definition}\label{definition relative omega log}
Let $Y$ be a normal crossing variety that admits a log structure of semistable type. The sheaf of \emph{semistable relative log differentials} is defined as
\[
\Omega_{Y}(\log):=\Omega_{(Y,M)/(S,L)},
\]
where we choose some log structure $M$ of semistable type on $Y$ and a morphism $(Y,M) \to (S,L)$ of log spaces to the standard log point.
Note that the sheaf $\Omega_Y(\log)$ is well defined thanks to Remark~\ref{remark structure morphism semistable type} and Corollary~\ref{corollary independence of log differentials}.
\end{definition}

\begin{corollary}\label{corollary relative log differentials on normalization component}
Let $Y$ be a normal crossing variety admitting a log structure of semistable type. Let $\nu\colon Y^\nu\to Y$ be the normalisation and $\nu_i\colon Y_i\to Y$ the restriction of $\nu$ to an irreducible component of $Y_i \subset Y^\nu$. 
Denote by $\cond \subset Y^\nu$ the conductor and $B_i:= Y_i \cap \cond$. Then there is a canonical isomorphism
\[
\nu_i^*\Omega_{Y}(\log) \isom \Omega_{Y_i}(\log B_i).
\]
\end{corollary}
\begin{proof}
By Proposition~\ref{proposition relative log differentials into meromorphic differentials}, the sheaf $\nu_i^*\Omega_{Y}(\log)$ embeds canonically into $\Omega_{Y_i}\tensor \sK_{Y_i}$. Then the claim follows from a local computation.
\end{proof}

\begin{lemma}\label{lemma log differentials pull back}
Let $Y$ be a normal crossing variety and $p\colon Y'\to Y$ an étale morphism. If $M$ is a log structure on $Y$ of embedding or semistable type, then also $M':=p^*M$ is a log structure on $Y'$ of embedding or semistable type. Moreover, $\Omega_{Y'}(\log) = p^*\Omega_{Y}(\log)$.
\end{lemma}
\begin{proof}
Since $p$ is a local isomorphism, a local embedding of $Y$ as an snc divisor induces a local embedding of $Y'$ as an snc divisor. This is compatible with the pullback of log structures so that the pullback log structure $M'$ on $Y'$ is of embedding or semistable type, just as $M$ is. For the claim about differentials, it suffices to show that $\Omega_{(Y',M')/(Y,M)}=0$. Since the induced morphism of log structures is strict (i.e. $M'=p^*M$), we have $\Omega_{(Y',M')/(Y,M)}=\Omega_{Y'/Y}$. The latter vanishes since $p$ is étale.
\end{proof}

%----------------------------------------------------
%----------------------------------------------------
\section{Deformation theory}\label{section deformation theory}
%----------------------------------------------------
%----------------------------------------------------

This section provides the basic deformation theoretic background needed for our unobstructedness results in Section~\ref{section unobstructedness}. The main deformation problem we encounter consists in the deformation of morphisms. We are mainly interested in deformations of closed immersions or, slightly more generally, étale covers of closed subvarieties. This is probably well known to experts but somewhat difficult to trace in the literature.

%----------------------------------------------------
\subsection{Deformations of Morphisms}\label{section deformations of morphisms}
%----------------------------------------------------

We will first recall some general deformation problems related to morphisms. 

\begin{definition}\label{definition deformations}
Let $f:Y \to X$ be a (proper) morphism between compact complex spaces and let $(S,0)$ be a pointed complex space (that is, a complex space together with a distinguished point). A \emph{deformation of $f$} over $(S,0)$ is a diagram
\begin{equation}\label{eq deformation of morphism}
\xymatrix{
Y \ar[r]\ar[d]^f & \sY \ar[d]^F\\
X \ar[r]\ar[d] & \sX \ar[d]\\
\{0\} \ar@{^(->}[r] & S\\
}
\end{equation}
where $F:\sY \to \sX$ is a morphism, $\sY \to S$ and $\sX \to S$ are proper and flat and $f:Y\to X$ is the base change of $F:\sY \to\sX$ along the inclusion $\{0\}\hookrightarrow S$. It is called a \emph{deformation of $f$ with fixed target} over $(S,0)$ if $\sX = X\times S$ and $\sX \to S$ is the projection. 

Let $F\colon \sY\to X\times S$, $F'\colon \sY'\to X\times S$ be deformations of $f$ with fixed target over $S$. An isomorphism of deformations of $f$ with fixed target is an $S$-isomorphism $\Phi\colon \sY \to\sY'$ such that the diagram
\begin{equation}\label{eq definition isomorphism deformation with fixed target}
\xymatrix{
&Y\ar[dl]\ar[dr] & \\
\sY\ar[dr]_F \ar[rr]^\Phi && \sY'\ar[dl]^{F'}\\
&X\times S&\\
}
\end{equation}
is commutative. 
\end{definition}

\begin{notation}\label{notation deformation functors}
Let $\Art_\C$ be the category of artinian local $\C$-algebras with residue field $\C$. With the notation of Definition~\ref{definition deformations}, we denote by 
$$
D_f^\ft\colon\Art_\C \to \Set
$$
 the (set-valued) functor of deformations of $f$ with fixed target. To $A\in\Art_\C$ it associates the set $D_f^\ft(A)$ of isomorphism classes of deformations $F$ of $f$ with fixed target over $S:=\Spec(A)$.
\end{notation}

\begin{proposition}\label{proposition t1 space deformations with fixed target}
Let $X$ be a complex manifold and let $f\colon Y \to X$ a proper morphism that is étale onto a closed subvariety. For $A\in \Art_\C$, we set $T:=\Spec(A)$. Suppose that $F\colon \sY\to X\times T$ be a deformation of $f$ over $T$ with fixed target.

Let $K_F^\bullet$ be the complex $\bigl[F^*\Omega_{X\times T/T}\to \Omega_{\sY/T}\bigr]$ concentrated in degrees $-1,0$. If we denote $T':=\Spec(B)$ where $B:=A[t]/(t^2)$, then there is a canonical identification of the fibre of $D_f^\ft(B)\to D_f^\ft(A)$ over $F$ with 
\[
\mathbb E xt^1_\sY(K_F^\bullet,\sO_{\sY}).
\]
\end{proposition}
\begin{proof}
By Lemma~\ref{lemma homological nonsense} below, elements of $\mathbb E xt^1_\sY(K_F^\bullet,\sO_\sY)$ are in bijection with equivalence classes of extensions
\begin{equation}\label{eq extension of complexes for t1}
0 \to \sO_\sY \to L^\bullet \to K_F^\bullet \to 0
\end{equation}
in the abelian category of complexes in $\Coh(\sY)$. Here, $L^\bullet$ is a complex of coherent $\sO_\sY$-modules concentrated in degrees $-1,0$. In particular, the degree zero part of such an extension is an extension of $K_F^0=\Omega_{\sY/T}$ by $\sO_\sY$, and thus gives an element $\xi^0\in\Ext^1_\sY(\Omega_{\sY/T},\sO_\sY)$. We claim that the class $\xi^0$ gives a lift of $\sY\to T$ over $T'$. For $A=\C$, this is well known; see e.g.~\cite[Theorem~1.1.10]{Ser06}. Note that the reference is for schemes, but the argument works mutatis mutandis for complex spaces. We will sketch the argument for a general $A\in \Art_\C$, even though it is essentially the same as for $A=\C$.

Let us recall how this lift of $\sY$ to $T'$ is constructed. For this purpose we consider the diagram
\begin{equation}\label{eq extension of complexes in degree zero for t1}
\xymatrix{
0 \ar[r] & \sO_\sY \ar[r]^\chi\ar@{=}[d] & L^0 \ar[r]^\psi & \Omega_{\sY/T} \ar[r] & 0\\
0 \ar[r] & \sO_\sY \ar[r] & \sA \ar[r]\ar[u] & \sO_\sY \ar[r] \ar[u]_d & 0\\
}
\end{equation}
where the first row is the extension corresponding to $\xi^0$ and the second row is obtained from it by pullback along the exterior differential $d\colon\sO_\sY \to \Omega_{\sY/T}$. Recall that by definition this means that the middle term is the fibre product $\sA:=\sO_\sY \times_{\Omega_{\sY/T}} L^0$ and the morphisms in the bottom row are the projection to the first factor and the inclusion of $\sO_\sY$ via $\chi$ into the second factor.  

Denote $\varepsilon:=\bar t\in B$ the residue class of $t$. Then the sheaf $\sA$ is given a $B$-algebra structure by defining the multiplication
\[
\forall (a,e),(a',e') \in \sA = \sO_\sY \times_{\Omega_{\sY/T}} L^0: \quad (a,e)\cdot (a',e'):= (aa',ae'+a'e)
\]
and the scalar multiplication by $\varepsilon$ as $\varepsilon\cdot (a,e):=(0,\chi(a))$. Then $\sY'$ is defined to be the locally ringed space $(Y,\sA)$.

To upgrade this to a deformation $F'\colon \sY' \to X\times T'$ of $f$ over $T'$ lifting $F$, we have to take into account the degree $-1$ part of \eqref{eq extension of complexes for t1}. What it contributes is a lift $\varphi$ of $F^*$, i.e. a morphism making the diagram
\begin{equation}\label{eq degree minus one lift}
\xymatrix{
& F^*\Omega_{X\times T/T} \ar[d]^(0.45){F^*}\ar[dl]_(0.55)\varphi \\
L^0 \ar[r]_\psi & \Omega_{\sY/T} \\
}
\end{equation}
commutative. Since the continuous map underlying the sought-for $F'\colon\sY' \to X \times T'$ is determined by $f$, it remains to construct a morphism $\sO_{X\times T'}=\sO_{X\times T}[\varepsilon]\to {F'}_*\sO_{\sY'}$ of sheaves of $B$-algebras whose reduction modulo $\varepsilon$ is the given morphism $F^\sharp\colon\sO_{X\times T} \to F_*\sO_\sY$. Such a morphism is determined by a morphism 
\[
\sO_{X\times T}\to {F'}_*\sO_{\sY'} = F_*\sO_Y \times_{F_*\Omega_{\sY/T}} F_*L^0
\]
of sheaves of $A$-algebras. We obtain the latter by
\[
b \mapsto (F^\sharp(b),\varphi(db)),
\]
where we abuse notation and write $\varphi$ for the morphism obtained from $\varphi$ in~\eqref{eq degree minus one lift} by adjunction. This is well defined since $dF^\sharp(b) = F^*db = \psi(\varphi(db))$ and therefore indeed gives a deformation $F'$ of $f$ over $T'$ lifting $F$.

Conversely, let $F'\colon \sY' \to X \times T'$ be a lift of $F$. Let us write $T'=T\times S_1$ with $S_1:=\Spec(\C[t]/(t^2))$. In particular, there is a projection $T'\to T$. Then we can write $\Omega_{X\times T'/T} = \pr^*\Omega_{S_1}\oplus \Omega_{X\times T'/T'}$, where $\pr\colon X\times T' \to T$ is the projection. Next we consider the commutative diagram
\[
\xymatrix{
& & F^*\Omega_{X\times T'/T'} \tensor \sO_\sY  \ar[r]^(0.6)\isom \ar[d] & F^*\Omega_{X\times T/T}  \ar[d]& \\
0 \ar[r] & \sO_\sY \ar[r] & \Omega_{\sY'/T} \tensor \sO_\sY  \ar[r] & \Omega_{\sY/T} \ar[r]& 0,\\
}
\]
where the left vertical map is obtained by restricting the cotangent map ${F'}^*\Omega_{X\times T'/T}\to \Omega_{\sY'/T}$ and the aforementioned splitting, and the bottom row is exact by the argument in~\cite[Theorem~1.1.10]{Ser06}. This defines an extension of $K_F^\bullet$ by $\sO_Y$ in the category of complexes in $\Coh(\sY)$ and therefore an element in $\mathbb E xt^1_\sY(K_F^\bullet,\sO_\sY)$. We leave it to the reader to verify that these constructions are mutually inverse.
\end{proof}

In the special case $A=\C$, we have proven the following result.

\begin{corollary}\label{corollary tangent spaces}
Let $X$ be a complex manifold and let $f\colon Y \to X$ a proper morphism that is étale onto a closed subvariety. 

Let $K_f^\bullet$ be the complex $\bigl[f^*\Omega_X\to \Omega_Y\bigr]$ concentrated in degrees $-1,0$. Then the tangent space to the functor $D_f^\ft$ of deformations of $f$ with fixed target is given by
\[
T^1_{D_f^\ft} = \mathbb E xt^1_Y(K_f^\bullet,\sO_Y).
\]
\end{corollary}

\begin{lemma}\label{lemma homological nonsense}
Let $\sA$ be an abelian category and let 
\begin{equation}\label{eq homological triangle}
A\to B \to C \to A[1]
\end{equation}
be an exact triangle in the bounded derived category $D^b(\sA)$. Suppose that $A$ is a complex concentrated in degree $0$ and $C$ is a complex concentrated in degrees $-1,0$. Then there is an exact sequnce $0\to A \to B' \to C\to 0$ of complexes which induces the given triangle. Here, $B'$ can be chosen to be a complex concentrated in degrees $-1,0$.
\end{lemma}
\begin{proof}
First, we see that $B$ may be represented by a $2$-term complex in $\sA$: just replace $B$ by $\tau_{\geq -1}\tau_{\leq 0} B$ where $\tau$ denotes the truncation for the standard $t$-structure; see also~\cite[Exercise~2.31]{Huy06}. Denoting the individual terms of a complex $K$ by $K^i$ for $i\in\Z$, we may consider the short exact sequence $0\to C^0\to C \to C^{-1}[1]\to 0$ in the category of chain complexes in $\sA$. Then we get a long exact sequence
\[
\xymatrix{
\ldots \ar[r] & \Hom_\sD(C,A[1]) \ar[r]^f & \Hom_\sD(C^0,A[1]) \ar@{=}[d]\ar[r]^g & \Hom_\sD(C^{-1}[1],A[2]) \ar@{=}[d]\\
& & \Ext_\sA^1(C^0,A^0) \ar[r] & \Ext^1_\sA(C^{-1},A^0)\\
}
\]
where $\sD:=D^b(\sA)$. The triangle \eqref{eq homological triangle} gives rise to a connecting homomorphism $\delta \in \Hom_\sD(C,A[1])$. Then $f(\delta)$ corresponds to an extension
\[
0 \to A^0 \to E \to C^0 \to 0.
\]
Since $g(f(\delta))=0$, this the pullback extension along the differential $d$ of $C$ splits. In particular, there is a  lift of $d$ to $E$ as in the following diagram.
\[
\xymatrix{
C^{-1} \ar@{-->}[d]_{\exists d_E} \ar@{=}[r]& C^{-1} \ar[d]^d\\
E \ar[r] & C^0\\
}
\]
We define $B'$ to be the left column.
\end{proof}

\begin{remark}
For consistency, the reader may compare Corollary~\ref{corollary tangent spaces} to \cite[Theorem~3.4.8]{Ser06}. The former treats arbitrary deformations with fixed target, the latter locally trivial deformations in the sense of~\cite[3.4.2]{Ser06}. The theorem proves that the tangent space for locally trivial deformations is $\mathbb H^1(Y,(K_f^\bullet)^*)$, which is naturally a subspace of $\mathbb E xt^1_Y(K_f^\bullet,\sO_Y)$ via the local to global spectral sequence for $\Ext$ groups of $K_f^\bullet$ into $\sO_Y$.
\end{remark}

%----------------------------------------------------
\subsection{Automorphisms}\label{section automorphisms}
%----------------------------------------------------

We show next the existence of a universal deformation for deformations with fixed target. It is a consequence of the absence of automorphisms. 

\begin{proposition}\label{proposition universal deformation}
Let $X$ be a complex manifold and let $f\colon Y \to X$ a proper morphism that is étale onto a closed subvariety. Then the functor $D_f^\ft$ has a formal universal deformation.
\end{proposition}
\begin{proof}
Let $A\in \Art_\C$ and $S:=\Spec(A)$. Let $F:\sY \to X\times S$ be a deformation of $f$ over $S$. We will show first that $F$ does not have any nontrivial automorphisms. Suppose that $\Phi\colon \sY \to \sY$ is an automorphism of $F$, that is, an $S$-morphism making the diagram 
\[
\xymatrix{
&Y\ar[dl]\ar[dr] & \\
\sY\ar[dr]_F \ar[rr]^\Phi && \sY\ar[dl]^{F}\\
&X\times S&\\
}
\]
commute; see \eqref{eq definition isomorphism deformation with fixed target}. In particular, $\Phi$ restricts to $Y$ as the identity. On the level of topological spaces, $\Phi$ is therefore the identity as well. It suffices to show that $\Phi^*\colon \sO_{\sY,y}\to\sO_{\sY,y}$ is the identity for every point $y\in \sY$. Being an automorphism of a deformation with fixed target means that the diagram
\[
\xymatrix{
&\sO_{X\times S, F(y)} \ar[dl]_{F^*} \ar[dr]^{F^*}&\\
\sO_{\sY,y} \ar[rr]^{\Phi^*}&&\sO_{\sY,y} \\
}
\]
commutes. The stalks are to be understood in the analytic topology, so $F^*$ is surjective and it follows that $\Phi^*$ must be the identity.

For the existence of a universal deformation, we first note that the existence of a semiuniversal deformation is shown literally as in the proof of~\cite[Theorem~3.4.8]{Ser06}. The universality claim now essentially follows from the argument in~\cite[Theorem~2.6.1]{Ser06}. We sketch it here for convenience. 
To show that a semiuniversal formal element for $D_f^\ft$ is universal, it suffices to show by~\cite[Theorem~2.11]{Sch68} that the following condition is satisfied. Whenever $A'\to A$ is a small extension\footnote{That is, its kernel is annihilated by the maximal ideal $\mathfrak m'$ of $A'$.} in $\Art_\C$, then the canonical morphism 
\begin{equation}\label{eq schlessinger condition}
D_f^\ft(A'\times_A A') \to D_f^\ft(A')\times_{D_f^\ft(A)} D_f^\ft(A')
\end{equation}
is injective. Note that Schlessinger requires this morphism to be an isomorphism, but the surjectivity already follows from the existence of a semiuniversal formal element: given deformations $F_i\colon \sY_i \to X \times S'$ for $i=1,2$ of $f$ over $S':=\Spec(A')$ that are isomorphic over $S:=\Spec(A)$ in the sense of~\eqref{eq definition isomorphism deformation with fixed target}, one can always form the cofibre product deformation $\widetilde F \colon \widetilde\sY \to \widetilde S$, where $\widetilde S:=\Spec(A'\times_AA')$. This $\widetilde F$ is a preimage of $(F_1,F_2)$ under the morphism \eqref{eq schlessinger condition}. Suppose that $\widetilde F'\colon \widetilde \sY' \to X\times \widetilde S$ is another preimage. Then we obtain an automorphism of $F:=F_1\times_{A'}A=F_2\times_{A'}A$ as the composition of isomorphisms
\[
\sY\isom \sY_1\times_{S'}S \isom \widetilde \sY' \times_{\widetilde S} S \isom \sY_2\times_{S'}S \isom \sY
\]
over $X\times S$. By the first part of the proof, this isomorphism is necessarily the identity. Hence, in the diagram
\[
\xymatrix@R=5mm{
&\sY_1\ar[dr]\ar@/^5mm/[drr]&&\\
\sY \ar[ru]\ar[rd] && \widetilde \sY \ar@{-->}[r]^{\exists}& \widetilde\sY'\\
& \sY_2\ar[ru]\ar@/_5mm/[rru]&&\
}
\]
of $X\times \widetilde S$-morphisms, the outer diagram is commutative. By the universal property of the cofibre product, we obtain the  morphism $\widetilde \sY \to \widetilde\sY'$. Since it restricts to the identity over $\Spec(\C)$, it is an isomorphism by flatness.
\end{proof}

Let $f\colon Y \to X$ be a proper morphism between compact complex spaces with an equivariant $G$-action.
We show that this induces an action on the universal deformation with fixed target.

\begin{proposition}\label{proposition group action}
Let $X$ be a complex manifold and let $f\colon Y \to X$ a proper morphism that is étale onto a closed subvariety. Let
\begin{equation}\label{eq universal deformation induced action}
\xymatrix{
\sY \ar[dr]_\pi\ar[r]^F & X \times S \ar[d]^{\pr_2} \\
& S
}
\end{equation}
be the universal deformation of $f$, either in the category of formal complex spaces or of germs of complex spaces. Let $0\in S$ be the distinguished point such that the fibre of \eqref{eq universal deformation induced action} over it is $f$. 
Let $G\subset \Aut(Y)$ be a group of holomorphic automorphisms of $Y$ such that $f$ is $G$-invariant, i.e. $f\circ g=f$ for every $g\in G$. Then there is a canonical action of $G$ on $\sY$ and on $S$ extending the action on $Y=\sY_0$ such that $\pi$ is equivariant.  
\end{proposition}
\begin{proof}
By Definition~\ref{definition deformations}, the identification $Y \isom \sY_0$ of the central fibre is part of the datum of a deformation. Let us denote $\iota\colon Y\to \sY$ the resulting inclusion. For $g\in G$, we consider the deformation
\[
\xymatrix{
Y \ar[d]\ar[r]^{\iota\circ g^{-1}} & \sY \ar[d]^\pi\\
\{0\} \ar[r] & S\\
}
\]
of $Y$. By universality, there are unique isomorphisms
\[
\varphi_g\colon S \to S \quad \textrm{ and } \quad \Phi_g\colon \sY \xrightarrow{\isom} \sY\times_{(S,\varphi_g)}S\to \sY
\]
such that $\pi \circ \Phi_g = \varphi_g \circ \pi$. We will show next that 
\[
\Phi\colon G \to \Aut(\sY), \quad g\mapsto \Phi_g \quad \textrm{and}\quad 
\varphi\colon G \to \Aut(S), \quad g\mapsto \varphi_g
\]
are homomorphisms. For $\Phi$, this follows from the commutativity of the diagram
\[
\xymatrix{
&Y\ar[dl]_{\iota \circ g^{-1}\circ h^{-1}}\ar[d]|{\iota \circ h^{-1}}\ar[dr]^\iota & \\
\sY \ar[r]_{\Phi_g}\ar@/_1cm/[rr]_{\Phi_{hg}} & \sY \ar[r]_{\Phi_h} & \sY\\
}
\]
for any $h,g\in G$. The claim for $\varphi$ follows in the same way. 
\end{proof}

Note that even though $f$ is invariant for the action of $G$ on $Y$, it is not necessarily true that the universal deformation $F$ is invariant for the induced action of $G$.

%----------------------------------------------------
\subsection{\texorpdfstring{$T^1$}{T1}-lifting principle}\label{section t1 lifting}
%----------------------------------------------------

The $T^1$-lifting principle due to Kawamata--Ran \cite{Ran92,Kaw92t1,Kaw97err} allows to prove unobstructedness results even if the obstruction spaces are nonzero. Recall Kawamata's definition of $T^1$-spaces from~\cite{Kaw92t1}. Let $D$ be a deformation functor, which for simplicity we assume to be prorepresentable; see~\cite{Kaw97err} for a more general statement. Let $A\in \Art_\C$ and put $B:=A[\varepsilon]$ the ring of dual numbers with coefficients in $A$. We denote by $\varphi:B \to A$ the canonical projection. For a deformation $\xi \in D(A)$ we define 
\begin{equation}\label{eq kawamata t1}
    T_D^1(\xi/A):=\{ \xi' \in D(B) \mid D(\varphi)(\xi')=\xi\}.
\end{equation}
If $D$ is clear from the context, we will write $T^1(\xi/A)$ instead of $T_D^1(\xi/A)$. We abbreviate $A_n:=\C[t]/(t^{n+1})$. The $T^1$-lifting criterion of Kawamata--Ran says the following. For a given deformation $\xi_n \in D(A_n)$, we denote $\xi_{n-1}\in D(A_{n-1})$ the image of $\xi_n$ under the map induced by the canonical projection $A_n\to A_{n-1}$. Then $D$ is unobstructed if for all $n \in \N$ and all $\xi_n\in D(A_n)$ the morphism
\begin{equation}\label{eq kawamata t1 lifting}
\xymatrix{
T^1(\xi_{n}/A_{n}) \ar[r] &  T^1(\xi_{n-1}/A_{n-1})\\
}
\end{equation}
is surjective; see~\cite[Theorem~1.1]{Kaw92t1}. As a next step, we identify the $T^1$-spaces for the functors of interest to us.

\begin{corollary}\label{corollary kawamata t1 for fixed target}
Let $f\colon Y \to X$ be a proper morphism from a reduced compact complex space to a compact complex manifold. For $A\in \Art_\C$ let $S:=\Spec(A)$, and let
\[
\xymatrix{
&\sY\ar[r]^F\ar[ddr] & X\times S\ar[dd]\\
Y\ar[ru]^\iota\ar[r]^f\ar[dr] & X\ar[d]\ar[ur]&\\
& \{0\} \ar@^{^(->}[r] & S\\
}
\]
be a deformation of $f$. Let $K_F^\bullet$ be the complex $\bigl[F^*\Omega_{X\times S/S}\to \Omega_{\sY/S}\bigr]$ concentrated in degrees $-1,0$. Then we have
\[
T^1_{D_f^\ft}(F/A)=\mathbb Ext^1_Y(K_F^\bullet,\sO_\sY).
\]
\end{corollary}
\begin{proof}
This is a reformulation of Proposition~\ref{proposition t1 space deformations with fixed target}.
\end{proof}

%----------------------------------------------------
%----------------------------------------------------
\section{The case when \texorpdfstring{$D$}{D} is not uniruled}
%----------------------------------------------------
%----------------------------------------------------

In this section we give a description of the divisor $D$ and its singularities if $D$ is not uniruled.
We start with a structure statement for non-projective K\"ahler threefolds, then show how to apply this in the hyperk\"ahler setting.

\begin{proposition} \label{prop:normalisation0}
Let $D$ be a compact Kähler Gorenstein variety of dimension three which is not uniruled, and let $\holom{\nu}{D^\nu}{D}$ be the normalisation of $D$. Assume that the canonical bundle $K_D$ is nef with $\nd(K_D)=1$, and assume that $K_D$ has degree zero on all curves on $D$. Then
\begin{enumerate}[\normalfont (a)]
\item the variety $D^\nu$ is not Moishezon;
\item the divisor $\nu^* K_D$ is semiample and has degree zero on all curves on $D^\nu$.
\end{enumerate}
\end{proposition}

\begin{proof}
By the projection formula, we have $ \nu^* K_D \cdot C=K_D \cdot \nu_* C=0 $ for every curve $C$ on $ D^\nu$. By Lemma~\ref{lemma:notmoishezon}, this implies (a) and the second statement in (b).

In the remainder of the proof, we show the first statement in (b).

Let $f\colon W\to D^\nu$ be a resolution of $D^\nu$, and set $\sL:=f^* \nu^* \sO_D(K_D)$. It suffices to show that $\sL$ is semiample. Note that $\sL$ is a nef line bundle with $\nd(\sL)=1$. By the projection formula and by the assumption of the proposition we have $ \sL\cdot C=K_D \cdot \nu_*f_* C=0 $ for every curve $C$ on $W$.

There exist effective $f$-exceptional $\Q$-divisors $F$ and $G$ on $W$ such that
$K_W+F\sim_\Q f^*K_{D^\nu}+G$, and by Proposition~\ref{normalisation-formula} there exists an effective Weil divisor $\Xi$ on $D^\nu$ such that $K_{D^\nu}\sim\nu^*K_D-\Xi$. Setting $E:=F+f^*\Xi\geq0$, we conclude from the last two formulas that
\begin{equation}\label{eq:656}
\sO_W(K_W+E)\isom_\Q \sL\otimes\sO_W(G).
\end{equation}
Note that 
$$\kappa\big(W,\sL\otimes\sO_W(G)\big)=\kappa(W,f^* \nu^* K_D+G)=\kappa(W,f^* \nu^* K_D)=\kappa(W,\sL)$$
since $G$ is effective and $f$-exceptional. Since $W$ is not uniruled, we have $\kappa(W,K_W)\geq0$, which together with \eqref{eq:656} gives
\begin{equation}\label{eq:656a}
\kappa(W,\sL)=\kappa\big(W,\sL\otimes\sO_W(G)\big)=\kappa(W,K_W+E)\geq\kappa(W,K_W)\geq0.
\end{equation}

Assume first that $\kappa(W,K_W)\geq1$. Then $\kappa(W,\sL)\geq1$ by \eqref{eq:656a}, and since $\kappa(W,\sL)\leq\nd(\sL)=1$, we conclude that $\kappa(W,\sL)=\nd(\sL)=1$. By~\cite[Proposition 2.14]{Nak87}, there exists a bimeromorphic morphism $\mu\colon Y\to W$, a fibration $g\colon Y\to Z$ to a projective variety $Z$, and a big and nef $\Q$-line bundle $\mathcal M$ on $Z$ such that
    \begin{equation}\label{eq:9000}
    \mu^*\sL\isom_\Q g^*\mathcal M.
	\end{equation}
	In particular, we have
    $$ \dim Z=\kappa(Z,\mathcal M)=\kappa(Y,g^*\mathcal M)=\kappa(Y,\mu^*\sL)=\kappa(W,\sL)=1, $$
    hence $Z$ is a projective curve. Consequently, $\mathcal M$ is ample, which implies that $\sL$ is semiample by \eqref{eq:9000}.

Finally, assume that $\kappa(W,K_W)=0$. By~\cite[Theorem~1.1]{HP16} and Theorem~\ref{thm:3fold-abundance}, there exists a $K_W$-MMP $ \theta\colon W \dashrightarrow W_{\min} $ which terminates with a minimal model $W_{\min}$ such that $K_{W_{\min}}\sim_\Q0$. 
Then by Lemma~\ref{lemma:invariance-degree-zero-MMP} and by the first paragraph of the proof, there exists a line bundle $\sL_{\min}$ on $W_{\min}$ with the following properties:
\begin{enumerate}[\normalfont (i)]
\item $\sL_{\min}$ is nef by Lemmas~\ref{lemma:invariance-degree-zero} and~\ref{lemma:invariance-degree-zero-MMP}~(b),
\item $\kappa(W_{\min},\sL_{\min})\geq0$ by Lemma~\ref{lemma:invariance-degree-zero-MMP}~(c) and by \eqref{eq:656a},
\item $\sL$ is semiample if and only if $\sL_{\min}$ is semiample by Lemma~\ref{lemma:invariance-degree-zero-MMP}~(d).
\end{enumerate}
Then the line bundle $\sL_{\min}$ is semiample by (i), (ii) and Lemma~\ref{lemma:dim3}, hence $\sL$ is semiample by (iii).
\end{proof}

\begin{lemma} \label{lemma:lagrangian}
Let $X$ be a hyperkähler fourfold. Assume that there exists a non-uniruled prime divisor $D$ on $X$, let $\holom{\nu}{D^\nu}{D}$ be its normalisation, and assume that $D^\nu$ is not Moishezon. Assume that $\nu^*K_D$ is semiample and that it defines the Iitaka fibration $\varphi\colon D^\nu\to \Gamma$ to a projective curve $\Gamma$. Let $F \subseteq D^\nu$ be a general fibre of $\varphi$. Then the following statements hold.
\begin{enumerate}[\normalfont (a)]
\item Let $ \holom{j}{F}{D^\nu}$ be the inclusion, and denote $\holom{i:=\nu \circ j}{F}{X}$. Then $i(F) \subseteq X$ is a Lagrangian surface.
\item The variety $F$ is a smooth projective $K$-trivial surface.
\item The morphism $\varphi$ does not admit a multisection.
\item By Proposition~\ref{normalisation-formula}, there exists an effective divisor $\Xi$ on $D^\nu$ such that $K_{D^\nu} \sim \nu^* K_D - \cond$. Then $\cond\cap F=\emptyset$.
\end{enumerate}
\end{lemma}

\begin{proof}
We prove the lemma in several steps.

\medskip

\emph{Step 1.}
We set $\sL:=\sO_X(D)$, and note that $\sL$ is nef with $q_X(\sL,\sL)=0$. In this step we show that there exists $\lambda>0$ such that
$$ i_*[F] = \lambda c_1(\sL)^2. $$

Indeed, note that $\sL|_D=\sO_D(K_D)$ by the adjunction formula. Since $\varphi$ is the Iitaka fibration associated to $\nu^*\sL|_D$ and since $\Gamma$ is a curve, both $\nu^*\sL|_D$ and $F$ are proportional to the pullback of a point on $\Gamma$. Hence, there exists $\lambda>0$ such that in $H^2(D^\nu, \R)$ we have $ j_*[F] = \lambda c_1(\nu^*\sL|_D) $. Therefore, for an arbitrary class $\eta \in H^4(X, \R)$ we obtain by the projection formula:
\begin{align*}
(i_* [F] \cdot \eta)_X &= (j_*[F] \cdot \nu^* \eta)_{D^\nu} = (\lambda \nu^* c_1(\sL|_D) \cdot \nu^* \eta)_{D^\nu}\\
& = (\lambda c_1(\sL) \cdot \eta \cdot D)_X = (\lambda c_1(\sL)^2 \cdot \eta)_X.
\end{align*}
We conclude by the Poincar\'e duality.

\medskip

\emph{Step 2.}
In this step we prove (a).

Let $\sigma$ be a holomorphic symplectic form on $X$. By~\cite[Corollary~5.2]{HLL25} we have
$$ \int_X  c_1(\sL)^2  \wedge \sigma \wedge \overline \sigma = 0. $$
Assume, for contradiction, that $i^* \sigma\neq0$. Since $F$ is a surface, the form $i^* \sigma \wedge i^* \overline \sigma$ is a volume form on $F$ away from a set of measure zero, and thus, by Step 1 we have
$$ \lambda\int_X  c_1(\sL)^2  \wedge \sigma \wedge \overline \sigma=\int_X i_*[F] \wedge \sigma \wedge \overline \sigma = \int_F i^* \sigma \wedge i^* \overline \sigma > 0. $$
But this contradicts Step 2.

\medskip

\emph{Step 3.}
In this step we show (b) and (c).

Since $i(F) \subseteq X$ is a Lagrangian surface by (a), we have that $i(F)$ is projective by~\cite[Theorem~1.1]{Leh16}. Since $i$ is finite onto its image, this shows that $F$ is projective. Thus, since $D^\nu$ is not Moishezon, we conclude by Corollary~\ref{corollary:no-multisection} that $\varphi$ does not admit a multisection, which proves (c). In particular, $D^\nu_{\sing}$ is not a multisection, hence $F$ avoids $D^\nu_{\sing}$. Thus, by generic smoothness, $F$ is smooth. Moreover, $F$ is $K$-trivial since it is a fibre of the Iitaka fibration $\varphi$. This shows (b).

\medskip

\emph{Step 5.}
Finally, in this step we show (d).

Since $F$ is a fibre of $\varphi$, we have $(\nu^* K_D)|_F\equiv0$, hence
\begin{equation} \label{eqnKF}
K_F \sim K_{D^\nu}|_F \sim (\nu^* K_D)|_F - \cond|_F \equiv {-} \cond|_F.
\end{equation}
Since $D$ is not uniruled, neither is $F$. Therefore, $K_F$ is pseudoeffective, hence ${-} \cond|_F$ is pseudoeffective by \eqref{eqnKF}. This is only possible if $\cond|_F=0$, which gives (d).
\end{proof}

\begin{proposition} \label{prop:non-uniruled}
Let $X$ be a hyperkähler fourfold. Assume that there exists a non-uniruled prime divisor $D$ on $X$. Assume that the canonical bundle $K_D$ is nef with $\nd(K_D)=1$, and assume that $K_D$ has degree zero on all curves in $D$. Then $K_D$ is semiample and $\kappa(D,K_D)=1$.

If $\holom{\psi}{D}{T}$ is the Iitaka fibration associated to $K_D$, then a general fibre of $\psi$ is abelian or bielliptic.
\end{proposition}

\begin{proof}
Let $\holom{\nu}{D^\nu}{D}$ be the normalisation of $D$. Then by Proposition~\ref{prop:normalisation0} the variety $D^\nu$ is not Moishezon, and the divisor $\nu^* K_D$ is semiample. Therefore, $\nu^* K_D$ defines the Iitaka fibration $\varphi\colon D^\nu\to \Gamma$ to a projective curve $\Gamma$, and let $F^\nu \subseteq D^\nu$ be a general fibre of $\varphi$. By Lemma~\ref{lemma:lagrangian}(d), $F^\nu$ is contained in the locus where $\nu$ is an isomorphism. Thus, by Proposition~\ref{prop:semiample-nonnormal}, the variety $F:=\nu(F^\nu)$ is a semiample divisor in $D$ with $\kappa(D,F)=1$.

Let $\holom{\psi}{D}{T}$ be the Iitaka fibration associated to $K_D$. 
By Lemma~\ref{lemma:lagrangian}(d) the non-normal locus is disjoint
from a general $\psi$-fibre, so by Lemma~\ref{lemma:lagrangian}(b)  the general fibre is a smooth projective $K$-trivial surface.

Let us now prove that the general $\phi$-fibre is abelian or bielliptic. By Proposition~\ref{prop:KS-trivial} it is sufficient to show that $F$ contains no rational curves.

Arguing by contradiction assume that $F$ contains a rational curve $C$. Then $C \subseteq D \subseteq X$ is a rational curve, and hence deforms in a family of dimension at least two in $X$~\cite[Corollary~4]{Ran93}. Choose an irreducible deformation family $\sK \subseteq \Rat{X}$ of $C$, and let $V \subseteq X$ be the closure of the locus covered by the deformations. Then $\dim V \geq 2$ and $V$ is uniruled.

If $\dim V=3$, it is a divisor and therefore $D=V$ by Lemma~\ref{lemma-unique-divisor}, a contradiction to our assumption. Thus, we have $\dim V=2$. Let $\overline{\sK}$ be the closure of the image of $\sK$ in the cycle space. Then every point of $\overline{\sK}$ parametrises a connected cycle of rational curves. Since  $\dim \overline{\sK}=2$ and $\dim V=2$, we obtain that for every point $v \in V$, the family $\overline{\sK}_v$ of cycles passing through $v$ has dimension at least one. 

The intersection $V \cap D$ is not empty by construction, let us show that $V \subseteq D$. Otherwise the intersection $V \cap D$ has dimension one, so for a point $v \in V \cap D$ not all the cycles parametrised by $\overline{\sK}_v$ are contained in $V \cap D$. Thus, we can find a curve in $\overline{\sK}_v$ that passes through a point of $V \cap D$ but is not contained in $D$. Thus, we have $D \cdot C > 0$, a contradiction to Corollary~\ref{cor:allzero}.

Thus, we have $V \subseteq D$. Since $C \subseteq F \subseteq D$ is a rational curve in a general fibre we know that $K_D \cdot C=0$. Since $V$ is connected by deformations of $C$ we deduce that $V$ is contracted by the Iitaka fibration on $D$, so we have $V=F$. Yet $V$ is uniruled and $F$ is not, a contradiction.
\end{proof}

%----------------------------------------------------
%----------------------------------------------------
\section{The case when \texorpdfstring{$D$}{D} is uniruled}\label{section:uniruled}
%----------------------------------------------------
%----------------------------------------------------

In this section we give a description of the divisor $D$ and its singularities along a certain non-algebraic surface.
We start with a structure statement for non-projective K\"ahler threefolds, then show how to apply this in the hyperk\"ahler setting.
We then move on to a number of technical statements that prepare the main arguments in Section~\ref{section:extending} and~\ref{section:descending}.

\begin{proposition} \label{prop:normalisation}
Let $D$ be a compact uniruled Kähler Gorenstein variety of dimension three, and let $\holom{\nu}{D^\nu}{D}$ be the normalisation of $D$. Assume that the canonical bundle $K_D$ is nef with $\nd(K_D)=1$, and assume that $K_D$ has degree zero on all curves on $D$. Let $\pi\colon D^\nu\dashrightarrow S$ be the MRC fibration of $D^\nu$ with $S$ smooth. Let $\holom{g}{D'}{D^\nu}$ be a resolution of indeterminacies of $\pi$ which does not alter the locus where $\pi$ is a morphism and such that $D'$ is normal, and let $\pi'\colon D'\to S$ be the induced morphism. Then the following statements hold.
\begin{enumerate}[\normalfont (a)]
\item The variety $D^\nu$ is not Moishezon.
\item The variety $S$ is a surface with $a(S)=1$.
\item There exists a nef line bundle $\sL_S$ on $S$, which has degree zero on all curves on $S$, such that $g^*\nu^*\sO_D(K_D)\isom(\pi')^*\sL_S$.
\item There exists a proper elliptic fibration $\alpha\colon S\to \Gamma$ onto a projective curve $\Gamma$ such that the composition $\varphi:=\alpha\circ\pi$ is a morphism, and there exists a Zariski open subset $\Curv_\circ \subseteq \Curv$ such that, if we set $D^\nu_\circ := \fibre{\varphi}{\Curv_\circ}$ and $S_\circ:=\alpha^{-1}(\Gamma_\circ)$, we have the proper fibration $ \holom{\pi_\circ}{D^\nu_\circ}{S_\circ}$ which is a $\PP^1$-bundle.
$$\xymatrix{
D' \ar[r]^g \ar[d]_{\pi'} & D^\nu  \ar@{-->}[dl]_\pi \ar[d]^\varphi \ar[r]^\nu & D \\
S \ar[r]_\alpha & \Gamma &}$$
\item The divisor $\nu^* K_D$ is cohomologically trivial over $\Gamma$.
\end{enumerate}
\end{proposition}

\begin{proof}
By the projection formula we have $ \nu^* K_D \cdot C=K_D \cdot \nu_* C=0 $ for every curve $C$ on $ D^\nu$. By Lemma~\ref{lemma:notmoishezon} we obtain (a).

By (a) and by Proposition~\ref{proposition:MRC}, the variety $S$ is a surface which is not Moishezon. In particular, $a(S)\leq1$.

We prove (b), (c) and (d) in several steps.

\medskip

\emph{Step 1.}
In this step we prove (c).

Let $f\colon W\to D'$ be a resolution of $D'$, and set $\sL:=f^*g^* \nu^* \sO_D(K_D)$. Note that that $\sL$ is a nef line bundle on $W$ with $\nd(\sL)=1$, and as in the first paragraph of the proof we can show that $\sL$ has degree zero on all curves on $W$.

A general fibre of $\pi$ is a rational curve, hence a general fibre of $\pi'\circ f$ is a rational curve. In particular, $K_W$ is not pseudoeffective over $S$. The morphism $\pi' \circ f$ is projective by~\cite[Theorem~1.1]{CH24}, so e.g. by~\cite[Theorem~1.2]{Fuj24} there exists a relative $K_W$-MMP $ \theta\colon W \dashrightarrow W' $ over $S$, which terminates with a Mori fibre space $\rho\colon W'\to S'$. Then $S'$ must be a surface since $S$ is a surface, and let $\sigma\colon S'\to S$ be the induced birational morphism. Let $(p,q)\colon Z\to W\times S'$ be a resolution of indeterminacies of $\rho\circ\theta$.
$$ \xymatrix{
Z \ar[rd]^q \ar@/^1pc/[rr]^p & W' \ar[d]_{\rho} & W \ar@{-->}[l]^\theta \ar[r]^f & D' \ar[r]^g \ar[d]_{\pi'} & D^\nu \ar[r]^\nu \ar@{-->}[dl]^\pi & D \\
& S' \ar[rr]^\sigma &  & S. & & } $$ 

Then by Lemma~\ref{lemma:invariance-degree-zero-MMP} there exists a line bundle $\sL'$ on $S'$ with the following properties:
\begin{enumerate}[\normalfont (i)]
\item $\sL'$ has degree zero on all curves in $S'$,
\item $p^*\sL\isom q^*\sL'$, and in particular, $\sL'$ is nef,
\item $\nd(\sL')=1$.
\end{enumerate}
The surface $S'$ has klt singularities by~\cite[Lemma 4.1]{HP15}, and $S$ is smooth. Thus, by Lemma~\ref{lemma:factor-birational}, the morphism $\sigma$ is a composition of steps in a $K_{S'}$-MMP. Therefore, by (i) and by Lemma~\ref{lemma:invariance-degree-zero-MMP} there exists a nef line bundle $\sL_S$ on $S$ such that
\begin{equation}\label{eq:8989}
\sL' \isom \sigma^* \sL_S,
\end{equation}
and $\sL_S$ has degree zero on all curves in $S$. In particular, (ii) yields that $\sL_S$ is nef by Lemma~\ref{lemma:invariance-degree-zero}. Since the diagram above is commutative, we have $p^* f^* (\pi')^*\sL_S = q^*\sigma^* \sL_S$. From \eqref{eq:8989} and (ii) we obtain 
$$q^*\sigma^* \sL_S\isom q^*\sL'\isom p^*\sL=p^*f^*g^* \nu^* \sO_D(K_D).$$
This, together with the previous sentence and the projection formula, implies part (c).

\emph{Step 2.}
In this step we show (b).

Assume, for contradiction, that $a(S) =0$. Let $\mu\colon S \to S_{\min}$ be the minimal model of $S$. Then $a(S_{\min})=0$, hence $K_{S_{\min}} \equiv 0$. By Lemma~\ref{lemma:invariance-degree-zero-MMP}
the line bundle $\sL_S$ descends to a nef line bundle $\sL_{\min}$ on $S_{\min}$ such that $\nd(\sL_{\min})=1$. Then by Proposition~\ref{proposition:generalised-abundance-surface}, the line bundle $\sL_{\min}$ is numerically equivalent to a semiample line bundle $\mathcal M$. In particular, we have $\nd(\mathcal M)=1$, hence $\kappa(S_{\min},\mathcal M)=1$, which is a contradiction since $a(S_{\min})=0$.

\medskip

\medskip

\emph{Step 3.}
In this step we prove (d).

We use the notation from Step 2. By Remark~\ref{remark:alg-reduction}, the algebraic reduction of $S$ is an elliptic fibration $\holom{\alpha}{S}{\Curv}$ onto a curve $\Curv$.

Denote $\xi:=\alpha\circ\pi_0$, and let $\varphi\colon D^\nu\dashrightarrow\Gamma$ be the induced map, so that we have the commutative diagram
$$\xymatrix{
D' \ar[r]^g \ar[d]_{\pi'} \ar[dr]_{\xi} & D^\nu  \ar@{-->}[d]^\varphi  \ar[r]^\nu & D \\
S \ar[r]_\alpha & \Gamma. &}$$
A general fibre of the morphism $\xi\colon D'\to\Gamma$ is a uniruled surface, hence it is a Moishezon surface. Since $D'$ is not Moishezon, it follows by Corollary~\ref{corollary:no-multisection} that $\xi$ does not have a multisection. Therefore, $\xi$ contracts every curve in~$D'$. The fibres of $g$ are connected by curves by Lemma~\ref{lemma:birational-curves}, and we have just seen that $\xi$ contracts every curve on $D'$. Thus, by the Rigidity lemma
\cite[Lemma 4.1.13]{BS95}
the map $\varphi$ is a morphism.

Let $\Delta \subseteq S$ be the locus parametrizing singular fibres of $\pi'$. Since $D^\nu$ is normal and the fibres of $\pi'$ are curves, the inclusion $\Delta \subseteq S$ is strict. Moreover, if $E \subseteq D'$ is the exceptional locus of $g$, then the inclusion $\pi'(E) \subseteq S$ is also strict by the construction of $D'$. The surface $S$ is not Moishezon, hence $\alpha$ does not have a multisection by Corollary~\ref{corollary:no-multisection}. In particular, the set $\alpha\big(\Delta\cup\pi'(E)\big)=\alpha(\Delta) \cup \xi(E)$ is finite. We set 
$$ \Curv_\circ := \Curv \setminus \big\{ \alpha(\Delta) \cup \xi(E) \big\},\quad D^\nu_\circ := \fibre{\varphi}{\Curv_\circ},\quad S_\circ:=\alpha^{-1}(\Gamma_\circ).$$
Then the restriction map $ \merom{\pi|_{D^\nu_\circ}}{D^\nu_\circ}{S_\circ}$ is holomorphic by construction, and it is a $\PP^1$-bundle since all its fibres are smooth. We set $\pi_\circ:=\pi|_{D^\nu_\circ}$. This finishes the proof of (d).

\medskip

\emph{Step 4.}
Finally, in this step we prove (e).

By (c) and (d), it suffices to show that $\sL_S$ is cohomologically trivial over $\Gamma$. To that end, by flatness it suffices to show that $\sL_S$ is cohomologically trivial on a general fibre of $\alpha$. But this holds since $\sL_S$ has degree zero on all curves in $S$ by (c).
\end{proof}

\begin{setup}\label{setup:section7}
Let $X$ be a hyperkähler fourfold. Assume that there exists a uniruled prime divisor $D$ on $X$, let $\holom{\nu_D}{D^\nu}{D}$ be its normalisation, and assume that $D^\nu$ is not Moishezon. Let $\pi\colon D^\nu\dashrightarrow S$ be the MRC fibration of $D^\nu$ with $S$ smooth. Assume that there exists a proper elliptic fibration $\alpha\colon S\to \Gamma$ onto a projective curve $\Gamma$ such that the composition $\varphi:=\alpha\circ\pi$ is a morphism, and there exists a Zariski open subset $\Curv_\circ \subseteq \Curv$ such that, if we set $D^\nu_\circ := \fibre{\varphi}{\Curv_\circ}$ and $S_\circ:=\alpha^{-1}(\Gamma_\circ)$, we have the proper fibration $ \holom{\pi_\circ}{D^\nu_\circ}{S_\circ}$ which is a $\PP^1$-bundle.
$$\xymatrix{
 & D^\nu  \ar@{-->}[dl]_\pi \ar[d]^\varphi \ar[r]^{\nu_D} & D \\
S \ar[r]_\alpha & \Gamma. &}$$
Let $F \subseteq D^\nu$ be a general fibre of $\varphi$. 
\end{setup}

\begin{lemma} \label{lemma:lagrangian2}
Assume Setup~\ref{setup:section7}. Then the following statements hold.
\begin{enumerate}[\normalfont (a)]
\item The variety $F$ is a smooth uniruled projective surface.
\item The morphism $\varphi$ does not admit a multisection.
\item The divisor ${-}K_F$ in $F$ is not big.
\end{enumerate}
\end{lemma}

\begin{proof}
We prove the result in several steps.

\medskip

\emph{Step 1.}
In this step we show (a) and (b).

Let $F'$ be a resolution of $F$. Then the surface $F$ is uniruled by the assumption of the lemma on the map $\pi_\circ$, hence so is $F'$. In particular, we have $h^0(F', K_{F'})=0$, hence $h^2(F',\sO_{F'})=0$ by Serre duality. Thus, $F'$ is projective by Kodaira's criterion, hence so is $F$.
Since $D^\nu$ is not Moishezon and $F$ is projective, we deduce (b) by Corollary~\ref{corollary:no-multisection}. In particular, $D^\nu_{\sing}$ is not a multisection, hence $F$ avoids $D^\nu_{\sing}$. Thus, by generic smoothness, $F$ is smooth, which shows (a).

\medskip

\emph{Step 2.}
Next we show (c). Assume that ${-}K_F$ is big. By the previous paragraph we have that $F$ avoids $D^\nu_{\sing}$, hence there exists a resolution of singularities $\holom{\mu}{\widehat D}{D^\nu}$ which is an isomorphism in a neighbourhood of $F$. Thus, the divisor ${-}K_{D'}$ is relatively big over $\Gamma$. Thus, $\varphi \circ \mu$ is a Moishezon morphism (cf. Definition~\ref{defn:moishezon-morphism}), hence $D'$ is Moishezon (cf. Remark {remarks:moishezon-morphism}), a contradiction. 
This shows (c).
\end{proof}

\begin{lemma}\label{lemma:lagrangian3}
Assume Setup~\ref{setup:section7}. Assume, furthermore, that $\nu^*K_D$ is cohomologically trivial over $\Gamma$. By Proposition~\ref{normalisation-formula}, there exists an effective divisor $\Xi$ on $D^\nu$ such that $K_{D^\nu} \sim \nu_D^* K_D - \cond$, and let $\cond^h$ denote the horizontal part of $\cond$ with respect to $\varphi$. Let $ \nonn := \nu_D(\cond^h)$ be the set-theoretical image of $\cond^h$,  this is a not necessarily irreducible surface in $D$. Then after shrinking $\Gamma_\circ$ the following statements hold.
\begin{enumerate}[\normalfont (a)]
\item One of the following holds.
\begin{enumerate}[\normalfont (i)]
\item The divisor $\cond^h$ is reduced and irreducible, and the map $\cond^h|_{D^\nu_\circ} \rightarrow S_\circ$ is \'etale of degree two.
\item We have $\cond^h=\cond_1+\cond_2$, where $\cond_1|_{D^\nu_\circ}$ and $\cond_2|_{D^\nu_\circ}$ are disjoint sections of $\pi_\circ$.
\item We have $\cond^h = 2 \cond_1$, where the $\cond_1|_{D^\nu_\circ}$ is a section of $\pi_\circ$.
\end{enumerate}
\item If $K_D$ is nef with $\nd(K_D)=1$, and if $\cond^h$ is reduced, then $\nu_D^*K_D$ is semiample.
\item Assume that $\cond^h$ is reduced and irreducible as in part (a)(i) or that $\cond^h=2\Xi_1$ as in part (a)(iii). Assume that $a(S)=1$. Then there exists an elliptic fibration 
$$ 
\holom{\alpha_{\Theta}}{\Theta}{\Curv_{\Theta}} 
$$ such that each irreducible analytic subset of $\Theta$ of dimension at most $1$ is contained in a fibre of $\alpha_{\Theta}$. Moreover, there exists a Zariski open subset $\Theta_\circ \subseteq \Theta$ such that $\Theta_\circ$ is smooth, 
$$
\alpha_{\Theta}|_{\Theta_\circ}\colon \Theta_\circ\to\Gamma_{\Theta,\circ}
$$ 
is a smooth elliptic fibration, and $\Theta_\circ$ is disjoint from the image by $\nu$ of any component of $\Xi$ which is vertical over $\Gamma$.
\item Assume that $\cond^h=\Xi_1+\Xi_2$ as in part (a)(ii), and that $a(S)=1$. Set $ \nonn_1 := \nu(\cond_1)$ and $ \nonn_2 := \nu(\cond_2)$. Then for each $i\in\{1,2\}$ there exists an elliptic fibration $ \holom{\alpha_{\Theta_i}}{\Theta_i}{\Curv_{\Theta_i}} $ such that each irreducible analytic subset of $\Theta_i$ of dimension at most $1$ is contained in a fibre of $\alpha_{\Theta_i}$. Moreover, for each $i\in\{1,2\}$ there exists a Zariski open subset $\Theta_{i,\circ} \subseteq \Theta_i$ such that each $\Theta_{i,\circ}$ is smooth, each 
$$
\alpha_{\Theta_i}|_{\Theta_{i,\circ}}\colon \Theta_{i,\circ}\to\Gamma_{\Theta_i,\circ}
$$ is a smooth elliptic fibration.
Moreover, if $\Theta_1 \neq \Theta_2$, and $\Theta_{1,\circ}\cap\Theta_{2,\circ}=\emptyset$.
\end{enumerate}
\end{lemma}
\begin{proof}
We prove the result in several steps.

\medskip

\emph{Step 1.}
Now we show (a). Since $\nu^* K_D$ is cohomologically trivial over $\Gamma$ we have $(\nu^* K_D)|_F\equiv0$. Therefore,
\begin{equation}\label{eq:767}
K_F \sim K_{D^\nu}|_F \sim (\nu^* K_D)|_F - \cond|_F \equiv {-} \cond|_F.
\end{equation}
Since ${-}K_F$ has degree two on a fibre of $\pi|_F$, \eqref{eq:767} shows that $\cond|_F$ has degree two on a fibre of $\pi|_F$. 
Note that all the components of $\cond^h$ are $\pi$-horizontal: the image of a $\pi$-vertical component
in $S$ is a multisection of the algebraic reduction $\alpha: S \rightarrow \Gamma$, by
Corollary~\ref{corollary:no-multisection} such a multisection does not exist.
In particular, $\cond^h$ is not empty, and it has at most two irreducible components.

If $\cond^h$ is irreducible but not reduced, then we deduce (iii). If $\cond^h$ is irreducible and reduced, then the map $\cond^h|_{D^\nu_\circ} \rightarrow S_\circ$ is \'etale after shrinking $\Gamma_\circ$, since otherwise its ramification divisor would define a multisection of $\varphi$, which would contradict Lemma~\ref{lemma:lagrangian2}(b); this gives (i). If $\cond^h$ is reducible, then after shrinking $\Gamma_\circ$, its components $\cond_1$ and $\cond_2$ are disjoint, since otherwise the intersection $\cond_1\cap\cond_2$ would define a multisection of $\varphi$, which would contradict Lemma~\ref{lemma:lagrangian2}(b); this gives (ii).

\medskip

\emph{Step 2.}
In this step we show (b). By (a), the divisor $\cond|_F$ on $F$ is smooth, hence the pair $(F, \cond|_F)$ is log-canonical. By \eqref{eq:767} we have
$$ K_F + \cond|_F\sim(\nu^* K_D)|_F\equiv0, $$
hence $K_F + \cond|_F\sim_\Q0$ by Abundance on surfaces. Therefore, $(\nu^* K_D)|_F\sim_\Q0$. We then conclude by Lemma~\ref{lemma:zariski} that $\nu^* K_D \sim_\Q \varphi^* A$ for a $\Q$-divisor $A$ on $\Curv$. Then $A$ is nef and $\nd(A)=1$, hence $A$ is ample. Therefore, $\nu^*K_D$ is semiample, as desired.

\medskip

\emph{Step 3.}
In this step we show (c). In the case of (a)(i), the map $\cond^h|_{D^\nu_\circ} \rightarrow S_\circ$ is \'etale of degree two, and in the case of (a)(iii), $\cond_1|_{D^\nu_\circ}$ is isomorphic to $S_\circ$. Since the algebraic dimension is invariant under generically finite meromorphic dominant maps
(cf. Lemma~\ref{lemma:alg-dimension}), we obtain that $a(\cond^{h})=1$. By the same argument, $a(\nonn)=1$. We set $\alpha_\Theta$ to be the algebraic reduction of $\Theta$, cf. Corollary~\ref{corollary:alg-dim-one}. Then each irreducible analytic subset of $\Theta$ of dimension at most $1$ is contained in a fibre of $\alpha_{\Theta}$ by Corollary~\ref{corollary:alg-dim-one}. In particular, the locus along which $\alpha_\Theta$ is not a smooth fibration is contained in finitely many fibres of $\alpha_\Theta$, which shows the existence of $\Theta_\circ$ satisfying the first two conditions from the last statement of (c).

Now we prove the very last statement of (c). Let $S$ be a component of $\cond$ which is vertical over $\Gamma$. Then $S$ is contracted by $\varphi$ onto a point. Thus $S$ is uniruled surface by Lemma~\ref{lemma:degenerate-uniruled}, so $\nu(S) \subseteq X$ is Lagrangian and therefore projective by~\cite[Theorem~1.1]{Leh16}. Since $a(\nonn)=1$ by the previous paragraph, and $S$ is projective, it cannot be dominant over $\nonn$, hence $\nu(S) \cap \nonn \neq \nonn$. By the previous paragraph, this implies that $\nu(S) \cap \nonn $ is contained in finitely many fibres of $\alpha_\nonn$, which implies that the last sentence of (c), after shrinking $\Theta_\circ$.

\medskip

\emph{Step 4.}
In this step we show (d). By (a)(ii), $\cond_1|_{D^\nu_\circ}$ and $\cond_2|_{D^\nu_\circ}$ are isomorphic to $S_\circ$. We conclude similarly as in Step 3; note that the final statement of (d) follows if we choose choose $\Theta_{1,\circ}$ and $\Theta_{1,\circ}$ so that they avoid the analytic subset $\Theta_1\cap\Theta_2$.
\end{proof}

We will now give a precise description of the singularities of the divisor $D$
along $\Theta$, this will be a key ingredient for the extension 
argument in Proposition~\ref{prop:projective-extension}.
We introduce some terminology:

\begin{definition} \label{defn:kodaira-type}
Let $X_\circ$ be a manifold of dimension $n$, and let $D_\circ \subset X_\circ$ be a reduced divisor.
Assume that $S_\circ := D_{\sing}$ is a submanifold of dimension $n-2$. 
Let $\holom{\mu_{\circ}}{Z_{\circ}}{X_\circ}$ be the blowup of $X_\circ$ along $S_\circ$,
let $D_{Z, \circ} \subset Z_{\circ}$ be the strict transform of $D_\circ$, and let $E_{\circ}$ be the exceptional divisor. 
We say that $D_\circ$ has singularities
of Kodaira type $IV$ (resp. $III$, resp. $II$) along $S_\circ$ if the following holds:

\begin{itemize}
\item Type IV: the divisor $D_{Z, \circ}+E_{\circ} \subset Z_{\circ}$ has snc singularities and the natural map
$$
D_{Z, \circ} \cap E_{\circ} \rightarrow S_\circ
$$
is an \'etale cover of degree three.
\item Type III:  the set-theoretical intersection $S_{Z, \circ} := (D_{Z, \circ} \cap E_{\circ})_{red}$ is
a submanifold of dimension $n-2$ in $Z_{\circ}$ such that
the natural map
$$
S_{Z, \circ} \rightarrow S_\circ
$$
is an isomorphism; and
$D_{Z, \circ}+E_{\circ}$ has singularities
of Kodaira type $IV$ along $S_{Z,\circ}$. 
\item Type II: the set-theoretical intersection $S_{Z, \circ} := (D_{Z, \circ} \cap E_{\circ})_{red}$ is
a submanifold of dimension $n-2$ in $Z_{\circ}$ such that
the natural map
$$
S_{Z, \circ} \rightarrow S_\circ
$$
is an isomorphism; and
$D_{Z, \circ}+E_{\circ}$ has singularities
of Kodaira type $III$ along $S_{Z,\circ}$. 

\end{itemize}
\end{definition}

The definition of singularities in terms of blow-ups is convenient for the proof
of Proposition~\ref{prop:lagrangian4} and Proposition~\ref{prop:projective-extension},
it is not very difficult to see that the definitions are equivalent to the following
local description:

\begin{definition} \label{definition:kodaira-type-geometric}
Let $X_\circ$ be a complex manifold of dimension $n$, and let $D_\circ \subseteq X_\circ$ be a reduced divisor. We say that $D_\circ$ has singularities of Kodaira type $II, III$, or $IV$ in a point $x \in D_\circ$ if there exist analytic coordinates $x_1, \ldots, x_n$  near $x$ such that 
\begin{itemize}
\item $D_\circ= \{ x_1^2-x_2^3=0 \}$ (Kodaira type II); resp.
\item $D_\circ= \{ x_2 \cdot (x_2-x_1^2) = 0 \}$ (Kodaira type III); resp.
\item $D_\circ= \{ x_1 \cdot x_2 \cdot (x_1 - x_2) = 0 \}$ (Kodaira type IV).
\end{itemize}
\end{definition}

\begin{remark} \label{remark:kodaira-typeI}
In the situation above we say that $D_\circ$ has Kodaira type $I$ if the local equation is 
$$
D_\circ= \{ x_1 x_2=0 \}.
$$ 
These singularities are more commonly called normal crossing points of multiplicity two, we will use both terminologies.
\end{remark}

\begin{proposition}\label{prop:lagrangian4}
Let $X$ be a hyperkähler fourfold. Assume that there exists a uniruled prime divisor $D$ on $X$ and let $\holom{\nu_D}{D^\nu}{D}$ be its normalisation. By Proposition~\ref{normalisation-formula}, there exists an effective divisor $\Xi$ on $D^\nu$ such that
\begin{equation}\label{eq:normalisation0}
K_{D^\nu} \sim \nu_D^* K_D - \cond.
\end{equation}
Assume that there exists an analytic surface $\Theta\subseteq D$ with $a(\Theta)=1$, and a smooth open subset $\Theta_\circ\subseteq\Theta$ such that, if we denote
$$X_\circ := X \setminus (D_{\sing}\setminus \Theta_\circ), \quad D_\circ := D \setminus (D_{\sing}\setminus \Theta_\circ), \quad \Xi_\circ:=\Xi|_{\nu_D^{-1}(D_\circ)},$$
then $\Supp(\cond_\circ)=\nu_D^{-1}(\Theta_\circ)$ and $\cond_\circ=2\Phi_\circ$, where $\Phi_\circ$ is a reduced and irreducible divisor on $D_\circ$. Then, after shrinking $\Theta_\circ$, there exist a bimeromorphic morphism $\mu\colon Z\to X$, a positive rational number $p$, and an effective $\Q$-divisor $\Delta$ on $Z$ such that, if we denote $Z_\circ:=\mu^{-1}(X_\circ)$ and $\Delta_\circ:=\Delta|_{Z_\circ}$, then the pair $(Z_\circ,\Delta_\circ)$ is a log smooth plt pair and
$$K_{Z_\circ}+\Delta_\circ\sim_\Q \mu^*(K_X+pD_\circ).$$
Moreover the divisor $D$ has singularities of Kodaira type $II, III$, or $IV$ along~$\Theta_\circ$.
\end{proposition}

\begin{proof}
We prove the result in several steps. Let $\alpha_\Theta\colon\Theta\to\Gamma_\Theta$ be the algebraic reduction of $\Theta$. Throughout the proof, for a variety $Y$ we denote by $\nu_Y$ its normalisation.

\medskip

\emph{Step 1.}
In this step we prepare several statements that we will need in the remainder of the proof.

Let $\holom{\mu'}{X'}{X}$ be the blowup along $\nonn$, denote by $E'$ the unique exceptional divisor that maps onto $\nonn$, and let $D' \subseteq X'$ be the strict transform of $D$. Set $X_\circ':=(\mu')^{-1}(X_\circ)$, $\mu_\circ':=\mu'|_{X'_\circ}$, $E_\circ':=E'|_{X'_\circ}$ and $D_\circ':=D'|_{X_\circ'}$. Then $\holom{\mu_\circ'}{X_\circ'}{X_\circ}$ is the blowup of $X_\circ$ along $\nonn_\circ$, hence 
\begin{equation}\label{eq:Esmooth}
E_\circ'\text{ is a smooth irreducible exceptional divisor of }\mu_\circ'.
\end{equation}
By construction of $D_\circ$ we have that
\begin{equation}\label{eq:Dsmoothaway}
D_\circ'\text{ is smooth away from }D_\circ'\cap E_\circ'.
\end{equation}

We consider $D' \cap E'$ with the reduced analytic scheme structure. Since $D' \cap E'$ is a reduced surface, the induced morphism $\mu'|_{D' \cap E'}\colon D' \cap E'\to\Theta$ is generically finite, hence the locus $F\subseteq\Theta$ over which $\mu'|_{D' \cap E'}$ is not a finite \'etale morphism is an analytic subset of $\nonn$ of dimension at most $1$. Thus, by Corollary~\ref{corollary:alg-dim-one} we have that $F$ is contained in finitely many fibres of $\alpha_\nonn$. Therefore, after shrinking $\Theta_\circ$, we may assume that the morphism $\mu_\circ'|_{D_\circ'\cap E_\circ'}$ is finite and \'etale. In particular, we may assume that
\begin{equation}\label{eq:thetasmooth}
D_\circ'\cap E_\circ'\text{ is smooth.}
\end{equation}
Moreover, since $\mu_\circ'|_{D_\circ'\cap E_\circ'}$ is finite,
\begin{equation}\label{eq:mufinite}
\text{the morphism } \holom{\mu_\circ'|_{D_\circ'}}{D_\circ'}{D_\circ} \text{ is finite.}
\end{equation}
By the universal property of normalisations we have a commutative diagram
$$\xymatrix{
(D_\circ')^\nu \ar[r] \ar[d]_{\nu_{D_\circ'}} & (D_\circ)^\nu  \ar[d]^{\nu_{D_\circ}}  \\
D_\circ' \ar[r]_{\mu_\circ'|_{D_\circ'}} & D_\circ.}$$
Then the morphism $(D_\circ')^\nu \rightarrow D_\circ^{\nu}$ that is finite and birational by \eqref{eq:mufinite}, hence an isomorphism by Zariski's main theorem. Thus, we may identify $(D_\circ')^\nu $ and $D_\circ^{\nu}$ via this isomorphism, and we have the factorisation
\begin{equation}\label{eq:factorisation}
\nu_{D_\circ} = \mu_\circ'|_{D_\circ'} \circ \nu_{D_\circ'} : (D_\circ')^\nu \isom D_\circ^{\nu}
\rightarrow D_\circ.
\end{equation}
Furthermore, by assumption and by \eqref{eq:factorisation} we have that
\begin{equation}\label{eq:image1}
\Phi_\circ=(\nu_{D_\circ})^{-1}(\Theta_\circ)= (\nu_{D_\circ'})^{-1}\big((\mu_\circ'|_{D_\circ'})^{-1} (\Theta_\circ)\big) = (\nu_{D_\circ'})^{-1}(D_\circ' \cap E_\circ').
\end{equation}
Since $\Phi_\circ$ is irreducible, we conclude that
\begin{equation}\label{eq:irreducible1}
D_\circ' \cap E_\circ'\text{ is irreducible.}
\end{equation}

\medskip

\emph{Step 2.}
By Proposition~\ref{normalisation-formula} and by \eqref{eq:factorisation}, there exists an effective divisor $\Xi_\circ'$ on $D_\circ^\nu$ such that
\begin{equation}\label{eq:normalisation'}
K_{D_\circ^\nu} \sim \nu_{D_\circ'}^* K_{D_\circ'} - \Xi_\circ'.
\end{equation}
In this step we show that one of the following holds:
\begin{enumerate}[\normalfont (i)]
\item $\mult_{\nonn_\circ} D_\circ=3$, $\mult_{\Phi_\circ}\Xi_\circ'=0$, and $\mult_{\Phi_\circ} \nu_{D_\circ'}^* (E_\circ'|_{D_\circ' }) =1$,
\item $\mult_{\nonn_\circ} D_\circ=2$, $\mult_{\Phi_\circ}\Xi_\circ'=1$, and $\mult_{\Phi_\circ} \nu_{D_\circ'}^* (E_\circ'|_{D_\circ' }) =1$,
\item $\mult_{\nonn_\circ} D_\circ=2$, $\mult_{\Phi_\circ}\Xi_\circ'=0$, and $\mult_{\Phi_\circ} \nu_{D_\circ'}^* (E_\circ'|_{D_\circ' }) =2$.
\end{enumerate}

To that end, set $m:=\mult_{\nonn_\circ} D_\circ$. Then
\begin{equation}\label{eq:twoequalities}
(\mu_\circ')^* D_\circ \sim  D_\circ' + m E_\circ'\quad\text{and}\quad K_{X_\circ'} \sim (\mu_\circ')^* K_{X_\circ} + E_\circ',
\end{equation}
so combining these with the adjunction formula we obtain 
$$ K_{D_\circ'} \sim (\mu_\circ'|_{D_\circ'})^* K_{D_\circ} + (1-m)(E_\circ'|_{D_\circ'}). $$
From this equation, from \eqref{eq:factorisation}, and from \eqref{eq:normalisation'} we obtain
$$ K_{D_\circ^\nu} \sim (\nu_{D_\circ})^* K_{D_\circ} - \cond_\circ' + (1-m) \nu_{D_\circ'}^* (E_\circ'|_{D_\circ'}), $$
which together with \eqref{eq:normalisation0} yields
\begin{equation}\label{eq:restrictionequation}
\Xi|_{D_\circ^{\nu}}=\cond_\circ' + (m-1) \nu_{D_\circ'}^* (E_\circ'|_{D_\circ'}).
\end{equation}
Since $\nu_{D_\circ}^{-1}(\Theta_\circ)=\Supp(\Xi_\circ)$ by assumption, the variety $D_\circ$ is not normal at a general point of $\Theta_\circ$, hence Lemma~\ref{lemma:Nagata} gives $m\geq2$. Moreover, by \eqref{eq:image1} we have $\mult_{\Phi_\circ} \nu_{D_\circ'}^* (E_\circ'|_{D_\circ' }) \geq1$. Then, by comparing multiplicities along $\Phi_\circ$ of both sides of the equation \eqref{eq:restrictionequation}, we obtain the claim.

We analyse the possibilities (i), (ii) and (iii) above in the next three steps.

\medskip

\emph{Step 3.}
In this step we prove the proposition when we are in case (i) of Step 2.

Since $\mult_{\Phi_\circ}\Xi_\circ'=0$ by case (i) of Step 2, we have that $\Phi_\circ$ is not a component of $\Xi_\circ'$, thus $D_\circ'$ is normal at a general point of $\nu_{D_\circ'}(\Phi_\circ)$. By Serre's criterion and by \eqref{eq:image1}, this implies that $D_\circ'$ is smooth at a general point of $D_\circ' \cap E_\circ'$, and $\nu_{D_\circ'}$ is an isomorphism at a general point of $\Xi_\circ$. Hence, the closure of the locus $G\subseteq D_\circ' \cap E_\circ'$ along which $D_\circ'$ is not smooth is an analytic subset of $D'\cap E'$ of dimension at most $1$. Thus, $\nu_D(G)\subseteq\Theta$ is an analytic subset of $\Theta$ of dimension at most $1$. By Corollary~\ref{corollary:alg-dim-one} we have that $\nu_D(G)$ is contained in finitely many fibres of $\alpha_\nonn$. Therefore, after shrinking $\Theta_\circ$, we may assume that $D_\circ'$ is smooth along $D_\circ'\cap E_\circ'$. This, together with \eqref{eq:Dsmoothaway}, yields that $D_\circ'$ is smooth.

In particular, $\nu_{D_\circ'}$ is an isomorphism. Since $\mult_{\Phi_\circ} \nu_{D_\circ'}^* (E_\circ'|_{D_\circ' }) =1$ by case (i) of Step 2, we conclude that $\mult_{D_\circ'\cap E_\circ'} (E_\circ'|_{D'_\circ}) =1$. Therefore, the scheme-theoretic intersection of $D_\circ'$ and $E_\circ$ is transversal at a general point of $D_\circ'\cap E_\circ'$ by Lemma~\ref{lem:intersectionmult}. By applying Corollary~\ref{corollary:alg-dim-one} as in the previous paragraph, we conclude that, after shrinking $\Theta_\circ$, we may assume that the scheme-theoretic intersection of $D_\circ'$ and $E_\circ'$ is transversal along $D_\circ'\cap E_\circ'$. Thus, the divisor $D_\circ'+E_\circ'$ has simple normal crossings support.

As $\mult_{\nonn_\circ} D_\circ=3$ by case (i) of Step 2, from \eqref{eq:twoequalities} we obtain
$$ \textstyle K_{X_\circ'} + \frac{2}{3} D_\circ' + E_\circ' \sim_\Q \mu_\circ^* \big(K_{X_\circ}+\frac{2}{3} D_\circ\big). $$
Now we set $Z:=X'$, $\mu:=\mu'$, $p:=\frac23$ and $\Delta:=\frac{2}{3} D' + E'$.

\medskip

\emph{Step 4.}
In Steps 4 and 5 we prove the proposition when we are in case (ii) of Step 2.

Let $\holom{\mu''}{X''}{X'}$ be the blowup along $D' \cap E'$, let $F''$ the unique exceptional divisor that dominates $D_\circ' \cap E'_\circ$, let $D'' \subseteq X''$ be the strict transform of $D$ on $X''$, and let $E'' \subseteq X''$ be the strict transform of $E'$ on $X''$. Set $X_\circ'':=(\mu'')^{-1}(X_\circ')$, $\mu_\circ'':=\mu''|_{X_\circ''}$, $D_\circ'':=D''|_{X_\circ''}$, $E_\circ'':=E''|_{X''_\circ}$ and $F_\circ'':=F''|_{X''_\circ}$. Then $\holom{\mu_\circ''}{X_\circ''}{X_\circ'}$ is the blowup of $X_\circ'$ along $D_\circ' \cap E'_\circ$, hence by \eqref{eq:Esmooth} and \eqref{eq:thetasmooth} we conclude that
\begin{equation}\label{eq:Esmooth'}
E_\circ''+F_\circ''\text{ has simple normal crossings support.}
\end{equation}
By \eqref{eq:Dsmoothaway} we have that
\begin{equation}\label{eq:Dsmoothaway'}
D_\circ''\text{ is smooth away from }D_\circ''\cap F_\circ''.
\end{equation}
As in Step 2 we deduce that, after shrinking $\Theta_\circ$, we may assume that the morphism $\mu_\circ''|_{D_\circ''\cap F_\circ''}$ is finite and \'etale. In particular, we may assume that
\begin{equation}\label{eq:thetasmooth'}
D_\circ''\cap F_\circ''\text{ is smooth.}
\end{equation}
As in Step 1 we deduce that the induced morphism $(D_\circ'')^\nu \rightarrow (D_\circ')^{\nu}$ is an isomorphism, and we have the factorisation
\begin{equation}\label{eq:factorisation'}
\nu_{D_\circ'} = \mu_\circ''|_{D_\circ''} \circ \nu_{D_\circ''}.
\end{equation}
Furthermore, by \eqref{eq:image1} and by \eqref{eq:factorisation'} we have that
\begin{align}
\Phi_\circ&=(\nu_{D_\circ'})^{-1}(D_\circ' \cap E_\circ') \label{eq:image1'}\\
&= (\nu_{D_\circ''})^{-1}\big((\mu_\circ''|_{D_\circ''})^{-1} (D_\circ' \cap E_\circ')\big) = (\nu_{D_\circ''})^{-1}(D_\circ'' \cap F_\circ'').\notag
\end{align}
Since $\Phi_\circ$ is irreducible, we conclude that
\begin{equation}\label{eq:irreducible1'}
D_\circ'' \cap F_\circ''\text{ is irreducible.}
\end{equation}

We claim that the divisor $D_\circ''+E_\circ''+ F_\circ''$ has simple normal crossings support, and that $\mult_{D_\circ' \cap E_\circ'} D_\circ'=2$. We will prove this claim in Step 5. Assuming the claim, we complete the proof of the proposition when we are in case (ii) of Step 2. As $\mult_{\nonn_\circ} D_\circ=2$ by case (ii) of Step 2, and as $\mult_{D_\circ' \cap E_\circ'} D_\circ'=2$ by the claim, by \eqref{eq:Esmooth}, \eqref{eq:thetasmooth} and \eqref{eq:twoequalities} we obtain
$$ K_{X_\circ''} \sim (\mu_\circ'')^* K_{X_\circ'} + F_\circ'' \sim (\mu_\circ'')^* (\mu_\circ')^* K_{X_\circ} + E''_\circ + 2 F_\circ'' $$
and
$$ (\mu_\circ'')^* (\mu_\circ')^* D_\circ = (\mu_\circ'')^* (D_\circ' + 2 E'_\circ) = D_\circ'' + 2 E''_\circ + 4 F_\circ''. $$
Therefore
$$ \textstyle K_{X_\circ''} + \frac{3}{4} D_\circ'' +  \frac{1}{2} E''_\circ + F_\circ'' \sim_\Q (\mu_\circ'')^* (\mu_\circ')^* (K_{X_\circ}+\frac{3}{4} D_\circ). $$
Now we set $Z:=X''$, $\mu:=\mu'\circ\mu''$, $p:=\frac34$ and $\Delta:=\frac34 D'' + \frac12 E'' + F''$.

\medskip

\emph{Step 5.}
In this step we prove the claim from Step 4. By \eqref{eq:Esmooth'}, it suffices to show that $D_\circ''+F_\circ''$ has simple normal crossings support, and that $E_\circ''\cap D_\circ''=\emptyset$. We show these two facts in this step.

By Proposition~\ref{normalisation-formula}, there exists an effective divisor $\Xi_\circ''$ on $D_\circ^\nu$ such that $ K_{D_\circ^\nu} \sim (\nu_{D_\circ''})^* K_{D_\circ''} - \Xi_\circ''$. Set $k:=\mult_{D_\circ' \cap E_\circ'} D_\circ'$. Then as in Step 2, by using \eqref{eq:factorisation'}, we obtain
\begin{equation}\label{eq:restrictionequation'}
\Xi'_\circ=\cond_\circ'' + (k-1) (\nu_{D_\circ''})^* F_\circ''|_{D_\circ''}.
\end{equation}
We have $(\nu_{D_\circ'})^{-1}(D_\circ' \cap E_\circ')=\Phi_\circ\subseteq\Xi_\circ'$ by \eqref{eq:image1} and since $\mult_{\Phi_\circ}\Xi_\circ'=1$ by case (ii) of Step 2. Thus, the variety $D_\circ'$ is not normal at a general point of $D_\circ' \cap E_\circ'$, hence Lemma~\ref{lemma:Nagata} gives $\geq2$. Moreover, by \eqref{eq:image1'} we have $\mult_{\Phi_\circ} (\nu_{D_\circ''})^* (F_\circ''|_{D_\circ''}) \geq1$. Then, by comparing multiplicities along $\Phi_\circ$ of both sides of the equation \eqref{eq:restrictionequation'}, we obtain
\begin{equation}\label{eq:threeequations}
\mult_{D_\circ' \cap E_\circ'} D_\circ'=2, \quad \mult_{\Phi_\circ}\Xi_\circ''=0, \quad \mult_{\Phi_\circ} (\nu_{D_\circ''})^* (F_\circ''|_{D_\circ'' }) =1.
\end{equation}

Now, since $\mult_{\Phi_\circ}\Xi_\circ''=0$ by \eqref{eq:threeequations}, as in Step 3 we conclude, after shrinking $\Theta_\circ$, that $D_\circ''$ is smooth along $D_\circ''\cap F_\circ''$, which, together with \eqref{eq:Dsmoothaway'}, yields that $D_\circ''$ is smooth. In particular, $\nu_{D_\circ''}$ is an isomorphism, and since we have $\mult_{\Phi_\circ} (\nu_{D_\circ''})^* (F_\circ'|_{D_\circ''}) =1$ by \eqref{eq:threeequations}, we conclude by \eqref{eq:image1'} that 
\begin{equation}\label{eq:multiplicityagain}
\mult_{D_\circ''\cap F_\circ''} (F_\circ''|_{D''_\circ}) =1.
\end{equation}
Therefore, as in Step 3 we  conclude, after shrinking $\Theta_\circ$, that the divisor $D_\circ''+F_\circ''$ has simple normal crossings support. It remains to show that $E_\circ''\cap D_\circ''=\emptyset$.

Since $\nu_{D_\circ''}$ is an isomorphism, \eqref{eq:factorisation'} becomes $\nu_{D_\circ'} = \mu_\circ''|_{D_\circ''}$. Therefore, by restricting the equation $(\mu_\circ'')^*E_\circ'=E_\circ''+F_\circ''$ to $D_\circ''$, we obtain
\begin{equation}\label{eq:restrictioneq}
(\nu_{D_\circ'})^*(E'_\circ|_{D_\circ'})=E_\circ''|_{D_\circ''}+F_\circ''|_{D_\circ''}.
\end{equation}
By \eqref{eq:image1'} and since we may assume that $\nu_{D_\circ''}$ is the identity by the previous paragraph, we have that the divisor $(\nu_{D_\circ'})^*(E'_\circ|_{D_\circ'})$ on $D_\circ''$ is supported on the irreducible set $(\nu_{D_\circ'})^{-1}(D_\circ' \cap E_\circ')=D_\circ'' \cap F_\circ''$. As the divisor $F_\circ''|_{D_\circ''}$ on $D_\circ''$ is clearly supported on $D_\circ'' \cap F_\circ''$, by \eqref{eq:restrictioneq} we conclude that
\begin{equation}\label{eq:supportinDF}
\Supp(E_\circ''|_{D_\circ''})\subseteq D_\circ'' \cap F_\circ''.
\end{equation}
Moreover, we have $\mult_{D_\circ''\cap F_\circ''} (\nu_{D_\circ'})^* (E_\circ'|_{D'_\circ}) =1$ by \eqref{eq:image1'} and by case (ii) of Step 2, and we have $\mult_{D_\circ''\cap F_\circ''} (F_\circ''|_{D_\circ'' }) =1$ by \eqref{eq:multiplicityagain}. Thus, from \eqref{eq:restrictioneq} we conclude that $\mult_{D_\circ''\cap F_\circ''}(E_\circ''|_{D_\circ''})=0$, which together with \eqref{eq:supportinDF} gives $E_\circ''\cap D_\circ''=\emptyset$, as desired.

\medskip

\emph{Step 6.}
In this step we prove the proposition when we are in case (iii) of Step 2.

As in Step 4, let $\holom{\mu''}{X''}{X'}$ be the blowup along $D' \cap E'$. Then we use the same notation as in Step 4 for $D''$, $E''$, $X_\circ''$, $\mu_\circ''$, $D_\circ''$, $E_\circ''$ and $F_\circ''$. Since $\mult_{\Phi_\circ}\Xi_\circ'=0$ by case (iii) of Step 2, as in Step 3 we conclude, after shrinking $\Theta_\circ$, that $D_\circ'$ is smooth. Then by \eqref{eq:Esmooth} and \eqref{eq:thetasmooth} we conclude that
\begin{equation}\label{eq:thetasmooth''}
D_\circ''+F_\circ''\text{ has simple normal crossings support,}
\end{equation}
and 
\begin{equation}\label{eq:Esmooth''}
E_\circ''+F_\circ''\text{ has simple normal crossings support.}
\end{equation}
As $\mult_{\nonn_\circ} D_\circ=2$ by case (iii) of Step 2, and as 
$$\mult_{D_\circ' \cap E_\circ'} D_\circ'=\mult_{D_\circ' \cap E_\circ'} E_\circ'=1$$
by \eqref{eq:Esmooth}, by \eqref{eq:thetasmooth}, and as $D_\circ'$ is smooth by above, the equalities \eqref{eq:twoequalities} give
\begin{equation}\label{eq:222}
K_{X_\circ''} \sim (\mu_\circ'')^* K_{X_\circ'} + F_\circ'' \sim (\mu_\circ'')^* (\mu_\circ')^* K_{X_\circ} + E''_\circ + 2 F_\circ''
\end{equation}
and
\begin{equation}\label{eq:223}
(\mu_\circ'')^* (\mu_\circ')^* D_\circ = (\mu_\circ'')^* (D_\circ' + 2 E'_\circ) = D_\circ'' + 2 E''_\circ + 3 F_\circ''.
\end{equation}

Note that $\mu_\circ''|_{D_\circ''}$ is an isomorphism since $D_\circ'$ and $D_\circ'\cap E_\circ'$ are smooth, and the map $\nu_{D_\circ'}$ is also an isomorphism, and we identify $(D_\circ')^\nu$ and $(D_\circ)^\nu$ with $D_\circ''$. Thus, by restricting the equation $(\mu_\circ'')^*E_\circ'=E_\circ''+F_\circ''$ to $D_\circ''$, we obtain
\begin{equation}\label{eq:restrictioneq'}
(\nu_{D_\circ'})^*(E'_\circ|_{D_\circ'})=E_\circ''|_{D_\circ''}+F_\circ''|_{D_\circ''}.
\end{equation}
As in \eqref{eq:image1'}, and since we identify $(D_\circ)^\nu$ with $D_\circ''$, we have
\begin{equation}\label{eq:irreducible1''}
D_\circ'' \cap F_\circ''=\Phi_\circ,
\end{equation}
hence $\mult_{D_\circ''\cap F_\circ''} (\nu_{D_\circ'})^* (E_\circ'|_{D'_\circ}) =2$ by case (iii) of Step 2. By \eqref{eq:thetasmooth''} we have $\mult_{D_\circ''\cap F_\circ''} (F_\circ''|_{D_\circ'' })=1$. Thus, from \eqref{eq:restrictioneq'} we obtain $\mult_{D_\circ''\cap F_\circ''}(E_\circ''|_{D_\circ''})=1$, and observe that 
\begin{equation}\label{eq:supportinDF'}
\Supp(E_\circ''|_{D_\circ''})\subseteq D_\circ'' \cap F_\circ''.
\end{equation}
as in \eqref{eq:supportinDF}. Therefore, as in Step 3 we conclude, after shrinking $\Theta_\circ$, that
\begin{equation}\label{eq:thetasmooth'''}
D_\circ''+E_\circ''\text{ has simple normal crossings support,}
\end{equation}
and also that $D_\circ''\cap E_\circ''\neq\emptyset$ as $\mult_{D_\circ''\cap F_\circ''}(E_\circ''|_{D_\circ''})=1$ by above. But since $D_\circ''\cap E_\circ''\subseteq D_\circ''\cap F_\circ''$ by \eqref{eq:supportinDF'}, and since $D_\circ''\cap F_\circ''$ is irreducible by \eqref{eq:irreducible1''}, we conclude that
\begin{equation}\label{eq:intersectionofthree}
D_\circ''\cap E_\circ''=D_\circ''\cap F_\circ''.
\end{equation}

Let $\holom{\mu'''}{X'''}{X''}$ be the blowup along $D'' \cap E''$, let $G'''$ the unique exceptional divisor that dominates $D_\circ'' \cap E''_\circ$, and let $D'''$, $E'''$ and $F'''$ be the strict transforms of $D$, $E'$ and $F''$ on $X''$. Set $X_\circ''':=(\mu''')^{-1}(X_\circ'')$, $\mu_\circ''':=\mu'''|_{X_\circ'''}$, $D_\circ''':=D'''|_{X_\circ''}$, $E_\circ''':=E'''|_{X'''_\circ}$, $F_\circ''':=F'''|_{X'''_\circ}$ and $G_\circ''':=G'''|_{X'''_\circ}$. Then the divisor
$$D_\circ'''+E_\circ'''+F_\circ'''+G_\circ'''\text{ has simple normal crossings support}$$
by \eqref{eq:thetasmooth''}, \eqref{eq:Esmooth''}, \eqref{eq:thetasmooth'''} and \eqref{eq:intersectionofthree}. Moreover, \eqref{eq:222} and \eqref{eq:223} give
$$ K_{X_\circ'''} \sim (\mu_\circ''')^* K_{X_\circ''} + G_\circ''' \sim (\mu_\circ''')^* (\mu_\circ'')^* (\mu_\circ')^* K_{X_\circ} + E'''_\circ + 2 F_\circ''' + 4 G_\circ''' $$
and
$$ (\mu_\circ''')^* (\mu_\circ'')^* (\mu_\circ')^* D_\circ = (\mu_\circ''')^* (D_\circ'' + 2 E''_\circ + 3 F_\circ'') = D_\circ''' + 2 E'''_\circ + 3 F_\circ''' + 6 G_\circ'''. $$
Therefore
$$ \textstyle K_{X_\circ''} + \frac56 D_\circ''' +  \frac23 E'''_\circ + \frac12 F_\circ''' + G_\circ''' \sim_\Q (\mu_\circ''')^* (\mu_\circ'')^* (\mu_\circ')^* (K_{X_\circ}+\frac56 D_\circ). $$
Now we set $Z:=X'''$, $\mu:=\mu'\circ\mu''\circ\mu'''$, $p:=\frac56$ and $\Delta:=\frac56 D''' + \frac23 E''' + \frac12 F''' + G'''$.
\end{proof}

The following lemma gives a description of the singularities of $D$ in the case where the horizontal part of the conductor is reduced.
This description will be the basis for the descent argument in Section~\ref{section:descending}.

\begin{lemma}\label{lemma:lagrangian5}
Assume Setup~\ref{setup:section7}. Assume, furthermore, that $\nu^*K_D$ is cohomologically trivial over $\Gamma$. By Proposition~\ref{normalisation-formula} there exists an effective divisor $\Xi$ on $D^\nu$ such that $K_{D^\nu} \sim \nu^* K_D - \cond$, and let $\cond^h$ denote the horizontal part of $\cond$ with respect to $\varphi$. Assume that $\cond^h$ is reduced, and set $ \nonn := \nu(\cond^h)$. Denote by $\holom{\alpha_\nonn}{\nonn}{\Gamma_\nonn}$ the algebraic reduction.

Then there exists a dense Zariski open subset $\Gamma_{\nonn, \circ} \subset \Gamma_\nonn$ such that the divisor
 $D$ has normal crossings  singularities of multiplicity two at every point of $\nonn_\circ := \fibre{\alpha_\nonn}{\Gamma_{\nonn,\circ}}$.
\end{lemma}

\begin{proof}
Let $\Gamma_{\nonn, \circ} \subset \Gamma_\nonn$ be the Zariski open set defined in Lemma~\ref{lemma:lagrangian3}(c),(d).
Recall that \footnote{By Lemma~\ref{lemma:lagrangian3}(d) the intersections $\Theta_{1,\circ} \cap \Theta_{2, \circ}$ is empty if
$\Theta$ has two components. In this case, for simplicity of notation we set $\Theta_\circ := \Theta_{1,\circ} \cup \Theta_{2, \circ}$.}
$$
\nonn_\circ := \fibre{\alpha_\nonn}{\Gamma_{\nonn,\circ}}.
$$
is a smooth surface in $X$.
We will show that the statement holds up to replacing $\Gamma_{\nonn, \circ}$ by another Zariski dense open subset.

The setup is analogous to the proof of Proposition~\ref{prop:lagrangian4} and we will use the same notation:
let $\holom{\mu'}{X'}{X}$ be the blowup along $\nonn$, denote by $E'$ the unique exceptional divisor that maps onto $\nonn$, and let $D' \subseteq X'$ be the strict transform of $D$. Set $X_\circ':=(\mu')^{-1}(X_\circ)$, $\mu_\circ':=\mu'|_{X'_\circ}$, $E_\circ':=E'|_{X'_\circ}$ and $D_\circ':=D'|_{X_\circ'}$. Then $\holom{\mu_\circ'}{X_\circ'}{X_\circ}$ is the blowup of $X_\circ$ along $\nonn_\circ$, hence 
$E_\circ'$ is a smooth irreducible exceptional divisor of $\mu_\circ'$.
By construction of $D_\circ$ we have that
$D_\circ'$ is smooth away from $D_\circ'\cap E_\circ'$.

By the universal property of normalisations we have a commutative diagram
$$\xymatrix{
(D_\circ')^\nu \ar[r] \ar[d]_{\nu_{D_\circ'}} & (D_\circ)^\nu  \ar[d]^{\nu_{D_\circ}}  \\
D_\circ' \ar[r]_{\mu_\circ'|_{D_\circ'}} & D_\circ}
$$
and Zariski's main theorem implies that $(D_\circ')^\nu \rightarrow D_\circ^{\nu}$ is an isomorphism.
For simplicity of notation we identify these two manifolds.
By Proposition~\ref{normalisation-formula} there exists an effective divisor $\Xi_\circ'$ on $D_\circ^\nu$ such that
\begin{equation}\label{eq:normalisation''}
K_{D_\circ^\nu} \sim \nu_{D_\circ'}^* K_{D_\circ'} - \Xi_\circ'.
\end{equation}
Combined with 
$$
(\mu_\circ')^* D_\circ \sim  D_\circ' + m E_\circ'\quad\text{and}\quad K_{X_\circ'} \sim (\mu_\circ')^* K_{X_\circ} + E_\circ',
$$
where $m:=\mult_{\nonn_\circ} D_\circ$, we obtain
$$
\Xi|_{D_\circ^{\nu}}=\cond_\circ' + (m-1) \nu_{D_\circ'}^* (E_\circ'|_{D_\circ'}).
$$
By construction we have $\Xi|_{D_\circ^{\nu}}=\Xi^h|_{D_\circ^{\nu}}$ which by assumption is a reduced divisor.
The support of $\nu_{D_\circ'}^* (E_\circ'|_{D_\circ'})$ contains $\Xi^h|_{D_\circ^{\nu}}$, so $m \geq 2$ implies that
$$
\cond_\circ'=0, \qquad m=2
$$
and $\nu_{D_\circ'}^* (E_\circ'|_{D_\circ'})$ is reduced. The first item shows that $D'_\circ$ is normal, so we can identify
it to $D^\nu_\circ$. The last item shows that the intersection $E_\circ' \cap D_\circ'$ is generically reduced. Thus by the second item the 
morphism
$$
E_\circ' \cap D_\circ' \rightarrow \Theta_\circ
$$
is generically \'etale of degree two. The non-\'etale locus of this morphism is an analytic subset of dimension at most one, so 
by Corollary~\ref{corollary:alg-dim-one} it is contained in the fibres of the algebraic reduction $\alpha_\Theta$. Thus, up to replacing 
$\Gamma_{\nonn, \circ}$ (and therefore $\Theta_\circ$) by another Zariski dense open subset, we can assume that the morphism is \'etale. Therefore $E_\circ' \cap D_\circ'$ is a smooth surface. Since it is a complete intersection of $E_\circ'$ and $D_\circ'$, this
shows that $D'_\circ$ is smooth. In conclusion we obtain that
$$
D'_\circ \rightarrow D_\circ
$$
is a finite map such that $D'_\circ$ is smooth and over $\Theta_\circ=D_{\circ, \sing}$ the map $\mu_\circ'|_{D_\circ'}$ is an \'etale cover of degree two. This shows the statement.
\end{proof}

%----------------------------------------------------
%----------------------------------------------------
\section{Extending sections}\label{section:extending}
%----------------------------------------------------
%----------------------------------------------------

In this section we show that $\kappa(X, D) \geq 1$ holds if the divisor $D \subset X$ is uniruled and has singularities
of Kodaira type $II, III, IV$ along a surface given by the horizontal part of the conductor (cf. Section~\ref{section:uniruled} for the terminology). This will be done in two steps: we first show a general result for projective manifolds (Proposition~\ref{prop:projective-extension}), then we use algebraic approximation to reduce to the projective case.

\begin{proposition} \label{prop:projective-extension}
Let $X$ be a smooth projective fourfold with $K_X\sim 0$. Let $D$ be a prime divisor on $X$, let $\holom{\nu}{D^\nu}{D}$ be its normalisation. Let $\merom{\pi}{D^\nu}{\sK}$ be its MRC fibration and assume that $\sK$ is a surface.  Assume moreover that the following conditions hold.
\begin{enumerate}[\normalfont (a)]
\item There exists a surface $S \subseteq D_\sing$ which is not uniruled, and there exists a non-empty Zariski open subset $S_\circ\subseteq S$ with the property that for every $x \in S_\circ$, the divisor $D$ has singularities of Kodaira type $II, III$ or $IV$ at $x$. 
\item The bundle $\nu^* K_D$ has degree zero on the general fibre of $\pi$.
\end{enumerate}
Then $\kappa(X, D) \geq 1$.
\end{proposition}

\begin{proof}
We will give the proof in full detail for singularities of Kodaira type $IV$. The other cases are completely analogous, writing
them is a straightforward, but lengthy exercise

Set
$$ D_\circ := D \setminus (S \setminus S_\circ), \qquad X_\circ := X \setminus (D \setminus D_\circ). $$

\medskip

\emph{Step 1.}
In this step, we prepare a log resolution of $(X,D)$ with particular properties we need in the remainder of the proof.

Let $Z$ be the normalisation of the blowup of $X$ along $S$, let $D_Z$ be the strict transform of $D$ on $Z$, and let $E_Z$ be the unique exceptional divisor mapping onto $S$. Let $X_1$ be a log resolution of the pair $(Z, D_Z+E_Z)$ which is an isomorphism over the locus where $(Z, D_Z+E_Z)$ is log smooth, and let $\holom{f}{X_1}{X}$ be the birational map obtained by composition. Let $D_1$ be the strict transform of $D_Z$ on $X$, and let $E_1$ be the strict transform of $E_Z$ on $X_1$. Let $F_1$ be the sum of all $f$-exceptional prime divisors on $X_1$, apart from~$E_1$.

Let $Z_\circ\subseteq Z$ be the normalisation of the blowup of $X_\circ$ along $S_\circ$, with the induced map $f_{Z_\circ}\colon Z_\circ \to X_\circ$. Since $D_\circ$ has singularities of Kodaira type $IV$ at every point of $S_\circ$, the pair $(X_\circ, \frac{2}{3} D_\circ)$ is log canonical and 
\begin{equation}\label{two-third}
\textstyle K_{Z_\circ} + \frac{2}{3} D_Z|_{Z_\circ} + E_Z|_{Z_\circ} \sim_\Q (f_{Z_\circ})^* \big(K_{X_\circ}+\frac{2}{3} D_\circ\big). 
\end{equation}
Note that $(Z_\circ, D_Z|_{Z_\circ}+E_Z|_{Z_\circ})$ is log smooth. Set $X_{1,\circ}:=f^{-1}(X_\circ)$ and $f_\circ:=f|_{X_{1,\circ}}$. Then \eqref{two-third} implies that
\begin{equation}\label{two-third1}
\textstyle K_{X_{1,\circ}} + \frac{2}{3} D_1|_{X_{1,\circ}} + E_1|_{X_{1,\circ}} \sim_\Q f_\circ^* \big(K_{X_\circ}+\frac{2}{3} D_\circ\big). 
\end{equation}
The construction implies also that
\begin{equation}\label{property dagger}
X_{1,\circ}\cap \supp F_1 = \emptyset.
\end{equation}

Now, fix a rational number $0< \varepsilon < 1$ and set 
$$\textstyle \Delta_1:=\frac23 D_1+E_1+\varepsilon F_1.$$
Then there exist effective $f$-exceptional divisors $A_1$ and $B_1$ on $X_1$ without common components such that
\begin{equation}\label{eq:-E2}
\textstyle K_{X_1}+\Delta_1 \sim_\Q f^*\big(K_X+\frac23 D\big)+A_1-B_1.
\end{equation}
By construction, $E_1$ is not a component of $A_1$ and of $B_1$, hence by \eqref{property dagger} we have
\begin{equation}\label{property dagger1}
X_{1,\circ}\cap \supp (A_1+B_1) = \emptyset.
\end{equation}

Further, since $K_X \sim 0$ and since $X$ is smooth, we have $K_{X_1} \sim H_1$, where $H_1\geq0$ and $\supp H_1=E_1+F_1$. Then we have
\begin{equation}\label{eq:effective}
\textstyle K_{X_1} + \Delta_1 \sim_\Q H_1+\frac{2}{3} D_1 + E_1 + \varepsilon F_1\geq0,
\end{equation}
and
\begin{equation}\label{support-condition}
\textstyle E_1 \subseteq \supp \big(H_1+\frac{2}{3} D_1 + E_1 + \varepsilon F_1\big) = \supp(D_1 + E_1+F_1).
\end{equation}

\medskip

\emph{Step 2.}
By~\cite[Corollary D]{LT22} there exists a sequence of $(K_{X_1}+\Delta_1)$-divisorial contractions or flips
$$ \merom{\tau_i}{(X_i,\Delta_i)}{(X_{i+1},\Delta_{i+1})}, $$
which terminates with a minimal dlt pair $(X_m, \Delta_m)$. We claim that this MMP does not contract $E_1$. We will prove this claim in the next step. We first show in this step how the claim implies the proposition.

By the claim, the pair $(X_m, \Delta_m)$ is plt and the strict transform of $E_1$ on $X_m$ is the unique log canonical place of $(X_m, \Delta_m)$. By  \eqref{support-condition} and by Theorem~\ref{thm:plt-extension}, we have that $\kappa(X_1,K_{X_1}+\Delta_1)=\kappa(X_m,K_{X_m}+\Delta_m) \geq 1$. Since $K_X\sim 0$, we have by \eqref{eq:-E2} that
\begin{align*}
\kappa(X,D)& \textstyle =\kappa\big(X,K_X+\frac23 D\big)=\kappa\big(X_1,f^*\big(K_X+\frac23 D\big)+A_1\big)\\
&=\kappa(X_1,K_{X_1}+\Delta_1+B_1)\geq \kappa(X_1,K_{X_1}+\Delta_1)\geq 1,
\end{align*}
as desired.

\medskip

\emph{Step 3.}
In the remainder of the proof we show the claim from Step 2. In this step, we construct two morphisms which we will use in the next steps.

First, we set
$$ E_{1,\circ}:=E_1\cap X_{1,\circ}\quad\text{and}\quad D_{1,\circ}:=D_1\cap X_{1,\circ}.$$
We define $g\colon E_{1,\circ}\to S_\circ$ as $g:=f|_{E_{1,\circ}}$. This is the first morphism we wanted to define.

Since $D_\circ$ has only singularities of Kodaira type $IV$, the MRC fibration $\holom{h}{D_{1,\circ}}{\sK_\circ}$ is a morphism and $\sK_\circ$ is a surface by assumption. This is the second morphism we wanted to define.

\medskip

\emph{Step 4.}
In the remainder of the proof we will prove the following stronger claim, which clearly implies the claim from Step 2, and thus implies the proposition: 

\emph{For each index $i$ there exist non-empty Zariski open subsets $S_i\subseteq S_\circ$ and $\sK_i\subseteq \sK_\circ$ such that the partial MMP $X_1\dashrightarrow X_i$ is an isomorphism on the set $g^{-1}(S_i)\cup h^{-1}(\sK_i)\subseteq X_1$.}

The proof is by induction on $i$. From now on we assume that the claim holds for an index $i$, and we show that it holds for $i+1$. 

Let $\xi_i\colon X_i\to Z_i$ be the extremal contraction corresponding to $\tau_i$ and let $V_i \subseteq X_i$ be the exceptional locus of $\xi_i$. Denote by $H_i$, $D_i$, $E_i$ and $F_i$ the strict transforms of $H_1$, $D_1$, $E_1$ and $F_1$ on $X_i$. The set $V_i$ is covered by rational curves which are $(K_{X_i}+\Delta_i)$-negative by~\cite[Theorem 1]{Kaw91}.
Since by~\eqref{eq:effective} we have
$$\textstyle K_{X_i} + \Delta_i \sim_\Q H_i+\frac{2}{3} D_i + E_i + \varepsilon F_i\geq0,$$
these rational curves are thus contained in $\supp\big(H_i+\frac{2}{3} D_i + E_i + \varepsilon F_i\big)$, hence in $\supp(D_i + E_i + F_i)$ by \eqref{support-condition}. Therefore,
\begin{equation}\label{eq:subset}
V_i\subseteq\supp(D_i + E_i + F_i).
\end{equation}

\medskip

\emph{Step 5.}
Assume first that $\xi_i$ is divisorial. Then \eqref{eq:subset} shows that $V_i$ is one of the components of the divisor $D_i + E_i + F_i$. If $V_i$ is one of the components of $F_i$, then 
$$V_i\cap\theta_i\big(g^{-1}(S_i)\big)=V_i\cap\theta_i\big(h^{-1}(\sK_i)\big)=\emptyset$$
by \eqref{property dagger}. We are then done if we set $S_{i+1}:=S_i$ and $\sK_{i+1}:=\sK_i$.

Now assume that $V_i=D_i$. Let $C_i$ be a rational curve through a general point of $D_i$ which is contracted by $\xi_i$. By induction, we have the following commutative diagram:
$$
\xymatrix{
h^{-1}(\sK_i) \ar[rr]^{\theta_i}_{\isom} \ar[rd]_{h} && \theta_i\big(h^{-1}(\sK_i)\big). \ar[dl]  \\
& \sK_i &
}$$
Since $\sK$ is not uniruled, the curve $C_i$ must be contracted by the morphism $\theta_i\big(h^{-1}(\sK_i)\big)\to \sK_i$. If we denote $C_1:=\theta_i^{-1}(C_i)\subseteq h^{-1}(\sK_i)$, then \eqref{property dagger1} gives $C_1\cap\supp(A_1+B_1)=\emptyset$. Moreover, note that either $f_*C_1=0$, or $f_*C_1$ is a curve such that $\nu^{-1}(f_*C_1)$ is contracted by $\pi$, since $\sK$ is not uniruled. Thus, by our assumption (b) we conclude that $\nu^*K_D\cdot \nu^{-1}(f_*C_1)=0$. All this, together with \eqref{eq:-E2}, implies:
\begin{align*}
(K_{X_1}+\Delta_1)\cdot C_1&\textstyle=f^*\big(K_X+ \frac{2}{3} D\big)\cdot C_1 = \big(K_X+ \frac{2}{3} D\big)\cdot f_*C_1 = \frac{2}{3} D\cdot f_*C_1 \\
&\textstyle = \frac{2}{3} D|_D\cdot f_*C_1 = \frac{2}{3} K_D\cdot f_*C_1 = \frac{2}{3} \nu^*K_D\cdot \nu^{-1}(f_*C_1)=0,
\end{align*}
where we used that $K_X \sim \sO_X$ and the adjunction formula on $D$. Thus, $\deg_{C_1}(K_{X_1}+\Delta_1)|_{C_1}=0$. Since $\theta_i$ is an isomorphism around $C_1$, we conclude that
$$\deg_{C_i}(K_{X_i}+\Delta_i)|_{C_i}=0,$$
which is a contradiction since the curve $C_i$ is $(K_{X_i}+\Delta_i)$-negative. Thus, $V_i\neq D_i$.

Finally, assume that $V_i=E_i$. Let $R_i$ be a rational curve through a general point of $E_i$ which is contracted by $\xi_i$. By induction, we have the following commutative diagram:
$$
\xymatrix{
g^{-1}(S_i) \ar[rr]^{\theta_i}_{\isom} \ar[rd]_{g} && \theta_i\big(g^{-1}(S_i)\big). \ar[dl]  \\
& S_i &
}$$
Since $S$ is not uniruled by our assumption (a), the curve $R_i$ must be contracted by the morphism $\theta_i\big(g^{-1}(S_i)\big)\to S_i$. If we denote $R_1:=\theta_i^{-1}(R_i)\subseteq g^{-1}(S_i)$, then we conclude by the definition of the map $g$, that $R_1$ is contracted by $f$, and moreover, \eqref{property dagger1} gives $C_1\cap\supp(A_1+B_1)=\emptyset$. All this, together with \eqref{eq:-E2}, implies:
$$\textstyle (K_{X_1}+\Delta_1)\cdot R_1=f^*\big(K_X+ \frac{2}{3} D\big)\cdot R_1 = \big(K_X+ \frac{2}{3} D\big)\cdot f_*R_1 =0. $$
We now obtain contradiction as at the end of the previous paragraph, hence $V_i\neq E_i$.

\medskip

\emph{Step 6.}
Finally, we assume that $\xi_i$ is a small contraction. Then the set $V_i$ is of codimension at least $2$ in $X_i$, and every component of $V_i$ is covered by rational curves, as mentioned in Step 4.

Let $W$ be a component of $V_i$. Assume that
\begin{equation}\label{eq:equalityofsets}
g\big(\theta_i^{-1}(W)\cap g^{-1}(S_i)\big)=S_i\quad\text{or}\quad h\big(\theta_i^{-1}(W)\cap h^{-1}(\sK_i)\big)=\sK_i.
\end{equation}
Since $\dim S_i=2$, this is only possible if the generic point of $W$ belongs to $\theta\big(g^{-1}(S_i)\big)$ or to $\theta\big(h^{-1}(\sK_i)\big)$. Since $W$ is uniruled, then \eqref{eq:equalityofsets} would imply that $S$ or $\sK$ are also uniruled, which contradicts our assumption (a). Therefore, for any component $W$ of $V_i$, the equalities in \eqref{eq:equalityofsets} cannot hold.

If we now set 
$$S_{i+1}:=S_i\setminus g\big(\theta_i^{-1}(V_i)\cap g^{-1}(S_i)\big)\quad\text{and}\quad \sK_{i+1}:=\sK_i\setminus h\big(\theta_i^{-1}(V_i)\cap h^{-1}(\sK_i)\big),$$
then the discussion above shows that $S_{i+1}\neq\emptyset$ and $\sK_{i+1}\neq\emptyset$, which concludes the proof of the claim from Step 4, and thus the proof of the proposition.
\end{proof}

{\em Terminology:} Let $B$ be an irreducible and reduced complex analytic space. 
In what follows we will say that a property holds for a very general point $b \in B$
if there exists a countable union of irreducible analytic subsets $Z_j \subset B$ such that
the property holds for every $b \in B \setminus (\cup_j Z_j)$.

We say that the property holds for a general point $b \in B$ if there exists an {\em analytic}
subset $Z \subsetneq B$ such that the property holds for every $b \in B \setminus Z$. Note that an analytic subset can have
infinitely many irreducible components, but it is nowhere dense.

\begin{lemma} \label{lemma:analytic-sets1}
Let $A$ and $B$ be irreducible and reduced complex spaces, and let $\holom{f}{A}{B}$ be a proper and surjective morphism. Then the following statements hold.
\begin{enumerate}[\normalfont (a)]
\item A general fibre of $f$ is reduced.
\item If a very general fibre of $f$ is irreducible, then a general fibre of $f$ is irreducible.
\item Assume that the fibres of $f$ are compact Kähler spaces. If a very general fibre of $f$ 
contains an irreducible component that is not uniruled, then for a general fibre of $f$ all the irreducible components are not uniruled.
\item The set $\{b\in B\mid f^{-1}(b)\text{ is finite}\}$ is Zariski open. In particular, if there exists a finite fibre of $f$, then a general fibre of $f$ is finite.
\end{enumerate}
\end{lemma}

Note that we do not claim that the locus where the fibres are irreducible is open in the Zariski topology (this is wrong, e.g. for a conic bundle with non-reduced fibres), we only say that it contains a non-empty Zariski open set.

\begin{proof}
Statement (a) follows from ~\cite[Proposition~2.2]{Le15}.

{\em Proof of (b) and (c).}
Let $\holom{\mu}{A'}{A}$ be a resolution of singularities, then 
$$
\holom{f \circ \mu}{A'}{B}
$$
is a proper and surjective morphism from a complex manifold $A'$ onto $B$. By generic smoothness,
there exists an analytic subset $Z \subsetneq B$ such that $f \circ \mu$ is smooth over
$B \setminus Z$. In particular, $B \setminus Z$ is smooth.

If the very general $f$-fibre is irreducible, the very general $f \circ \mu$-fibre is irreducible.
Since $f \circ \mu$ is smooth over $B \setminus Z$, all
the $f \circ \mu$-fibres over $B \setminus Z$ are connected. Hence, all the $f$-fibres
over $B \setminus Z$ are irreducible.

We first prove (c) under the additional assumption that the very general $f$-fibre is irreducible:
if the very general $f$-fibre is not uniruled, the very general $f \circ \mu$-fibre $F_b$
is not uniruled. Therefore the canonical bundle $K_{F_b}$ is pseudoeffective by Ou's theorem~\cite[Theorem~1.1]{Ou25}.
Since $f \circ \mu$ is smooth over $B \setminus Z$, is smooth this implies that
$K_{F_b}$ is pseudoeffective for every $b \in B \setminus Z$ (e.g. by
\cite[Theorem~A.1]{AH21}). In particular, $F_b$ is not uniruled.

Let us now show the general case: let $\holom{\nu_A}{A^\nu}{A}$ be the normalisation, and let
$$
\holom{f'}{A^\nu}{B'}
$$
be the Stein factorisation of $f \circ \nu_A$, and denote $g: B' \rightarrow B$
such that $g \circ f'=f \circ \nu_A$. Since $A^\nu$ is normal the general fibre of $f'$ is irreducible, moreover it is the normalisation of an irreducible component of an $f$-fibre. Thus by assumption a very general $f'$-fibre is not uniruled, by the above the general $f'$-fibre
is not uniruled. For $b \in B$ general, the normalisation of $\fibre{f}{b}$ is the union
$\fibre{(f')}{\fibre{g}{b}}$, a disjoint union of varieties that are not uniruled.

{\em Proof of (d).} This immediately follows from upper semicontinuity of the fibre dimension (see~\cite[Chapter~V.3.3, Theorem~5]{Loj91} or~\cite[Chapter~3.4, Proposition]{Fis76}) and properness.
\end{proof}

Let us also recall the following well-known result.

\begin{lemma} \label{lemma:reducible-fibres}
Let $A$ and $B$ be irreducible and reduced complex spaces, and let $\holom{f}{A}{B}$ be a proper and surjective morphism. Let $F_1$ and $F_2$ be two irreducible components of a general fibre of $f$. Then the points $[F_1]$ and $[F_2]$ belong to the same irreducible component of the cycle space of $A$.
In particular, they have the same homology class in $A$.
\end{lemma}

\begin{proof}
Let $\holom{\nu_A}{A^\nu}{A}$ be the normalisation, and let 
$$
\holom{f'}{A^\nu}{B'}
$$
be the Stein factorisation of $f \circ \nu_A$. By Seidenberg's theorem, a general $f'$\nobreakdash-fibre
is normal. Since it is connected, it is irreducible and maps onto an irreducible component
of a general $f$-fibre. Conversely, the strict transform of an irreducible component of a general $f$-fibre is a general $f'$-fibre. Since~$B'$ is irreducible, this shows the statement.
\end{proof}

\begin{proposition} \label{prop:cycle-space-proper}
Let $\holom{f}{A}{B}$ be a proper morphism of complex spaces. Assume one of the following:
\begin{itemize}
\item the morphism $f$ is smooth; or
\item there exists an embedding $A \hookrightarrow C$ over $B$ such that $C \rightarrow B$ is a smooth proper morphism of complex spaces.
\end{itemize} 
Let $0 \in B$ be a point such that the fibre $A_0$ is K\"ahler (resp.\ $C_0$ is K\"ahler in the second case). 
Then there exists an analytic neighbourhood of $0\in U \subseteq B$ such that  every irreducible component $Z \subseteq \sC(\fibre{f}{U}/U)$ of the relative cycle space is proper over $U$. 
\end{proposition}

\begin{proof}
The first case is~\cite[Proposition~2.6]{GLR13}, for the second case apply~\cite[Proposition~2.6]{GLR13} 
to $C \rightarrow B$ and note that $\sC(A_U/U) \subset \sC(C_U/U)$ is a closed subspace.
\end{proof}

Let $\mbox{Kur}$ be the Kuranishi space of some hyperkähler fourfold $M$. For $b\in \mbox{Kur}$, we will write $X_b$ for the corresponding hyperkähler fourfold. Before we come to the next statement let us recall that if
$\mathfrak B \subset \mbox{Kur}$ is an irreducible component of the Hodge locus 
for some integral $(1,1)$-class on $M$, then for a very  
general point $b \in \mathfrak{B}$ the Picard group of $X_b$ is generated by
a line bundle $L_b$. If $\kappa(X_b, L_b) = 0$, then we know by Lemma~\ref{lemma-unique-divisor} that
$X_b$ contains a unique prime divisor $D_b \subset X_b$.

\begin{proposition} \label{prop:no-Kodaira-type}
Let $\mbox{Kur}$ be the Kuranishi space of some hyperkähler fourfold $M$, and let $\mathfrak B \subset \mbox{Kur}$ be the an irreducible component of the Hodge locus for some parabolic (1,1)-class $c_1(L) \in H^2(M, \Z)$.
Assume that for a very general point $b \in \mathfrak{B}$, the following holds:
\begin{enumerate}
\item we have $\kappa(X_b, L_b)=0$; 
\item if $D_b \subset X_b$ is unique prime divisor, it is uniruled and we are in 
the situation of Proposition~\ref{prop:lagrangian4}, i.e. the horizontal part $\cond^h$ of the conductor is not reduced.
\end{enumerate}
Then there exists a point $p \in \mathfrak B$ such that $X_p$ is projective and $\kappa(X_p, L_p) \geq 1$.
\end{proposition}

In the situation above the line bundle $L_b$ is nef: it is pseudoeffective by assumption $(1)$
and parabolic, so defines a class in the positive cone of $X_b$. Since $\rho(X_b)=1$, we have
$L_b \cdot C=0$ for any curve in $X_b$, so $L_b$ is nef by~\cite[Proposition~3.2]{Huy03}.

We will use a well-known statement.

\begin{lemma} \label{lemma:HK-MRC}
Let $D \subset X$ be a prime divisor in a hyperk\"ahler manifold of dimension $2n$, and let $\holom{\mu}{\tilde D}{D}$
be a resolution of singularities. Then the base of the MRC fibration of $\tilde D$ has dimension at least $2n-2$.
\end{lemma}

\begin{proof}
Let $F \subset \tilde D$ be a general fibre of the MRC fibration. Then $F$ is a smooth rationally connected manifold, so if $\sigma$ is the symplectic form on $X$, we have $(\mu^* \sigma)|_F=0$.
Thus, for $x \in F$ the tangent space $T_{F,x}$ is contained in the kernel of $\mu^* \sigma$.
The morphism
$$
\tilde D \rightarrow D \hookrightarrow X
$$
is an immersion in a general point $x \in F$, so $T_{\tilde D, x}$ is a hyperplane in $T_{X,x}$.
The form $\sigma$ being symplectic, linear algebra shows that $\sigma|_{T_{\tilde D, x}}$ has
rank $2n-2$. Thus, $T_{F,x}$ has dimension at most one.
\end{proof}

\begin{proof}[Proof of Proposition~\ref{prop:no-Kodaira-type}]
Denote by 
$$
\holom{\psi}{\mathfrak X}{\mathfrak B}
$$
the universal family over $\mathfrak B$. Recall that the Hodge locus is a smooth hypersurface in $\mbox{Kur}$ by~\cite[1.14]{Huy99}.  By~\cite[Proposition~2.6]{GLR13}, up to replacing $\mathfrak B$ by an open neighbourhood of some point, we can assume that all the irreducible components of the relative cycle space $\Chow{\mathfrak X/\mathfrak B}$ are proper over $\mathfrak B$.

Since $X_b$ contains a unique prime divisor $D_b$ for all very general $b\in \mathfrak B$, there exists a unique prime divisor $\mathfrak D \subset \mathfrak X$ such that
the fibration
$$
\mathfrak D \rightarrow \mathfrak B
$$
has $D_b$ as its fibre over a very general point $b \in \mathfrak B$.
Since $\pic(X_b) \tensor \Q \isom \Q L_b$ for a very general $b \in \mathfrak B$, we can assume up to replacing
$L$ by a rational positive multiple that $D_b \in |L_b|$. Thus, if $\mathfrak L \rightarrow \mathfrak X$
is the line bundle associated to $\mathfrak D \subset \mathfrak X$, we have $K_{D_b} \isom L_b|_{D_b} \isom \mathfrak L|_{D_b}$ by adjunction.

By Proposition~\ref{prop:projective-extension} we are done if we show that there exists
a point $p \in \mathfrak B$ such that $X_p$ is projective, the divisor $D_p$ is prime and and satisfies the conditions (a) and (b) of Proposition~\ref{prop:projective-extension}. 

Let us first see that for very general $b \in \mathfrak B$ these properties hold: by Lemma~\ref{lemma:lagrangian3}
we know that $\cond^h=2 \Xi$ with $\Xi$ an irreducible surface that is not projective since it is bimeromorphic
to the non-projective surface $S$ (cf. Proposition~\ref{prop:normalisation}(b)).
Note that the vertical components of the conductor are projective since they are contained in $\varphi$-fibres and therefore uniruled (Lemma~\ref{lemma:degenerate-uniruled}).
Thus the singular locus of $D_b$ has a unique irreducible component $\Theta_b$  that is a non-projective surface. In particular, $\Theta_b$ is not uniruled. 
By Proposition~\ref{prop:lagrangian4} the divisor $D_b$ has singularities of Kodaira type $II, III$ or $IV$ along the surface $\Theta_b$.

{\em For the rest of the proof we assume that $D_b$ has singularities of Kodaira type IV along
$\Theta_b$, the other cases are completely analogous.}

The line bundle $L_b$ is parabolic and $\rho(X_b)=1$,
so by Corollary~\ref{cor:allzero} this implies that $L_b$ has degree zero on every curve in $X_b$. In particular
it is clear that $K_{D_b} \isom L_b|_{D_b}$ has degree zero on the fibres of the 
MRC fibration. Thus, all the assumptions of Proposition~\ref{prop:projective-extension} are satisfied
for very general $b \in \mathfrak B$, except the projectivity of $X_b$.

Let $\mathfrak{B}_{\text{\tiny proj}} \subseteq \mathfrak{B}$ be the subset parametrising projective varieties. Then $\mathfrak{B}_{\text{\tiny proj}}$ is a countable union of hypersurfaces that is dense in $\mathfrak{B}$ by~\cite[Proposition~26.6]{GHJ03}. If $N \subsetneq \mathfrak B$ is an analytic subset, then
$\mathfrak{B}_{\text{\tiny proj}} \setminus N \subseteq \mathfrak{B} \setminus N$ is still dense.
Thus, we are done if we show that, up to shrinking the base, 
the properties  (a) and (b) of Proposition~\ref{prop:projective-extension} hold for every
{\em every} $b \in \mathfrak B$. 
This will be the content of the rest of the proof:

By Lemma~\ref{lemma:analytic-sets1}(a) and~(b) we can assume, up to shrinking, that
every fibre of $\mathfrak D \rightarrow \mathfrak B$ is irreducible and reduced.
Let $\mathfrak D_{\sing}$ be the singular locus of the divisor $\mathfrak D$. 
Since $(D_b)_{\sing}$ contains the surface $\Theta_b$ for a very general point $b \in \mathfrak B$,
there exists an irreducible component $\mathfrak T \subset \mathfrak D_{\sing}$ that
contains $\Theta_b$ for very general $b \in \mathfrak B$. A priori the very general fibre
of 
$$
\psi_\mathfrak T: \mathfrak T \rightarrow \mathfrak B
$$
might be reducible. Let us see that this is not the case. Applying Lemma~\ref{lemma:analytic-sets1}(c) to $\mathfrak T \rightarrow \mathfrak B$, we see that all the irreducible components of a fibre of $\psi_\mathfrak T$
are non-uniruled. Yet $(D_b)_{\sing}$ contains a unique component that is not uniruled by
Proposition~\ref{prop:lagrangian4}, so the very general fibre of $\psi_\mathfrak T$ is irreducible.
By Lemma~\ref{lemma:analytic-sets1}(b) and~(c), up to shrinking, all the fibres of
$\psi_\mathfrak T$ can be assumed to be irreducible and non-uniruled.

 Denote by 
 $$
 \holom{\psi_{\mathfrak D}}{\mathfrak D}{\mathfrak B}
 $$ 
the restriction of $\psi$ to $\mathfrak D$, and let $\holom{\nu_{\mathfrak D}}{\mathfrak D^\nu}{\mathfrak D}$ be the normalisation of~$\mathfrak D$. By Seidenberg's theorem 
\cite[Theorem~1.7.1]{BS95} 
a general fibre of $\mathfrak D^\nu \rightarrow \mathfrak B$ is normal, hence after shrinking $\mathfrak B$ we can assume that all the fibres  $\mathfrak D^\nu_b$ are normal. Thus, $\nu_{\mathfrak D}$ is a simultaneous normalisation and for very general $b \in \mathfrak B$ we have $(D_b)^\nu = \mathfrak D^\nu_b$. In particular, the very general fibre of $\mathfrak D^\nu \rightarrow \mathfrak B$
is irreducible. By Lemma~\ref{lemma:analytic-sets1}(b) we can assume, up to shrinking, that they are all irreducible.

{\em Claim 1. We show that, up to shrinking, for all $b \in \mathfrak B$, the divisor $D_b$
has singularities of Kodaira type $IV$ along $\mathfrak T_b$.
}
Let $\mathfrak T_\circ \subset \mathfrak T$ be the locus where $\psi_\mathfrak T$ is smooth.
Up to shrinking, we can assume that $\mathfrak T_\circ \rightarrow \mathfrak B$ is surjective.
Therefore  $(\mathfrak T_\circ)_b \subset \mathfrak T_b$ is a non-empty dense Zariski open subset for all $b$.
We set
$$
\mathfrak X_\circ := \mathfrak X \setminus (\mathfrak D_{\sing} \setminus \mathfrak T_\circ), 
\qquad
\mathfrak D_\circ := \mathfrak D \setminus (\mathfrak D_{\sing} \setminus \mathfrak T_\circ).
$$
Denote by 
$$
\holom{\mu}{\mathfrak X_\circ'}{\mathfrak X_\circ}
$$
the blow-up along $\mathfrak T_\circ$ and by $\mathfrak E_\circ \subset \mathfrak X_\circ'$
the exceptional divisor. Denote by $\mathfrak D'_\circ \subset \mathfrak X'_\circ$ the strict transform of
$\mathfrak D_\circ$. By the proof of Proposition~\ref{prop:lagrangian4} 
we can assume, up to shrinking $\mathfrak T_\circ$, that for very general $b \in \mathfrak B$,
the strict transform $(D'_b)_\circ$ is smooth and $(D'_b)_\circ \rightarrow (D_b)_\circ$ is finite.
By generic smoothness, up to shrinking, 
we can assume that $(D'_b)_\circ$ is smooth and $(D'_b)_\circ \rightarrow (D_b)_\circ$ is finite
for all  $b \in \mathfrak B$. In particular,  $\mathfrak D'_\circ$ is smooth and we have a finite morphism $\mathfrak D'_\circ \rightarrow \mathfrak D'_\circ$.

By the universal property of normalisations and Zariski's main theorem we obtain a morphism
$$
\nu_{\mathfrak D_\circ'} :  (\mathfrak D_\circ)^\nu \rightarrow \mathfrak D'_\circ.
$$
Since $D_b$ has singularities of Kodaira type IV for very general $b \in \mathfrak B$, the intersection
$(\mathfrak D'_\circ \cap \mathfrak E)_b$ is irreducible for very general $b \in \mathfrak B$, cf. \eqref{eq:irreducible1}.
Thus, by Lemma~\ref{lemma:analytic-sets1}(b) we can assume, up to shrinking, that this holds for all $b \in \mathfrak B$ and $\mathfrak D'_\circ \cap \mathfrak E$ is irreducible. 

By the proof of Proposition~\ref{prop:lagrangian4} we have to check that 
$$
\mult_{\nonn_\circ} D_\circ=3, \mult_{\Phi_\circ}\Xi_\circ'=0, \mult_{\Phi_\circ} \nu_{D_\circ'}^* (E_\circ'|_{D_\circ' }) =1.
$$
By assumption this holds for very general $b \in \mathfrak B$. Note that these multiplicities can be obtained by computing them on the total spaces and then restrict
to the general fibres over $\mathfrak B$.  Since multiplicities are upper semi-continuous in the Zariski topology we can assume, up to shrinking, that they are constant. This proves Claim 1.

Since $D_b$ is uniruled for $b \in \mathfrak B$ and uniruledness is a closed property in families of compact K\"ahler spaces, this hold for all $b \in \mathfrak B$. In particular, the general fibre
of the MRC fibration has positive dimension.

{\em Claim 2. We show that, up to shrinking, for all $b \in \mathfrak B$, the divisor $\nu^* K_{D_b}$
has degree zero on the general fibre of the MRC fibration of $D^\nu_b$.
}
Recall that for a very general $b \in \mathfrak B$, the threefold $D^\nu_b$ has a unique dominating family of rational curves, this family is given by the $\PP^1$-bundle $ \holom{\pi_{b,\circ}}{D^\nu_{b,\circ}}{S_{b,\circ}}$. Note that a general member of this family is contained in the smooth locus of $D^{\nu}_{b}$. Since $D^\nu_b$ is a general fibre of $\psi_{\mathfrak D} \circ \nu_{\mathfrak D}$, we have by Bertini
$$
(\mathfrak D^\nu)_\sing \cap D^\nu_b = (D^\nu_b)_\sing.
$$
Thus, a general fibre of $\pi_b$ is contained in the smooth locus of $\mathfrak D^\nu$ and
by the deformation theory of rational curves
$$ 
T_{\mathfrak D^\nu} \otimes \sO_{\PP^1} \isom \sO_{\PP^1}(2) \oplus \sO_{\PP^1}^{\dim \mathfrak D^\nu -1}. 
$$
Therefore the relative cycle space $\sC(\mathfrak D/\mathfrak B)$ is smooth of relative dimension
two in the point corresponding to this rational curve.
Thus, the $\pi_b$-fibres are parametrised by a unique component in the relative cycle space $\sC(\mathfrak D/\mathfrak B)$. Denote by $\mathfrak N$ the normalisation of that component. By the second case of Proposition~\ref{prop:cycle-space-proper}, the natural morphism $ \holom{\psi_*}{\mathfrak N}{\mathfrak B} $ induced by $\sC(\mathfrak D/\mathfrak B) \rightarrow \mathfrak B$ is proper. Denote by $
\holom{\pi}{\mathfrak U}{\mathfrak N} $ the normalisation of the universal family, and by $ \holom{\epsilon}{\mathfrak U}{\mathfrak D^\nu} $ the evaluation morphism.
Then we have a commutative diagram
$$
\xymatrix{
\mathfrak U  \ar[d]^{\epsilon} \ar[r]^{\pi} & \mathfrak N \ar[d]^{\psi_*} \\
\mathfrak D^\nu  \ar[r]^{\psi_\mathfrak D} & \mathfrak B}
$$
For very general $b \in \mathfrak B$ the total space $\mathfrak U_b$ is the graph
of the MRC fibration of $D_b$, in particular it is irreducible. 
By Lemma~\ref{lemma:analytic-sets1}(a) and~(b), applied to $\psi_* \circ \pi$
we can assume, up to shrinking, that $\mathfrak U_b$ is reduced and irreducible for all $b \in \mathfrak B$. 
Fix now an arbitary $b \in \mathfrak B$. Applying Lemma~\ref{lemma:reducible-fibres} to $\pi_b: \mathfrak U_b \rightarrow \mathcal N_b$, we see that for $s \in \mathcal N_b$ general the fibre $\fibre{\pi_b}{s}$ is a cycle of rational curves $\sum \alpha_i C_i$ 
where $[C_i]=[C_j]$ for all $i,j$. Note now that the fibre $\fibre{\pi_b}{s}$ has the same homology
class as a fibre $\fibre{\pi_{b'}}{s'}$ with $b' \in \mathfrak B$ very general. Thus
we have
$$
0 = \mathfrak L \cdot \fibre{\pi_{b'}}{s'} = \mathfrak L \cdot \fibre{\pi_b}{s} 
= \mathfrak L \cdot \sum C_i.
$$
By the above this implies $\mathfrak L \cdot C_i=0$ for each $i$. Yet the $C_i$ define a covering family of rational curves on $D^\nu_b$; by Lemma~\ref{lemma:HK-MRC} such a family is unique and given by the fibres of the MRC fibration. Since $K_{D_b} \isom \mathfrak L|_{D_b}$ this shows the statement.
\end{proof}

%----------------------------------------------------
%----------------------------------------------------
\section{Descending global sections}\label{section:descending}
%----------------------------------------------------
%----------------------------------------------------

In this section, we prove the following result which will be the starting point of the deformation theoretic part.

\begin{theorem} \label{theorem:HK-divisor-abundance}
Let $X$ be a hyperkähler fourfold such that $\pic{X} \isom \Z L$ with a nef parabolic line bundle $L$ with $\kappa(X, L) = 0$. Let
$D$ be the unique prime divisor on $X$, assume that the horizontal part of the conductor $\cond^h$ of the normalisation is reduced and not empty (cf. Lemma~\ref{lemma:lagrangian3} for the notation).
Then $K_D$ is semiample and $\kappa(D,K_D)=1$. 
\end{theorem}

We will reduce this theorem to a technical statement:

\begin{proposition} \label{prop:slc-descent}
Let $D$ be an irreducible and reduced compact K\"ahler threefold such that
$\Theta:=D_\sing$ is a smooth surface and $D$ has normal crossing singularities of multiplicity two in every point of $\Theta$.
Let 
$$
\holom{\nu_{D}}{D^\nu}{D}
$$
be the normalisation of $D$ and denote by $\cond\subset D^\nu$ the conductor. Let $L$ be a line bundle on $D$ such that the following holds:
\begin{enumerate}
\item $\nu_D^* L$ is globally generated and defines a fibration 
$$
\holom{\varphi}{D^\nu}{\Curv}
$$
onto a smooth curve $\Curv$; 
\item every irreducible component of $\cond$ is $\varphi$-horizontal;
\item there exists a modification $\holom{g}{D'}{D^\nu}$  which is an isomorphism near a general $\varphi$-fibre such that
\begin{itemize}
\item $D'$ is normal;
\item the strict transform $\cond' \subset D'$ is normal;
\item there exists a fibration 
$$
\pi'\colon D'\to S
$$
onto a normal surface $S$ such that the algebraic reduction 
$\holom{\alpha}{S}{\Gamma}$ gives a commutative diagram
$$
\xymatrix{
D' \ar[r]^g \ar[d]_{\pi'} & D^\nu  \ar[d]^\varphi \ar[r]^{\nu_D} & D \\
S \ar[r]_\alpha & \Gamma &}
$$
\item the map
$$
\holom{\pi'_{\cond'}}{\cond'}{S}
$$
is generically finite of degree two.
\end{itemize}
\end{enumerate}
Then $L$ is semiample.
\end{proposition}

Our proof is based on the strategy of~\cite{Fuj00}, but the fact that we are not considering a log-canonical class requires some additional care for the linearisation of our line bundle.

\begin{proof}
Since $D$ has only normal crossing singularities of multiplicity two, the morphism
$$
\holom{\gamma_1:= \nu_D|_\cond}{\cond}{\Theta}
$$
is \'etale of degree two. We let
\begin{equation}
\label{define:iota1}
\holom{\iota_1}{\cond}{\cond}
\end{equation}
denote the covering involution. 
Since $(\nu_D^* L)|_\cond \isom \gamma_1^* L|_\Theta$ is a pull-back, it has a natural $G_1$-linearisation, where $G_1\isom \Z/2\Z$ is the subgroup of $\Aut \cond$ generated by the involution $\iota_1$.  Thus, $G_1$ acts by pullback on the global sections and gives a linear automorphism
\begin{equation}
\label{action:iota1}
\iota_1^* : H^0(\cond, m \nu_D^* L) \rightarrow
H^0(\cond, m \nu_D^* L)
\end{equation}
for every $m \in\Z$. 

{\em Step 1. Let $m\in \Z$ be given. We show that a section $s \in H^0(D^\nu, m \nu_D^* L)$ descends to
$D$ if and only if $s|_\cond$ is invariant under the involution~$\iota_1^*$. } 

This is classical: the normalisation gives an exact sequence
$$
0 \rightarrow \sO_D \rightarrow (\nu_D)_* \sO_{D^\nu} \rightarrow Q \rightarrow 0
$$
of sheaves on $D$, where the sheaf
$$
Q \isom ((\gamma_1)_* \sO_{\cond})/\sO_{\Theta}
$$ 
is a line bundle on $\Sigma$. 
Twisting the exact sequence by the line bundle $mL$, we see that a section
$$
s \in H^0(D^\nu, m \nu_D^* L) \isom H^0(D, (\nu_D)_* \sO_{D^\nu} \otimes mL)
$$ 
descends to $D$ if and only if its image in
$H^0(\Theta, Q \otimes mL)$ is zero. Yet by the description of $Q$, cf.
\cite[2.1]{Rei94},  this is the case
if and only if $s|_{\cond}$ is a pull-back from $\Theta$, i.e.
is invariant under the covering involution.

\medskip

By Step 1, we are done if we show that some positive multiple of $\nu_D^* L$ is generated by $\iota_1^*$-invariant global sections. Proving this requires an auxiliary construction:
let 
$$
\holom{\eta}{\cond}{\cond_{st}}, \qquad \holom{\gamma_2}{\cond_{st}}{S}
$$
be the  Stein factorisation of the morphism $\holom{\pi'_{\cond'}}{\cond'}{S}$. By assumption, $\gamma_2$ has degree two, so we have
a covering involution
$$
\holom{\iota_2}{\cond_{st}}{\cond_{st}}
$$
Since $\nu_D^* L$ defines the fibration $\varphi$, we can write
$\nu_D^* L \isom \varphi^* A$ for some ample divisor $A$ on $\Gamma$.
The pull-back
$
\gamma_2^* \alpha^* A
$
has a natural $G_2$-linearisation where $G_2\isom \Z/2\Z$ is the subgroup of $\Aut \cond_{st}$ generated by the involution $\iota_2$.
Thus we have linear automorphisms
$$
\iota_2^* : H^0(\cond_{st}, m \gamma_2^* \alpha^* A) \rightarrow
H^0(\cond_{st}, m \gamma_2^* \alpha^* A).
$$
Now observe that the birational morphism $\cond' \rightarrow \cond$
and $\cond' \rightarrow \cond_{st}$ have connected fibres and
$$
\pi' \alpha^* A \isom g^* \varphi^* A \isom g^* \nu_D^* L.
$$
Thus, we have isomorphisms
$$
\begin{aligned}
H^0(\cond_{st}, m \gamma_2^* \alpha^* A) & \isom H^0(\cond', m \eta^* \gamma_2^* \alpha^* A)\\
&\isom  H^0(\cond', m g^* \nu_D^* L)\\
&\isom   H^0(\cond, m \nu_D^* L).\\
\end{aligned}
$$
Composing with these isomorphisms we obtain 
a linear automorphism
\begin{equation}
\label{action:iota2}
\iota_2^* : H^0(\cond, m \nu_D^* L) \rightarrow
H^0(\cond, m \nu_D^* L).
\end{equation}
The involution $\iota_2$ induces a bimeromorphic map
on $\cond$, so we identify $G_2$ with a subgroup of $\mbox{Bim}(\cond)$.

{\em Step 2. We show that a section $t \in H^0(\cond, m \nu_D^* L)$
extends to a section $s \in H^0(D^\nu, m \nu_D^* L)$
if and only if $t$ is $\iota_2^*$-invariant.}

Recall that 
$$
H^0(\cond, m \nu_D^* L) \isom H^0(\cond_{st}, m \gamma_2^* \alpha^* A) 
$$
and the $\iota_2^*$-action is defined via the covering involution
for $\gamma_2:\cond_{st} \rightarrow S$. Thus, $t$ is $\iota_2^*$-invariant
if and only if it is a pull-back from $S$, i.e. we have
$$
t \in \gamma_2^* H^0(S,m \alpha^* A) \subset H^0(\cond_{st}, m \gamma_2^* \alpha^* A).
$$
Yet the maps $\alpha$ and $\varphi$ have connected fibres, so
$$
 H^0(S,m \alpha^* A)
 \isom H^0(\Curv, m A) \isom
 H^0(D^\nu, m \varphi^* A)
 \isom
 H^0(D^\nu, m \nu_D^* L).
$$
This finishes Step 2.

{\em Step 3. Conclusion.}
By construction, the two involutions $\iota_1,\iota_2 \in \mbox{Bim}(\cond)$  come with canonical linearisations on the line bundle $(\nu_D^* L)|_\cond$. 
By Lem\-ma~\ref{lemma:factor-action}, the algebraic reduction allows to reduce the action of bimeromorphic automorphisms on $L$ to an action of biholomorphisms,
so Corollary~\ref{corollary erzeugnis mit linearisierung} applies: there is a group $G$ with $\iota_1, \iota_2 \in G$ such that $(\nu_D^* L)|_\cond$ has a $G$-linearisation extending the given linearisations for $\left\langle \iota_1\right\rangle$, $\left\langle \iota_2\right\rangle$.
In particular, the $G$-linearisation induces the actions on 
$H^0(\cond, m \nu_D^* L)$
from \eqref{action:iota1} and \eqref{action:iota2} for every $m \in \N$. 
The surface $\cond$ has algebraic dimension one and $(\nu_D^* L)|_\cond$ is nef, so by Corollary~\ref{cor:global-generation}, there exists an $m \in \N$
such that $(\nu_D^* L)|_\cond$ is globally generated by
$G$-invariant sections.
By Step 2, these sections extend to~$D^\nu$, so we obtain a subspace
$$
V \subset H^0(D^\nu, m \nu_D^*L)
$$
of global sections that are $\iota_1^*$-invariant.
By Step 1, these sections descend to $D$, so we are left to show
that $\mbox{Bs} |V|$ is empty:
by construction, $\mbox{Bs} |V|$ is disjoint from $\cond$. 
Yet $\nu_D^* L \isom \varphi^* A$,
so 
all the global sections of $m \nu_D^*L$ are pull-backs from $\Gamma$. 
Thus, if the base locus 
$\mbox{Bs} |V|$ is not empty, it contains a $\varphi$-fibre and is not disjoint
from $\cond^h$. This shows that $\mbox{Bs} |V| = \emptyset$.
\end{proof}

\begin{proof}[Proof of Theorem~\ref{theorem:HK-divisor-abundance}]
Let $\holom{\nu_D}{D^\nu}{D}$ be the normalisation, 
and denote by $\cond\subset D^\nu$ the conductor. Using the notation of Setup~\ref{setup:section7}, we denote by
$\cond^h \subset \cond$ the part of the conductor that is horizontal with respect to the fibration $\varphi$.

By assumption the divisor 
$
\cond^h
$
is non-empty and reduced. By Lemma~\ref{lemma:lagrangian3}(b), we already know that
$\nu_D^* K_D$ is semiample, so our task is to show that semiampleness descends under the normalisation.

Following the notation of Lemma~\ref{lemma:lagrangian3} we set
$$
\Theta := \nu_D(\cond^h)
$$
and denote by
$$
\holom{a_\Theta}{\Theta}{\Gamma_\Theta}
$$
the algebraic reduction.
By Lemma~\ref{lemma:lagrangian5}
 there exists a Zariski open subset $\Theta_\circ \subset \Theta$
that is smooth and proper over $\Gamma_{\Theta_\circ} \subset \Gamma_\Theta$
and $D$ has normal crossing singularities of multiplicity two along $\Theta_\circ$.

By~\cite[Corollary~10.55]{Kol13}, the complex space $D$ admits a semi-log resolution, that is a projective birational morphism
$$
\holom{\tau}{D'}{D}
$$
such that
\begin{enumerate}
\item $D'$ has only smooth points or snc points of multiplicity two;
\item $\tau$ is an isomorphism over $D^{\mbox{\tiny snc2}}$ where $D^{\mbox{\tiny snc2}} \subset D$ is the open set such that $D^{\mbox{\tiny snc2}}$ has at most snc points of multiplicity two.
\end{enumerate}
Since $\nonn_\circ \subset D^{\mbox{\tiny snc2}}$ the modification $\tau$ is an isomorphism
over $D_{\nons} \cup \nonn_\circ$, and we denote by
$\Theta' \subset D'$ the strict transform. Since $D'$ has only snc points of multiplicity
two, its singular locus  $D'_{\sing}$ is a disjoint union of smooth codimension one subsets.
The surface $\Theta' \subset D'$ has codimension one and is contained
in $D'_{\sing}$, so it is a connected component of $D'_{\sing}$.
Up to replacing $D'$ by 
the blowup of $D'$ in $D'_{\sing} \setminus \Theta'$
we can assume that
$$
D'_{\sing} = \Theta'.
$$

{\em Claim. If $\tau^* K_D$ is semiample, then so is $K_D$.}

\medskip

Assuming this for the time being, let us see how to conclude.
The threefold $D'$ and the line bundle $L:=m \tau^* K_D$
satisfy the conditions of Proposition~\ref{prop:slc-descent}:
the conditions on the singularities are by construction,
property (1) follows from Lemma~\ref{lemma:lagrangian3}(b) and (3) already hold for $D'$ by Proposition 
\ref{prop:normalisation}. Condition (2) is clear since by
construction $\Theta'$ is the strict transform
of $\Theta=\nu_D(\cond^h)$. Thus $\tau^* K_D$ is semiample by 
Proposition~\ref{prop:slc-descent} and we conclude with the claim.

{\em Proof of the claim.}
Since $\tau^* K_D$ is semiample, for $m > 0$ sufficiently divisible we have a morphism onto an irreducible and reduced complex space of dimension one
$$
\holom{\varphi' = \varphi_{|m \tau^* K_D|}}{D'}{\Curv'}
$$ 
such that $(\varphi')_* \sO_{D'} \isom \sO_{\Curv'}$.
Let 
$$
\holom{\nu_{D'}}{(D')^\nu}{D'}
$$
be the normalisation of $D'$. By the universal property of the normalisation, there exists
a morphism $\holom{\eta}{(D')^\nu}{D^\nu}$ such that we have a commutative diagram
\begin{equation}
\label{diagram-slc}
\xymatrix{
(D')^\nu \ar[r]^{\nu_{D'}} \ar[d]_\eta & D' \ar[d]_\tau \ar[rdd]^{\varphi'}  & 
\\
D^\nu \ar[d]_\varphi \ar[r]_{\nu_D} & D & \\
\Curv &  & \Curv'
}
\end{equation}

Let $F'$ be a fibre of $\varphi'$ over a general point $t \in \Curv'$ and 
set 
$$
F := \tau(F') \subset D.
$$ 

{\em Step 1. We show that $F$ is contained in the locus where $\tau$ is an isomorphism.}
Since $\tau$ is an isomorphism along $\Theta_\circ$ it is sufficient to show that
$$
F \cap D_{\sing} = F \cap \Theta_\circ.
$$
By Lemma~\ref{lemma:lagrangian3}(b) the line bundle $\nu_D^* K_D$ is semiample and some positive multiple defines the fibration $\varphi$. 
The morphism $\eta$ is a birational map between normal varieties, so has
connected fibres. 
Thus we can apply the projection formula and obtain
$$
H^0((D')^\nu, m \nu_{D'}^* \tau^* K_D)
\isom H^0((D')^\nu, m \eta^* \nu_D^* K_D)
\isom H^0(D^\nu, m \nu_D^* K_D).
$$
Since 
$$
\nu_{D'}^* H^0(D', m \tau^* K_D) \subset H^0((D')^\nu, m \nu_{D'}^* \tau^* K_D)
$$
we obtain that the irreducible components of $\nu_{D'}^* F'$ are general fibres of $\varphi \circ \eta$. Note that
$$
F = \tau(F') = (\tau \circ \nu_{D'})(\nu_{D'}^* F')
= (\nu_D \circ \eta) (\nu_{D'}^* F')
= \nu_D(\eta_* \nu_{D'}^* F'),
$$
so we are reduced to showing that
$$
(\eta_* \nu_{D'}^* F') \cap \fibre{\nu_D}{D_{\sing}} = (\eta_* \nu_{D'}^* F')  \cap \cond^h.
$$
Yet this is clear: by the above $\eta_* \nu_{D'}^* F'$ is a union of
$\varphi$-fibres, so it intersects the conductor $\cond$ only
along the horizontal part $\cond^h$. Moreover  $\fibre{\nu_D}{D_{\sing}}$
is the union of the conductor and $\fibre{\nu_D}{D_{\sing, \leq 1}}$
where $D_{\sing, \leq 1} \subset D_{\sing}$ denotes the union of the irreducible components of dimension at most one. Since
$D$ is normal near the generic point of every irreducible component
of $D_{\sing, \leq 1}$ we have $\fibre{\nu_D}{D_{\sing, \leq 1}} \subset D^\nu_\sing$. Yet by Lemma~\ref{lemma:lagrangian2}(a) the threefold $D^\nu$ is smooth near a general $\varphi$-fibre, so  $\fibre{\nu_D}{D_{\sing, \leq 1}}$
is disjoint from a general $\varphi$-fibre.

{\em Step 2. We show that $K_D$ is semiample.}
By Proposition~\ref{prop:semiample-nonnormal} some multiple of the Cartier divisor $F \subset D$ defines a fibration
$$
\holom{\bar \varphi}{D}{\bar C}
$$ 
such that $\bar \varphi_* \sO_{D} \isom \sO_{\bar C}$. Since $\nu_D^* F$ is a union of $\varphi$-fibres
and $\nu_D^* K_D$ is $\varphi$-cohomologically trivial, we know
that $K_D$ is cohomologically trivial on the $\bar \phi$-fibres. Since $\tau^* K_D$ is semiample by assumption, we obtain that $K_D$ is semiample by Proposition~\ref{prop:semiample+fibration}.
\end{proof}

The following result summarises what we have obtained so far:

\begin{theorem} \label{thm:only-case-left}
Let $X$ be a hyperkähler fourfold such that $\pic{X} \isom \Z L$ with $L$ a nef parabolic line bundle with $\kappa(X, L) = 0$. 
Then $X$ contains a unique prime divisor $D$ and, up to replacing $X$ by some very general deformation in the Hodge locus of $L$, some multiple of the canonical divisor $K_D$ is semiample defines a fibration
$$
\holom{\phi}{D}{U}
$$
onto a curve $U$ and the general fibre has singularities of Kodaira type I (cf. Remark~\ref{remark:kodaira-typeI}).
\end{theorem}
 
\begin{proof}
The existence of a unique prime divisor $D$ is given by Lemma~\ref{lemma-unique-divisor},
up to replacing $L$ by some positive rational multiple we can assume $D \in |L|$. 

We make a case distinction:

{\em 1st case. Assume that $D$ is not uniruled.}
By Proposition~\ref{prop:non-uniruled}, the canonical bundle $K_D$ is semiample, so we obtain a fibration \holom{\phi}{D}{U} onto a curve $U$. 
By Proposition~\ref{prop:non-uniruled}, the general fibre $F$ is Lagrangian and an abelian or bielliptic surface. If the surface is abelian, we infer from~\cite[Theorem~4.1]{GLR13} that $F$ is the fibre of a Lagrangian fibration on $X$. Since $\pic X=\Z L$, we obtain that $\kappa(X, L)>0$, a contradiction to our assumption. If the surface is bielliptic, we apply the same argument using the unobstructedness result for morphisms Theorem~\ref{theorem normal crossing cover corank one smoothing} to reduce to the case of an abelian surface treated above.

{\em 2nd case. Assume that $D$ is uniruled.}
By Proposition~\ref{prop:normalisation}  the general fibre of the MRC fibration
of $D^\nu$ is isomorphic to $\PP^1$, and we denote by $C := \nu(\PP^1)$ the corresponding
rational curve on $X$.
Let $\sK$ be the Kuranishi space of $X$ and let $\mathfrak B \subset
\sK$ be an irreducible component of the Hodge locus of the class $[C]$ that
contains the point $[X] \in \sK$. 

{\em Claim. The variety $\mathfrak B$ is the Hodge locus of the line bundle $L$.
Moreover we have $\kappa(X_t, L_t) \geq 0$ for every $t$ in a small analytic neighbourhood
of $[X] \in \mathfrak B$.}

{\em Proof of the claim.} Let $\sX \rightarrow \sK$
be the universal family over the Kuranishi family,
and let $\sX_{\mathfrak B} \rightarrow \mathfrak B$ be the restriction
over the Hodge locus.

Since $-K_{\sX} \cdot
C = -K_X \cdot C=0$ we know by the Riemann--Roch formula that
$$
\dim_{\nu|_{\PP1}} \mbox{Mor}(\PP^1, \sX) \geq \dim \sX.
$$
We choose an irreducible component of the space of morphisms containing the family of rational curves given by the MRC fibration 
and denote by $\mathcal U \subset \Chow{\sX/\sK}$ the closure
of the locus defined by the images of the deformations of $\nu|_{\PP1}$.
The automorphism group of $\PP^1$ has dimension three, so by the above
$$
\dim \mathcal U \geq \dim \sX - 3.
$$
Let $\holom{\pi}{\Gamma}{\mathcal U}$ be the universal family, and denote by
$\holom{\varepsilon}{\Gamma}{\sX}$ the natural map. 
Since $\mathcal U$ (and therefore $\Gamma$) is proper over $\sK$,
the morphism $\varepsilon$ is proper, so its image $\varepsilon(\Gamma) \subset \sX$
is an analytic subvariety. 
Let us show that the morphism
$$
\Gamma \rightarrow \varepsilon(\Gamma) 
$$
is generically finite: by construction, the fibre of  $\varepsilon(\Gamma) \rightarrow \mathfrak B$ over the point $[X]$ contains the divisor $D$ as an irreducible component. 
Let $x \in D_{\nons} \subset D^\nu$ be a very general point, then 
by Proposition~\ref{prop:normalisation} there is a unique rational curve passing through $x$, i.e
the rational curve $\nu(\PP^1)$ given by the fibre of the MRC fibration. Let now $[Z] \in \mathcal U$ be any cycle passing through $x$. Then the irreducible components of $Z$ are rational curves,
so we have $\nu(\PP^1) \subset Z$. Yet $Z$ is a deformation of $\nu(\PP^1)$, so they have the same cohomology class. Thus we have $\nu(\PP^1) = Z$, i.e. 
the fibre $\fibre{\varepsilon}{x}$ is a singleton.
Hence $\varepsilon$ is generically finite onto its image (cp~\cite[II,Ex.3.22(b)]{Har77}).

By the above
\begin{equation}
\label{imageepsilon}
\dim \varepsilon(\Gamma) = \dim \Gamma = \dim \mathcal U + 1 \geq \dim \sX-2.
\end{equation}
Now recall that the image of $\varepsilon$ is contained in $\sX_{\mathfrak B}$. Moreover for every $t \in \mathfrak B$
the fibre $\varepsilon(\Gamma)_t \subset \sX_t$ is a uniruled variety
and $\sX_t$ is not uniruled. 
Thus we obtain 
$$
\dim \varepsilon(\Gamma) < \dim \sX_{\mathfrak B} < \dim \sX
$$
where for the last inequality we use that the Hodge locus of the rational curve is a strict subvariety
of the Kuranishi space.  Combined with \eqref{imageepsilon} this shows
that $\dim \varepsilon(\Gamma) = \dim \sX -2$,
so $\varepsilon(\Gamma) \rightarrow \mathfrak B$ is surjective and $\mathfrak B \subset \sK$ has codimension one.

Thus $\varepsilon(\Gamma) \rightarrow \mathfrak B$ is flat family of effective divisors
in $\sX_{\mathfrak B} \rightarrow \mathfrak B$. Now recall that $D$ is the unique prime divisor in $X$, so the fibre of $\varepsilon(\Gamma) \rightarrow \mathfrak B$ over the point $[X]$ is linearly equivalent to $mD$ for some $m \in \N$.
Since the numerical equivalence class is constant in flat families we obtain
that for every $t \in \mathfrak B$ the fibre is an effective divisor
$D_t$ such that $c_1(D_t) = m c_1(L)$. In particular, $t$ is in the Hodge locus of $L$,
so there exists a line bundle $L_t \rightarrow \sX_t$ such that $c_1(L_t)=c_1(L)$.
On a hyperk\"ahler manifold numerical and linear equivalence coincide so we have
$D_t \isom m L_t$.  This finishes the proof of the claim.

{\em Final step.} 
By Lemma~\ref{lemma:allzero} we know that
for a very general point $[X_t] \in \mathfrak B$ the Picard group
is generated over $\Q$ by $L_t$ and $L_t$ is nef. If $\kappa(X_t, L_t) \geq 1$ we deduce that $L_t$ is semiample (e.g. by~\cite[Theorem~4.1]{HLL25}. Semiampleness is deformation invariant in the Hodge locus~\cite[Corollary~1.3]{Mat16}, so $\kappa(X, L) > 0$ in contradiction to our assumption. Thus $\kappa(X_t, L_t) =0$ and $D_t \subset X_t$ is the unique prime divisor.

We are in the situation of the Setup~\ref{setup:section7}, so we can consider the $\varphi_t$-horizontal part $\cond^h_t$ of the conductor of the normalisation $\holom{\nu_{D_t}}{D_t^\nu}{D_t}$. By Lemma~\ref{lemma:lagrangian3} the divisor $\cond^h_t$ is not zero.

Assume that for a very general point $[X_t] \in \mathfrak B$ 
the divisor $\cond^h_t$ is not reduced: then we can apply Proposition~\ref{prop:no-Kodaira-type}
to obtain a projective deformation $X_p$ in the Hodge locus of $L$ such that $\kappa(X_p, L_p) \geq 1$. As in the paragraph above
this implies the semiampleness of $L_p$ and yields a contradiction.

Assume that for a very general point $[X_t] \in \mathfrak B$ 
the divisor $\cond^h_t$ is reduced: then we know by Theorem~\ref{theorem:HK-divisor-abundance} that $K_{D_t}$ is semiample with Kodaira dimension one, so some multiple defines a fibration
$$
\holom{\phi_t}{D_t}{U_t}
$$
onto a curve. A general fibre $F_t$ of $\phi_t$ is normalised by a finite union of general fibres of the fibration $\holom{\varphi_t}{D_t}{\Curv_t}$, so $F_t$ meets the singular locus of $D_t$ only along $\Theta_t := \nu_{D_t}(\cond^h_t)$.
Thus, by Lemma~\ref{lemma:lagrangian5} the divisor $D_t$ has only normal crossing singularities of multiplicity two in an open neighbourhood of $F_t$. Since $F_t$ is a general $\phi_t$-fibre, this shows the statement.
\end{proof}

%----------------------------------------------------
%----------------------------------------------------
\section{Corank one normal crossing varieties}\label{section snc fibers}
%----------------------------------------------------
%----------------------------------------------------

We are interested in the global structure of normal crossing varieties with torsion canonical sheaf that can potentially show up as degenerations of abelian varieties. We will focus on ``corank one'': varieties for which the Albanese morphism of the normalisation is a $\PP^1$-fibration. We first analyse the irreducible components of such varieties and then how they are glued.

%----------------------------------------------------
\subsection{Friedman's $d$-semistability}\label{section d-semistability}
%----------------------------------------------------

For a simple normal crossing variety~$Y$, Friedman\footnote{Friedman uses the terms \emph{global} and \emph{local normal crossing} instead of the nowadays more commonly used \emph{simple normal crossing} and \emph{normal crossing}.} introduced the notion of \emph{$d$-semistability} in~\cite[Definition~1.13]{Fri83}. 
By a local computation, one can verify that $T^1_Y:=\E{\rm xt}^1(\Omega_Y,\sO_Y)$ is a coherent sheaf supported on the singular locus~$Y^{\sing}$ and it is a line bundle on its support.  Friedman shows in~\cite[Proposition~2.3]{Fri83} that $Y$ is $d$-semistable if and only if $T^1_Y \isom \sO_\Ysing$. 
Following Kato~\cite[Theorem~11.7]{Kat96}, we use this latter condition as a definition in the normal crossing case.

\begin{definition}\label{definition d-semistable nc}
Let $Y$ be a normal crossing variety. We say that $Y$ is \emph{d-semistable} if $T^1_Y \isom \sO_{Y^\sing}$.
\end{definition}

\begin{remark}
It is shown in~\cite[Corollary~1.12]{Fri83} that $d$-semistability is a necessary condition for $Y$ to be the central fibre of a semistable degeneration. 
\end{remark}

\begin{proposition}\label{proposition friedman kawamata namikawa}
Let $Y$ be a normal crossing variety. Then $Y$ is $d$-semi\-stable if and only if $Y$ admits a log structure of semistable type; see Definition~\ref{definition semistable type log structure}. 
\end{proposition}

\begin{proof}
This is~\cite[Theorem~11.7(2)]{Kat96}. In the simple normal crossing case, the result was proved in~\cite[Proposition~1.1]{KN94}; see also~\cite[Proposition~III.1.8.8]{Ogu18}.
\end{proof}

\begin{remark}
The log structure from Proposition~\ref{proposition friedman kawamata namikawa} is in general not unique. Note, however, that its log differentials are uniquely determined by Corollary~\ref{corollary independence of log differentials}.
\end{remark}

\begin{remark}\label{remark d-semistable two components}
The following special situation is of particular interest in this article. Assume $Y$ is an snc variety without triple points; cf. Section~\ref{section preliminaries} for the terminology. If $\Ysing_0 \subset \Ysing:=Y^\sing$ is a connected component of the singular locus of $Y$, since $Y$ is \emph{simple} normal crossing, there are exactly two irreducible components $Y_0, Y_0'$ of $Y$ containing $\Ysing_0$. By~\cite[Proposition~2.3]{Fri83}, one can express $d$-semistability can by requiring that for each such connected component one has
\[
N_{\Ysing_0/Y_0} \isom N_{\Ysing_0/Y_0'}^*.
\]
If each connected component $\Ysing_0$ of $Y^\sing$ has exactly two embeddings into the normalisation, the analogous statement holds for normal crossing varieties. 
\end{remark}

%----------------------------------------------------
\subsection{\'Etale bisections}\label{section etale bisections}
%----------------------------------------------------

In order to determine the global structure of normal crossing Lagrangian subvarieties, we investigate étale bisections of $\PP^1$-bundles. There are no connected ones with trivial log canonical bundle:

\begin{proposition}\label{proposition no etale bisection}
Let $A$ be a reduced Gorenstein compact complex space of positive dimension with trivial canonical bundle, let $V$ be a rank $2$ vector bundle on $A$, and let $\pi\colon Y:=\mathbb P(V)\to A$ be the associated projective bundle. If $A'$ is a connected étale bisection of $\pi$, then $K_Y + A'$ is a nontrivial $2$-torsion bundle. 
\end{proposition}
\begin{proof}
Assume that $A'\subset Y$ is a connected closed subvariety such that $p:=\pi|_A'\colon A'\to A$ is finite étale of degree $2$. Consider the base change
\[
\xymatrix@C=4em@R=4em{
Y' \ar[r]^{p'}\ar[d]^{\pi'} & Y \ar[d]^\pi\\
A' \ar@/^0.75pc/@{-->}[u]\ar@/^1.5pc/@{-->}[u]\ar[r]^p & A
}
\]
and note that we obtain two sections of $\pi'$, one coming from the inclusion $A' \subset Y$ and a second one coming from the inclusion composed with the covering involution of $p$. Observe that their images $\Sigma_0, \Sigma_\infty \subset Y'$ are disjoint since $p$ is étale. Since $K_A$ is trivial, their sum $\Sigma_0 + \Sigma_\infty$ is an anticanonical divisor in $Y'$ by Corollary~\ref{corollary log canonical bundle}. 

Let us consider the short exact sequence
\begin{equation}\label{eq cokernel sequence double cover}
0\to \sO_Y(-A') \to \sO_Y \to \sO_{A'} \to 0.
\end{equation}
Applying $\pi_*$, we obtain a sequence
\[
0\to \sO_A \to p_*\sO_{A'} \to R^1\pi_*\sO_Y(-A') \to 0.
\]
Since $\sO_Y(-A')$ and $\omega_Y$ become isomorphic when restricted to the fibres of~$\pi$, they differ by a line bundle coming from $A$, say $\omega_Y = \sO_Y(-A') \tensor \pi^*T$. Moreover, we have
\[
{p'}^*\omega_Y(A') = \omega_{Y'}(\Sigma_0 + \Sigma_\infty) \isom \sO_{Y'}
\]
so that $T$ is in fact $2$-torsion. 

It remains to be shown that $T$ is nontrivial. Let us consider the decomposition $p_*\sO_{A'} = \sO_A \oplus T'$ into eigensheaves for the covering involution. Here, $T'$ is a $2$-torsion line bundle on $A$. The inclusion of the first summand $\sO_A$ is nothing but the canonical morphism $\sO_A \to p_*\sO_{A'}$, hence from \eqref{eq cokernel sequence double cover} we infer that
\[
T'\isom R^1\pi_*\sO_Y(-A')  = \left(R^1\pi_*\omega_Y\right) \tensor T^{-1} \isom T^{-1},
\]
where the last isomorphism comes from relative duality. 

We are thus reduced to proving that $T'$ is nontrivial. Suppose otherwise. Then we choose a trivializing global section $t$. In the algebra $p_*\sO_{A'}$, the square $t^2$ lies in the invariant part for the covering involution so that $t^2=s$ for some global section $s\in H^0(A,\sO_A)=\C$. Note that since $A'$ is connected, so is $A=\pi(A')$. We have $s\neq 0$ since the cover is étale. It follows that the equations $t=\pm \sqrt{s}$ define the two connected components of the cover. This contradicts the connectedness assumption on $A'$, hence $T'$ must be nontrivial.
\end{proof}

There are, however, connected étale bisections for some specific bundles such that the log canonical bundle is torsion, as the following example shows.

\begin{example}\label{example many etale bisections}
Let $E$ be an elliptic curve and put $Y:= E \times \PP^1$. We choose a nontrivial $2$-torsion point $p\in E$ and consider the involution
\begin{equation}\label{eq involution hwang oguiso example}
\iota\colon Y \to Y, \qquad (e,z) \mapsto \bigl(e+p, \, \frac{1}{z}\, \bigr).
\end{equation}
Then $\iota$ is fixed point free and the quotient $Y':=Y/\left\langle \iota\right\rangle$ is a $\PP^1$-bundle over the elliptic curve $E':=E/\left\langle p\right\rangle$. We denote by $q\colon E\to E'$ the quotient map. We claim that this action of $\iota$ is induced by a linearisation for the group $G:=\Z/2\Z$ on the trivial rank $2$ bundle $V$ on $E$. Let us give such a linearisation. 

First note that $\tau_p^*V=V$ canonically. Hence, the matrix 
\[
A:=\begin{pmatrix} 0 & 1\\1 & 0\end{pmatrix}
\]
determines an isomorphism $f_A\colon \tau_p^*V\to V$.
From Lemma~\ref{lemma linearization as a group homomorphism}, we obtain a $G$-linearisation of $V$ corresponding to the group homomorphism
\[
\lambda\colon G \to \LinAut(E,V), \quad 1 \mapsto \left(\tau_p, f_A\right),
\]
where $\tau_p(e):=e+p$. Since the Möbius transformation corresponding to the matrix $A$ is exactly $z\mapsto \frac{1}{z}$, we see that $\lambda$ indeed induces $\iota$ on $Y=\PP_E(V)$. 

Since the $G$-action is fixed point free, the $G$-linearised vector bundle $V$ descends to a vector bundle $V'$ on $E'$ and so does the associated projective bundle:
\[
V\isom q^*V' \qquad \textrm{ and } \qquad  Y \isom \PP_E(q^*V') \isom E \times_{E'} \PP_{E'}(V').
\]
Since the linearisation induces $\iota$, the above identifications are $G$-equivariant and we obtain an isomorphism
\[
Y' \isom \PP_E'(V').
\]
It remains to determine $V'$. For this observe that $V$ decomposes into $(\pm 1)$-Eigensheaves:
\[
V = \mathrm{Eig}(\iota, 1) \oplus \mathrm{Eig}(\iota, -1).
\]
This decomposition descends to a decomposition $V'=\sO_{E'}\oplus L$ on $E'$ where $L$ is $2$-torsion. We conclude that 
\[
Y' \isom \PP_{E'}(\sO_{E'}\oplus L).
\]
From the description of $Y'$ as a quotient, we obtain a one dimensional family of étale bisections of the bundle $\pi'\colon Y'\to E'$. Indeed, for any $t \in \PP^1\setminus\{\pm 1\}$, the section $E\times \{t\}$ maps down isomorphically to an étale bisection of $\pi'$. Only for $t=\pm 1$, the image is a section. 
Observe that such an étale bisection actually has trivial normal bundle in $Y'$, since locally around $t \in \PP^1\setminus\{\pm 1\}$ it is the fibre of a fibration.
\end{example}

%----------------------------------------------------
\subsection{The model for trivial canonical class}\label{section trivial canonical}
%----------------------------------------------------

\begin{construction}\label{construction type i fiber}
Let us fix $m\in \Z_{\geq 1}$. We will construct a normal crossing variety $Y_m(A,\tau,L_1,\ldots,L_m)$ depending on the following parameters: an abelian variety $A$, an automorphism $\tau \in \Aut(A)$, and a collection of line bundles $L_1,\ldots, L_m\in \pic(A)$. For this, consider $m$ copies of the projective bundle 
\[
\pi_i\colon Y_i:=\PP(\sO_A\oplus L_i) \to A, \quad i=1,\ldots, m,
\]
and denote $\Sigma_i^{(0)}, \Sigma_i^{(\infty)} \subset Y_i$ the sections of $\pi_i$ corresponding to the surjections $\sO_A \oplus L_i \to \sO_A$ respectively $\sO_A \oplus L_i \to L_i$. The bundle projection $\pi_i$ induces isomorphisms $\Sigma_i^{(0)} \to A$, \ $\Sigma_i^{(\infty)} \to A$ and we set
\[
Y_m(A,\tau,L_1,\ldots,L_m) := \coprod_{i=1}^m Y_i \mathrel{\bigg/\!\sim},
\]
where $\sim$ is the equivalence relation generated by identifying $\Sigma_i^{(\infty)}$ with $\Sigma_{i+1}^{(0)}$ along the isomorphisms
\begin{equation}\label{eq construction cycle identifications via bundle projections}
\pi_{i+1}^{-1}\circ \pi_i\colon \Sigma_i^{(\infty)} \xrightarrow{\;\sim\;} \Sigma_{i+1}^{(0)}
\end{equation}
for $i=1,\ldots, m-1$ and the identification
\[
\pi_{1}^{-1}\circ \tau \circ \pi_m \colon \Sigma_m^{(\infty)} \xrightarrow{\;\sim\;} \Sigma_{1}^{(0)}.
\]
We will also abbreviate $Y_m(A,\tau,L):= Y_m(A,\tau,L,\ldots,L)$.
If $s\in A$ and $\tau:=\tau_s\colon A \to A$ is the translation by $s$, we also write $Y_m(A,s,L)$ instead of $Y_m(A,\tau_s,L)$ and analogously for $m$ different line bundles.
\end{construction}

\begin{construction}\label{construction type i bad fiber}
We use the notation of Construction~\ref{construction type i fiber}. We will construct a normal crossing variety $Y_m(A,L_1,\ldots,L_m, \iota_1,\iota_2)$, where $\iota_1,\iota_2\colon A\to A$ are fixed point free involutions. We define
\[
Y_m(A,L_1,\ldots,L_m,\iota_1,\iota_2) := \coprod_{i=1}^m Y_i \mathrel{\bigg/\!\sim},
\]
where this time $\sim$ is the equivalence relation generated by identifying $\Sigma_i^{(\infty)}$ with $\Sigma_{i+1}^{(0)}$ along the isomorphisms $\pi_{i+1}^{-1}\circ \pi_i$ for $i=1,\ldots, m-1$ as before, and additionally we identify points as follows
\[
a\sim \iota_1(a) \quad \textrm{ for } a\in \Sigma_1^{(0)} \qquad \textrm{ and } \qquad b\sim \iota_2(b) \quad \textrm{ for } b\in \Sigma_m^{(\infty)},
\]
where we identified $\Sigma_1^{(0)}{\isom} A$, $\Sigma_m^{(\infty)}{\isom} A$ via the projections. 
We will write $Y_m(A,L,\iota_1,\iota_2):=Y_m(A,L,\ldots,L,\iota_1,\iota_2)$.
\end{construction}

\begin{remark}
We observe that Construction~\ref{construction type i fiber} does not always yield projective varieties.

Consider the case where $A=E$ is an elliptic curve, the number of components is $m=1$, the automorphism $\tau=\tau_s$ is a translation by some $s\in E$, and $L\in \pic^0(E)$. Then $Y=Y_1(E,s,L)$ is the pushout
\[
\xymatrix{
\Sigma^{(0)} \coprod \Sigma^{(\infty)}\ar[d]_q\ar[r] & \PP(\sO_E\oplus L) \ar[d]\\
E \ar[r] & Y\\ 
}
\]
along the morphism $q$, which on $\Sigma^{(0)}$ is restriction of the bundle projection $\PP(\sO_E\oplus L) \to E$ and on $\Sigma^{(\infty)}$ is the restriction of the same projection composed with the translation $\tau_s$. This is due to~\cite[Theorem~2.2]{Fer03}; see also~\cite[\href{https://stacks.math.columbia.edu/tag/0E25}{Proposition 0E25}]{Stacks}. Note that~\cite{Fer03} only contains the local version of the claimed descent result explicitly; see however~\cite[Lemma~2.6 and Proposition~3.8~(i)]{TT16} for a global statement.

Applied to our situation, the condition for a line bundle $G=\sO_{Y_1}(k)\tensor \pi^*F$ on $Y^\nu=Y_1$ to descend to $Y$ is that 
\begin{equation}\label{eq descent condition line bundle}
L^k \isom F^* \tensor \tau_s^*F \isom \sO_E(ds - 0),
\end{equation}
where $d=\deg F$ and $0\in E$ is a chosen origin. Since $L \in {\pic}^0(E)$, the vector bundle $\sO_E \oplus L$ is nef, i.e.  $\sO_{Y_1}(1)$ is nef on $Y_1$; see e.g.~\cite[II, Theorem~6.2.12]{Laz04}. 
The descent is ample on $Y$ if and only if $G$ is ample on $Y^\nu$. Since $\sO_{Y_1}(1)$ is nef and not big, this is the case if and only if $F$ is ample on $E$ and $k>0$. 
It follows that there are ample bundles on $Y$ if and only if equation~\eqref{eq descent condition line bundle} has a solution. If we write $L\isom \sO_E(p-0)$, we see that this is the case if and only if $\Z_{>0} \cdot s \cap \Z_{>0}\cdot p\neq \emptyset$ inside $E$.
\end{remark}

\begin{definition}\label{definition cycle of p1 bundles}
Let $A$ be an abelian variety. A compact normal crossing variety $Y$ without triple points is called a \emph{cycle of $\PP^1$-bundles} over $A$ if the following holds.
\begin{enumerate}
	\item If $Y^\nu= Y_1\cup \ldots Y_m$ is a decomposition of the normalisation $\nu\colon Y^\nu \to Y$ into irreducible components, then each $Y_i$ is a $\PP^1$-bundle over $A$. Let us denote $\pi_i\colon Y_i \to A$ the bundle projection.
	\item The conductor $\cond \subset Y^\nu$ intersects each component $Y_i$ in two disjoint sections $\Sigma_i^{(0)}$, $\Sigma_i^{(\infty)}$ of the bundle $\pi_i$. 
\item Up to renumbering the $Y_i$ and renaming the sections, $\Sigma_i^{(\infty)}$ is glued isomorphically to $\Sigma_{i+1}^{(0)}$ under the normalisation map $\nu$. Here, $i=1,\ldots,m$ and the indices have to be read modulo $m$.
\end{enumerate}

We call $Y$ \emph{closed chain of $\PP^1$-bundles} over $A$ if (1), (2), and instead of (3) the condition $(3')$ below holds.
\begin{enumerate}
	\item[$(3')$] Up to renumbering the $Y_i$ and renaming the sections, $\Sigma_i^{(\infty)}$ is glued isomorphically to $\Sigma_{i+1}^{(0)}$ under the normalisation map $\nu$, where $i=1,\ldots,m-1$. Moreover, the restriction of $\nu$ to $\Sigma_1^{(0)}$ and to $\Sigma_m^{(\infty)}$ is an étale double cover onto its image and $\nu(\Sigma_1^{(0)})$ and $\nu(\Sigma_m^{(\infty)})$ are disjoint.
\end{enumerate}
\end{definition}

Notice that a closed chain or a cycle of $\PP^1$-bundles over $A$ does not come with a morphism to $A$ in general. Only the normalisation maps to $A$.

\begin{remark}\label{remark dual graph}
In terms of the dual graph (see Definition~\ref{definition dual graph}), Definition~\ref{definition cycle of p1 bundles} can be recast as follows. We require $(1)$ and $(2)$ as in the definition. Then $(3)$ is equivalent to the dual graph being a cycle with $m$ vertices\footnote{For $m=1$, this is a vertex with a loop of multiplicity two.}, while~$(3')$ is equivalent to the dual graph being a chain with loops of multiplicity one at each end. See Figure~\ref{figure gluing} for an illustration.
\end{remark}

\begin{figure}
\begin{tikzpicture}[vertex/.style={circle,draw,minimum size=3mm,inner sep=5pt, fill=black},thick]
\node[vertex] (a1) at (0,0) {};
\node[vertex] (a2) at (1,0) {};
\node[vertex] (a3) at (2,0) {};
\node[vertex] (a4) at (3,0) {};

\draw (a1)--(a2);
\draw[dashed] (a2)--(a3);
\draw (a3)--(a4);

\draw[-] (a1) edge[out=130,in=230,looseness=8] (a1);
\draw[-] (a4) edge[out=50,in=310,looseness=8] (a4);

\begin{scope}[xshift=7cm]
\node[vertex] (b1) at (180:1) {};
\node[vertex] (b2) at (108:1) {};
\node[vertex] (b3) at (36:1) {};
\node[vertex] (b4) at (-36:1) {};
\node[vertex] (b5) at (-108:1) {};

\draw (b4)--(b5)--(b1)--(b2)--(b3);
\draw[dashed] (b3)--(b4);

\end{scope}
\end{tikzpicture}
\caption{The possible dual graphs of $Y$. Only multiplicities different from one are displayed.}
\label{figure gluing}
\end{figure}
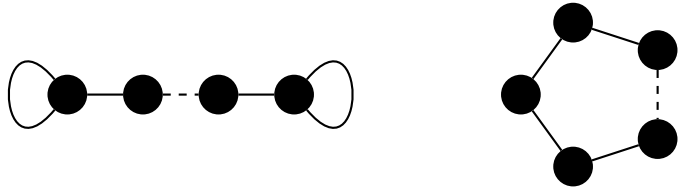

\begin{lemma}\label{lemma cycle of p1 bundles}
A cycle of $\PP^1$-bundles is isomorphic to a variety of the form $Y_m(A,\tau, L_1,\ldots,L_m)$ for some $L_1,\ldots,L_m$ and $\tau \in \Aut(A)$. A closed chain of $\PP^1$-bundles is isomorphic to a variety of the form $Y_m(A, L_1,\ldots,L_m,\iota_1,\iota_2)$ for some $L_1,\ldots,L_m$ and fixed point free involutions $\iota_1,\iota_2 \in \Aut(A)$. For a cycle, the automorphism $\tau$ is uniquely determined up to conjugation in $\Aut(A)$. For a closed chain, the involutions $\iota_1,\iota_2$ are uniquely determined up to simultaneous conjugation in $\Aut(A)$.

\end{lemma}
\begin{proof}
In either case, by definition, each irreducible component of $Y^\nu$ is a $\PP^1$-bundle over $A$ admitting two disjoint sections. Thus, it is necessarily of the form $\PP(V)$ for a decomposable vector bundle $V$. More precisely, there are line bundles $L_i$ on $A$ such that $Y_i \isom \PP(\sO_A \oplus L_i)$ and the given sections $\Sigma_i^{(0)}$ and $\Sigma_i^{(\infty)}$ correspond under this isomorphism to those sections determined by the surjections $\sO_A \oplus L_i \to \sO_A$ respectively $\sO_A \oplus L_i \to L_i$ in this order. 

Let us first assume that $Y$ is a cycle of $\PP^1$-bundles. We will show next that the components are glued as in Construction~\ref{construction type i fiber}. For this purpose, we denote by 
\[
\varphi_i\colon \Sigma_i^{(\infty)}\xrightarrow{\sim}\Sigma_{i+1}^{(0)}\qquad (i=1,\dots,m),
\]
where indices are taken modulo $m$, the isomorphisms that specify how $Y$ is obtained from gluing the $Y_i$. In other words,
\[
\varphi_i = \bigl(\nu|_{\Sigma_{i+1}^{(0)}}\bigr)^{-1}\circ \nu|_{\Sigma_{i}^{(\infty)}},
\]
where $\nu\colon Y^\nu \to Y$ denotes the normalisation. 
Let us denote 
\[
\pi_i\colon Y_i\to A
\]
the corresponding bundle projection. Then the restrictions
\[
\pi_i|_{\Sigma_i^{(0)}}\colon \Sigma_i^{(0)}\xrightarrow{\sim}A,\qquad
\pi_i|_{\Sigma_i^{(\infty)}}\colon \Sigma_i^{(\infty)}\xrightarrow{\sim}A
\]
are isomorphisms. Using the projections, each $\varphi_i$ determines a unique automorphism $a_i\in\Aut(A)$ by
\begin{equation}\label{eq automorphisms cycle}
a_i:=(\pi_{i+1}|_{\Sigma_{i+1}^{(0)}})\circ\varphi_i\circ(\pi_i|_{\Sigma_i^{(\infty)}})^{-1}
\end{equation}
for $i=1,\ldots,m$, where again the indices are to be understood modulo $m$. If we change the bundle projection $\pi_i$ by postcomposing $\pi_i$ with any $b_i\in\Aut(A)$, i.e. 
\[
\pi_i'=b_i\circ\pi_i,
\]
and if we denote $a_i'\in\Aut(A)$ the automorphism defined from $\varphi_i$ using $\pi_i'$ instead of $\pi_i$ in \eqref{eq automorphisms cycle}, then we have
\[
a_i'=b_{i+1}a_i b_i^{-1},
\]
where the indices have to be understood modulo $m$. We choose $b_1 =\id$ and define these reparametrisations inductively by
\[
\qquad b_{i+1}:=b_i a_i^{-1} \qquad \textrm{ for } i=1,\ldots,m-1.
\]
Then for $i=1,\dots,m-1$, we obtain
\[
a_i'=b_{i+1}a_i b_i^{-1}=(b_i a_i^{-1})a_i b_i^{-1}=\id_A.
\]
Hence, for the new bundle projections, we obtain
\[
\varphi_i=(\pi'_{i+1}|_{\Sigma_{i+1}^{(0)}})^{-1} \circ \pi'_i|_{\Sigma_i^{(\infty)}} \colon \Sigma_i^{(\infty)}\xrightarrow{\isom}\Sigma_{i+1}^{(0)}
\]
for $i=1,\ldots,m-1$. So the first $m-1$ gluings are given by the projections, as stipulated in \eqref{eq construction cycle identifications via bundle projections} of Construction~\ref{construction type i fiber}. For the last one,
\[
\tau:= a_m'=b_1 a_m b_m^{-1}= a_m a_{m-1}\cdots a_1\in\Aut(A).
\]
The only freedom we had was the choice of $b_1$. For an arbitrary choice $b_1:=b\in \Aut(A)$, we get that $b_m= ba_1^{-1}\cdots a_{m-1}^{-1}$ so that $a_m'$ is the conjugate of the above expression by $b$. 

For a closed chain of $\PP^1$ bundles, we obtain similarly automorphisms $a_i$ by the formula from~\eqref{eq automorphisms cycle}, this time for $i=1,\ldots,m-1$. The same choices $b_1,\ldots,b_m$ of altering the bundle projections give that the gluing is given by the bundle projections as for a cycle. Again, the only freedom we have left is to change all $b_i$ simultaneously by a different choice of $b:=b_1$. This leaves the gluings invariant and we claim that it acts by conjugation on the two involutions. To see this, let $\vartheta_1\colon\Sigma_1^{(0)}\to \Sigma_1^{(0)}$, $\vartheta_2\colon\Sigma_m^{(\infty)}\to\Sigma_m^{(\infty)}$ the covering involutions of $\Sigma_1^{(0)}$ respectively $\Sigma_m^{(\infty)}$ of the covers given by the restrictions of the normalisation. These are intrinsic to $Y$, i.e. they do not depend on the bundle projections. Having chosen bundle projections $\pi_i'$ as before, the involutions from Construction~\ref{construction type i bad fiber} are defined as
\[
\iota_1:=(\pi'_{1}|_{\Sigma_{1}^{(0)}}) \circ \vartheta_1 \circ (\pi'_{1}|_{\Sigma_{1}^{(0)}})^{-1}, \quad \iota_2:=(\pi'_{m}|_{\Sigma_{m}^{(0)}}) \circ \vartheta_2 \circ (\pi'_{m}|_{\Sigma_{m}^{(0)}})^{-1}.
\]
From this, it is clear that changing all $b_i$ by an overall left multiplication with $b\in\Aut(A)$, both $\iota_1$ and $\iota_2$ change by a conjugation by $b$, as claimed.
\end{proof}

\begin{definition}
Let $Y$ be a cycle of $\PP^1$-bundles over $A$. An automorphism $\tau$ as in Lemma~\ref{lemma cycle of p1 bundles} is called a \emph{gluing automorphism} of the cycle~$Y$. 
If $Y$ is a closed chain of $\PP^1$-bundles over $A$, involutions $\iota_1, \iota_2$ as in Lemma~\ref{lemma cycle of p1 bundles} are called \emph{gluing involutions} of the closed chain~$Y$.
\end{definition}

The following is well known, but spelled out here for clarity.

\begin{lemma}\label{lemma smooth ktrivial with forms is abelian}
Let $Y$ be a smooth projective variety of dimension $n$ with $h^{1,0}(Y)=n-1$. Then $\omega_Y\not\cong \sO_Y$.
\end{lemma}
\begin{proof}
Assume $\omega_Y\isom\sO_Y$. One then deduces from the Beauville--Bogomolov decomposition theorem~\cite[Théorème~1]{Bea83} that there is a finite étale cover $\pi\colon Y'\to Y$ such that $Y'$ is abelian. We may assume that $\pi$ is Galois. Let $G$ denote the Galois group. By assumption on $h^{1,0}(Y)$, the variety $Y$ cannot be abelian itself. Moreover, $G$ acts on $H^0(Y',\Omega_{Y'})$ with an $(n-1)$-dimensional space of invariants. But then it cannot act trivially on $H^0(Y',\omega_{Y'})$, contradicting the hypotheses.
\end{proof}

Recall that $\tOm_Y$ denotes the torsion free Kähler differentials; see Definition~\ref{definition torsion free differentials}.

\begin{lemma}\label{lemma gluing one forms}
Let $Y$ be a connected projective normal crossing variety without triple points. Assume that $\omega_Y \isom \sO_Y$. If $Y^\nu=Y_1\cup\ldots\cup Y_m$ is a decomposition into irreducible components and $\alpha_i\colon Y_i\to A_i$ is the Albanese morphism, then
\[
h^0(Y,\tOm_Y) \leq \min_{i=1}^m \; \dim A_i.
\]
Moreover, if $Y$ is not smooth, then the conductor $\cond \subset Y^\nu$ meets every $Y_i$ nontrivially and $Y_i\cap \cond$ is an anticanonical divisor in $Y_i$. In particular, every $Y_i$ is uniruled if $Y$ is not smooth.
\end{lemma}
\begin{proof}
The case where $Y$ is smooth follows from the fact that all holomorphic one-forms are obtained by pullback from the Albanese variety. We will thus assume that $Y$ is singular. As $Y$ is connected, the conductor $\cond \subset Y^\nu$ meets every irreducible component $Y_i$. By Proposition~\ref{normalisation-formula}, the condition $\omega_Y \isom \sO_Y$ implies that the conductor $\cond \subset Y^\nu$ is an anticanonical divisor. Hence, $-K_{Y_i}$ is effective for every $i$ and $Y_i$ is uniruled by~\cite[Corollary~0.3]{BDPP13}. 

Let $\mu\colon Y^\mu \to Y$ be a partial normalisation such that $Y^\mu$ is connected, has simple normal crossings, and the dual graph is a tree, that is, does not have loops or cycles. Then pullback gives an injection $H^0(Y,\tOm_Y) \to H^0(Y^\mu,\tOm_{Y^\mu})$. It thus suffices to show that 
\[
h^0(Y^\mu,\tOm_{Y^\mu})\leq \min_{i=1}^m \; \dim A_i.
\]
Since $(Y_i,\cond\cap Y_i)$ is log canonical, each $\alpha_i$ is a fibration over $A_i$ by~\cite[Corollary~2]{Zha05}. As  $Y^\mu$ has simple normal crossings, the irreducible components of $Y^\mu$ are those of $Y^\nu$. Up to renumbering them, we may assume that
\[
Y^\mu_{\leq j}:= Y_1 \cup \ldots\cup Y_j \subset Y^\mu
\]
is connected for all $j$ and that $\dim A_1 = \min_{i=1}^m \; \dim A_i$. We will prove the statement by induction on $j$. 

The base case $j=1$ is trivial:
\[
h^0(Y^\mu_{\leq 1},\tOm_{Y^\mu_{\leq 1}}) = h^0(Y_1,\Omega_{Y_1}) \leq \dim A_1 = \min_{i=1}^m \; \dim A_i,
\]
again, since all holomorphic one forms on a variety come by pullback from the Albanese.

For the induction step, note that $Y^\mu_{\leq j+1}=Y^\mu_{\leq j} \cup Y_{j+1}$. Since $Y^\mu$ was chosen so that its dual graph is a tree, there is a unique irreducible component of $Y^\mu_{\leq j}$ that intersects $Y_{j+1}$ and this intersection is irreducible. Without loss of generality, we may assume that $\Sigma:=Y_j\cap Y_{j+1}\neq \emptyset$ as subsets of $Y^\mu$. As in
\cite[Proposition~1.5]{Fri83}, one derives a short exact sequence
\[
0 \to \tOm_{Y^\mu_{\leq j+1}} \to \tOm_{Y^\mu_{\leq j}} \oplus \Omega_{Y_{j+1}} \xrightarrow{\delta} \Omega_\Sigma \to 0,
\]
where $\delta(\eta_1,\eta_2):=\eta_1|_\Sigma - \eta_2|_\Sigma$. Exactness on the right follows since there are no triple points. Taking global sections, we obtain the top row of the following diagram.
\[
\xymatrix@C=4.5mm{
0 \ar[r] & H^0(Y^\mu_{\leq j+1},\tOm_{Y^\mu_{\leq j+1}}) \ar[r] & H^0(Y^\mu_{\leq j},\tOm_{Y^\mu_{\leq j}}) \oplus H^0(Y_{j+1},\Omega_{Y_{j+1}}) \ar[r] & H^0(\Sigma,\Omega_\Sigma)\\
&& \phantom{H^0(Y^\mu_{\leq j},\tOm_{Y^\mu_{\leq j}}) \oplus} H^0(A_{j+1},\Omega_{A_{j+1}}) \ar@<-3em>[u]^{0\oplus\alpha_{j+1}^*} \ar@{^(->}@[]+R+<3pt,0pt>;[ur]+D+<0pt,-3pt> &\\
}
\]
Since $\Sigma$ is a log canonical centre of $(Y_{j+1},Y_{j+1}\cap \cond)$, it follows from~\cite[Theorem~1.5]{BFPT24} that $\Sigma$ surjects onto $A_{j+1}$ under the Albanese morphism $\alpha_{j+1}$. This implies that the above morphism $H^0(A_{j+1},\Omega_{A_{j+1}}) \to H^0(\Sigma,\Omega_\Sigma)$ is injective. Hence, the rank of the last horizontal morphism in the diagram has rank at least $\dim A_{j+1}$. We obtain
\[
\begin{aligned}
h^0(Y^\mu_{\leq j+1},\tOm_{Y^\mu_{\leq j+1}}) &\leq h^0(Y^\mu_{\leq j},\tOm_{Y^\mu_{\leq j}}) + h^0(Y_{j+1},\Omega_{Y_{j+1}}) - h^0(A_{j+1},\Omega_{A_{j+1}})\\
&=  h^0(Y^\mu_{\leq j},\tOm_{Y^\mu_{\leq j}}) \\
&\leq \min_{i=1}^m \; \dim A_i.\\
\end{aligned}
\]
The claim follows since $Y_{\leq m}^\mu=Y^\mu$.
\end{proof}

\begin{corollary}\label{corollary p1 bundles from locally trivial deformations}
Let $Y$ be a connected projective normal crossing variety of dimension $n$ without triple points. Denote by $\nu\colon Y^\nu \to Y$ the normalisation, and let $Y^\nu=Y_1 \cup \ldots \cup Y_m$ be the decomposition into irreducible components. Assume that $\omega_Y\isom \sO_Y$, that $-K_{Y^\nu}$ is nef, and that $h^0(Y,\tOm_Y)\geq n-1$. 

Then for each $i$, the Albanese morphism $\alpha_i\colon Y_i \to A_i$ is a $\PP^1$-bundle over an abelian variety of dimension $(n-1)$. Moreover, we have $$A_1\isom \ldots \isom A_m=:A.$$
\end{corollary}
\begin{proof}
Since $\omega_Y$ is trivial, $Y$ cannot be smooth by Lemma~\ref{lemma smooth ktrivial with forms is abelian}. It follows from Lemma~\ref{lemma gluing one forms} that all $Y_i$ are uniruled and $\dim A_i \geq n-1$ for all $i$. By~\cite[Theorem~1.2]{Cao19}, each $\alpha_i$ is a $\PP^1$-bundle over $A_i$ and $\dim A_i=n-1$. Again by Lemma~\ref{lemma gluing one forms}, the conductor $\cond \subset Y^\nu$ is an anticanonical divisor in each $Y_i$. More precisely, since $\omega_Y\isom \sO_Y$ and $Y$ has no triple points, we infer that $Y_i \cap \cond$ is an étale bisection of the Albanese fibration. Then by Proposition~\ref{proposition no etale bisection}, each $\cond \cap Y_i$ is disconnected, hence a disjoint union of two sections $\Sigma_i^{(0)}$ and $\Sigma_i^{(\infty)}$. In particular, $\Sigma_i^{(0)} \isom A_i \isom\Sigma_i^{(\infty)}$. Whenever $Y_i\cap Y_j \neq \emptyset$ for $i\neq j$, they are glued by identifying one such section in $Y_i$ with another one in $Y_j$. In particular, $A_i \isom A_j$ in this case. Since $Y$ is connected, all $A_i$ are pairwise isomorphic.
\end{proof}

\begin{proposition}\label{proposition torsion cycle or chain from torsion free differentials}
Let $Y$ be a connected projective normal crossing variety of dimension $n$ without triple points. Denote by $\nu\colon Y^\nu \to Y$ the normalisation. Assume that $\omega_Y\isom \sO_Y$, that $-K_{Y^\nu}$ is nef, and that $h^0(Y,\tOm_Y)=n-1$. 

Then $Y$ is either a closed chain or a cycle of $\PP^1$-bundles over~$A$.
\end{proposition}
\begin{proof}
Let $Y^\nu=Y_1 \cup \ldots \cup Y_m$ be the decomposition into irreducible components. By Corollary~\ref{corollary p1 bundles from locally trivial deformations}, there is an abelian variety $A$ such that each $Y_i$ is a $\PP^1$-bundle over $A$. By Lemma~\ref{lemma gluing one forms}, the conductor $\cond \subset Y^\nu$ meets every component $Y_i$ nontrivially. Since there are no triple points, the restriction $\nu_\cond\colon\cond \to \nu(\cond)$ is an étale double cover. Let us denote by $\iota\colon \cond \to \cond$ its deck involution. Then $\iota$ acts on $\pi_0(\cond)$, the set of connected components of~$\cond$. 

The global structure of $Y$ is described by its dual intersection graph; see Definition~\ref{definition dual graph}. It has $m$ vertices, and each vertex has two incident edges, counted with multiplicity, corresponding to the two connected components of the anticanonical divisor. 
Since $Y$ is connected, its dual graph is either a cycle with $m$ vertices or a chain with loops of multiplicity one at each end. Then $Y$ is either a cycle or a closed chain of $\PP^1$-bundles by Remark~\ref{remark dual graph}.
\end{proof}

\begin{lemma}\label{lemma gluing invariant is translation}
Let $Y$ be a closed chain or a cycle of $\PP^1$-bundles over an abelian variety $A$ of dimension $n-1$ and let $\nu\colon Y^\nu \to Y$ be the normalisation. Then $h^0(Y, \tOm_Y) \leq n-1$ and the following statements are equivalent.
\begin{enumerate}
	\item $\dim H^0(Y,\tOm_Y) = n-1$.
	\item $H^0(A,\Omega_A) = H^0(Y,\tOm_Y)$ as subspaces of $H^0(Y^\nu,\Omega_{Y^\nu})$.
	\item Any gluing automorphism of $Y$ is a translation (in case $Y$ is a cycle) respectively any gluing involutions of $Y$ are translations by $2$-torsion points (in case $Y$ is a closed chain). 
\end{enumerate}
\end{lemma}
\begin{proof}
We consider a partial normalisation $\mu\colon Y^\mu \to Y$ such that $Y^\mu$ has simple normal crossings and its dual graph is a tree. Then there is a projection $\pi\colon Y^\mu \to A$ and a commutative diagram
\[
\xymatrix@R=1mm@C=15mm{
&& Y\\
Y^\nu \ar[r]|(0.55){\,\lambda\,} \ar@/^/[rru]^\nu \ar@/_/[rrd]_{\rho} & Y^\mu \ar[ru]|{\,\mu\,}\ar[rd]|{\,\pi\,}\\
&& A\\
}
\]
where $\lambda\colon Y^\nu \to Y^\mu$ is an isomorphism over the normal locus of $Y^\mu$. 

By pullback of holomorphic forms we obtain morphisms
\[
\xymatrix{
H^0(Y,\tOm_Y) \ar@{^(->}[r] & H^0(Y^\mu,\tOm_{Y^\mu}) & H^0(A,\Omega_A). \ar[l]\\
}
\]
We claim that the right morphism is an isomorphism. To see this, note that torsion free differentials inject into the pushforward of Kähler differentials from the normalisation, i.e. $\tOm_{Y^\mu} \into \lambda_*\Omega_{Y^\nu}$. Applying $\pi_*$, we obtain morphisms
\[
\xymatrix{
\Omega_A \ar@{^(->}[r] & \pi_*\tOm_{Y^\mu} \ar@{^(->}[r] & \rho_*\Omega_{Y^\nu}.
}
\]
Their composition is an isomorphism since $\rho_*\Omega_{Y^\nu/A}=0$ and therefore also the first morphism is an isomorphism. The claim follows by taking global sections.

Consequently, we have that 
\begin{equation}\label{eq inclusion differential forms p1 bundles}
H^0(Y,\tOm_Y) \subset H^0(Y^\mu,\tOm_{Y^\mu}) = H^0(A,\Omega_A)
\end{equation}
as subspaces of $H^0(Y^\nu,\Omega_{Y^\nu})$. This implies (1)$\Longleftrightarrow$(2). 

The equivalence with (3) follows since sections of $H^0(Y^\mu,\tOm_{Y^\mu})$ descend to $H^0(Y,\tOm_Y)$ if and only if they are invariant under the gluing. We will only treat the case of a cycle. The case of a closed chain works completely analogously. 

Let us thus assume that $Y$ is a cycle of $\PP^1$-bundles. Then there is a unique connected component $S\subset Y^\sing$ of the singular locus such that $\mu$ is resolves the singularities precisely along $S$. More precisely, let $s_1,s_2\colon A \to Y^\mu$ be sections such that $Y$ is obtained from $Y^\mu$ by quotienting by the equivalence relation generated by $s_1(a) \sim s_2(a)$ for $a\in A$. The images of these sections are the connected components of $\mu^{-1}(S)$ and by Lemma~\ref{lemma cycle of p1 bundles} we may assume that $\pi\circ s_1=\id_A$ and $\pi\circ s_2=\tau$, where $\tau$ is a chosen gluing automorphism of~$Y$. Note that the gluing automorphism is only well defined up to conjugation by Lemma~\ref{lemma cycle of p1 bundles}, but the choice does not matter since translations are a normal subgroup of $\Aut(A)$. 
As in~\cite[Proposition~1.5]{Fri83}, one derives a short exact sequence 
\[
0 \to \tOm_Y \to \mu_* \Omega_{Y^\mu} \xrightarrow{\delta} \Omega_A \to 0,
\]
where $\delta(\eta):=s_1^*\eta - s_2^*\eta$ and $A$ is identified with $S \subset Y$ via $\mu\circ s_1$. The composition
\[
H^0(A,\Omega_A) \xrightarrow{\pi^*} H^0(Y^\mu,\Omega_{Y^\mu}) \xrightarrow{H^0(\delta)} H^0(A,\Omega_A)
\]
is therefore given by $\eta \mapsto \eta - \tau^*\eta$. The first inclusion from \eqref{eq inclusion differential forms p1 bundles} is an equality if and only if $H^0(\delta)$ is identically zero, which in turn is equivalent to $\tau^*=\id$ on $H^0(A,\Omega_A)$. This is equivalent to $\tau$ being a translation.
\end{proof}

\begin{theorem}\label{theorem fibers nc case}
Let $Y$ be a connected projective normal crossing variety of dimension $n$ without triple points, and denote by $\nu\colon Y^\nu \to Y$ the normalisation. Assume that $Y$ is $d$-semistable, $h^0(Y,\tOm_Y)=n-1$, the canonical sheaf $\omega_Y$ is torsion, and $-K_{Y^\nu}$ is nef. 

Let $p\colon Y'\to Y$ be the index one cover. Then either $Y'$ is an abelian variety or there is an abelian variety $A$, a point $s\in A$, and $L\in \pic^0(A)$ such that $Y'\isom Y_m(A,s,L)$ as in Construction~\ref{construction type i fiber}. 
\end{theorem}
\begin{proof}
Let $p\colon Y'\to Y$ be the index one cover. Then $Y'$ is a connected projective normal crossing variety without triple points. Moreover, $\omega_{Y'}$ is trivial and $-K_{{Y'}^\nu}=-p^*K_{Y^\nu}$ is nef, where ${Y'}^\nu\to Y'$ is the normalisation. By Lemma~\ref{lemma log differentials pull back}, also $Y'$ is $d$-semistable. Since pullback is injective, the space of torsion free one forms on $Y'$ has at least dimension $n-1$. If this dimension is strictly bigger than $n-1$, then the Albanese dimension of each irreducible component of ${Y'}^\nu$ is $n$ by Lemma~\ref{lemma gluing one forms}. But this implies that the conductor is zero, so $Y'$ 
is an abelian variety and $Y$ is a finite étale quotient\footnote{One may always pass to a Galois closure of $p$.} of it.

It remains to treat the case where $h^0(Y',\tOm_{Y'})=n-1$. Replacing $Y$ by $Y'$, we may assume that $\omega_Y$ is already trivial. 
Let us denote $Y^\nu=Y_1\cup \ldots \cup Y_m$ the decomposition into irreducible components. Then $Y$ is either a closed chain or a cycle of $\PP^1$ bundles over a abelian variety $A$ of dimension $n-1$ by Proposition~\ref{proposition torsion cycle or chain from torsion free differentials}. By Lemma~\ref{lemma cycle of p1 bundles}, there is an abelian variety $A$, line bundles $L_1,\ldots,L_m$, and an automorphism $\tau \in \Aut(A)$ in the case of a cycle respectively fixed point free involutions $\iota_1,\iota_2 \in \Aut(A)$ in the case of a closed chain such that $Y$ is isomorphic to a variety of the form $Y_m(A,\tau, L_1,\ldots,L_m)$ or to a variety of the form $Y_m(A, L_1,\ldots,L_m,\iota_1,\iota_2)$. Moreover, by Lemma~\ref{lemma gluing invariant is translation} we have that $\tau$ is a translation respectively $\iota_1, \iota_2$ are translations by $2$-torsion points. The next claim will rule out the case of a closed chain.

\emph{Claim:} Let $P \to A$ be a $\PP^1$-fibration over $A$, and let $\Sigma\subset P$ be a section. Let $\iota\colon \Sigma\to\Sigma$ be a translation by a $2$-torsion point, and consider the pushout diagram
\begin{equation}\label{eq pushout diagram bad gluing}
\xymatrix{
\Sigma \ar[d]_p\ar@{^(->}[r] & P \ar[d]^\nu\\
S \ar@{^(->}[r] & Z,\\
}
\end{equation}
where $S:=\Sigma/\iota$. Then the dualizing sheaf $\omega_Z|_S$ is non-trivial.

\emph{Proof of the Claim.}
By~\cite[Theorem~2.6]{Rei94}, there is an exact sequence 
\[
0 \to \omega_Z \to \nu_*\omega_P(\Sigma) \to \omega_S \to 0.
\]
We will show that this sequence remains exact when restricted to $S$. Since restriction and pushforward along affine morphisms commute, we obtain from \eqref{eq pushout diagram bad gluing} that
\[
\nu_*\omega_P(\Sigma)|_S \isom p_*\left(\omega_P(\Sigma)|_\Sigma\right) \isom p_*\omega_\Sigma,
\]
where the second isomorphism is given by adjunction. Restricting to $S$ is right exact, so in the resulting sequence
\begin{equation}\label{eq commutative diagram restriction}
0 \to \omega_Z|_S \to p_*\omega_\Sigma \to \omega_S \to 0,
\end{equation}
it only remains to show exactness at $\omega_Z|_S$. 
But $p_*\omega_\Sigma$ is locally free of rank two and $\omega_S$ of rank one, so $\ker\left(p_*\omega_\Sigma \to \omega_S\right)$ is also locally free of rank one. As $Z$ is Gorenstein, $\omega_Z$ is locally free of rank one and surjects onto this kernel. Therefore, it is isomorphic to it, that is, \eqref{eq commutative diagram restriction} is exact.

We will show next that $\omega_Z|_S \subset p_*\omega_\Sigma$ is the $(-1)$-Eigensheaf for the involution $\iota$ acting on $p_*\omega_\Sigma$; cf.~\cite[Proposition~5.8]{Kol13}.
This claim can be verified locally. Let $Z'\subset Z$ be an open subspace that is isomorphic to the complex space given by the equation $xy=0$ inside an open subset $U \subset \C^{n+1}$. Here, we denote $x,y,z_1,\ldots,z_{n-1}$ the coordinates of $\C^{n+1}$. Let 
\[
\xymatrix{
\Sigma' \ar[d]_p\ar@{^(->}[r] & P' \ar[d]^\nu\\
S' \ar@{^(->}[r] & Z'\\
}
\]
be the diagram obtained from \eqref{eq pushout diagram bad gluing} by base change along $Z'\to Z$, that is, by taking preimages of $Z'$ under all morphisms in the diagram. It now suffices to verify that $\omega_{Z'}|_{S'}$ is the $(-1)$-Eigensheaf for the action of the involution $\iota|_{\Sigma'}$ on $\nu_*\omega_{P'}(\Sigma')|_{S'}$. This will follow from an explicit description of this subsheaf in the chosen local coordinates.

By adjunction, we have $\omega_{Z'} \isom \omega_{U}(Z')|_{Z'}$ and a generator of the right hand side is given by $\vol:=\frac{dx}{x}\wedge \frac{dy}{y} \wedge \eta$ where $\eta:=dz_1\wedge \ldots \wedge dz_{n-1}$. Let $\zeta$ be the corresponding local generator of $\omega_{Z'}$. The subsheaf in question was obtained as the composition
\begin{equation}\label{eq pullback and adjunction}
\xymatrix{
\omega_{Z'}|_{S'} \ar@{^(->}[r] &  \nu_*\omega_{P'}(\Sigma')|_{S'} \ar[r]^(0.6)\isom & p_*\omega_{\Sigma'}\\
}
\end{equation}
of pullback and adjunction, so in local coordinates it can be described as
\begin{equation}\label{eq residue in coordinates}
\begin{aligned}
\zeta &\longmapsto \bigl(\res_{x=0} (\vol), \res_{y=0}(\vol) \bigr) = \Bigl( \frac{dy}{y}\wedge \eta, -\frac{dx}{x}\wedge \eta\Bigr)\\
& \longmapsto (\eta,-\eta).
\end{aligned}
\end{equation}
Note that $S'=\{x=y=0\}$ in these coordinates.

To describe the involution $\iota|_{\Sigma'}$, we consider the involution $\iota'\colon P' \to P'$ interchanging the branches. In coordinates, this is given by $(x,y)\mapsto (y,x)$. We observe that, after possibly rescaling the coordinates, the involution $\iota'|_{\Sigma'}$ coincides with $\iota|_{\Sigma'}$. Thus, \eqref{eq residue in coordinates} implies that $\omega_{Z'}|_{S'}$ is indeed the $(-1)$-Eigensheaf for this involution.

To conclude the proof of the Claim, we will show now that $\omega_Z|_S$ is not trivial. For this, we observe that $p_*\omega_\Sigma$ is equivariantly isomorphic to $p_*\sO_\Sigma$, as translation acts trivially on both sides and the Eigensheaf decomposition is given by $p_*\sO_\Sigma = \sO_S \oplus T$ where the $(-1)$-Eigensheaf $T$ is a $2$-torsion line bundle on $S$. Since $\Sigma$ is compact, reduced, connected, and positive dimensional, $T$ is nontrivial (cp. the proof of Proposition~\ref{proposition no etale bisection}). This finishes the proof of the Claim.\qed

Hence, $Y$ is isomorphic to $Y(A,s,L_1,\ldots,L_m)$ for some $s\in A$. By assumption, $-K_{Y_i}$ is nef. But then by Lemma~\ref{lemma nef anticanonical}, we must have $L_i\in\pic^0(A)$. 

It remains to show that the bundles $L_i$ are all isomorphic for $i=1,\ldots,m$. By Lemma~\ref{lemma class of section of p1 bundle in pic}, we have isomorphisms
\[
\sO_{Y_i}(\Sigma_i^{(0)}) = \sO_{\PP(\sO_A \oplus L_i)}(1) \tensor L_i^*, \qquad \sO_{Y_i}(\Sigma_i^{(\infty)}) = \sO_{\PP(\sO_A \oplus L_i)}(1).
\]
Hence, for the normal bundles, we obtain 
\[
N_{\Sigma_i^{(0)}/\PP(\sO_A \oplus L_i)} \isom \sO_{Y_i}(\Sigma_i^{(0)})\big\vert_{\Sigma_i^{(0)}} = L_i^*
\]
and
\[
N_{\Sigma_i^{(\infty)}/\PP(\sO_A \oplus L_i)} \isom \sO_{Y_i}(\Sigma_i^{(\infty)})\big\vert_{\Sigma_i^{(\infty)}} = L_i.
\]
The claim now follows from the $d$-semistability assumption; cf. Remark~\ref{remark d-semistable two components}.
\end{proof}

\begin{proposition}\label{proposition construction is d semistable}
Let $Y=Y_m(A,s,L)$ be as in Construction~\ref{construction type i fiber} for some $s\in A$. Assume that $L\in\pic^0(A)$. Then $Y$ is $d$-semistable.
\end{proposition}
\begin{proof}
Since there are no triple points, $d$-semistability is equivalent to the statement that for each component of $Y^\sing$, the normal bundles for the two embeddings into $Y^\nu$ are dual to one another; see Remark~\ref{remark d-semistable two components}. For $m=1$, this translates to $$N_{\Sigma_1^{(0)}/Y_1} \isom \tau_s^* N_{\Sigma_1^{(\infty)}/Y_1}^*,$$ where we identify $\Sigma_1^{(0)}\isom A \isom \Sigma_1^{(\infty)}$ via the bundle projection and $\tau_s$ is the translation by $s$. Since 
\[
N_{\Sigma_1^{(0)}/Y_1}\isom L \quad \textrm{ and } \quad N_{\Sigma_1^{(\infty)}/Y_1} \isom L^*,
\]
this is equivalent to $L \isom  \tau_s^* L$, hence satisfied for $L\in {\pic}^0(A)$.

Now assume $m \geq 1$. Then for $i=1,\ldots,m-1$, we glue \smash{$\Sigma_i^{(\infty)}$ to $\Sigma_{i+1}^{(0)}$} via the canonical isomorphism and furthermore we glue $\Sigma^{(\infty)}_m$ to $\Sigma_1^{(0)}$ with an additional translation by $s\in A$. Then as above $d$-semistability is equivalent to the conditions
\[
N_{\Sigma_{i+1}^{(\infty)}/Y_{i+1}} \isom\ N_{\Sigma_i^{(\infty)}/Y_i}^* \quad i=1,\ldots,m-1, \quad \textrm{ and } N_{\Sigma_1^{(0)}/Y_{1}} \isom\ \tau_s^*N_{\Sigma_m^{(\infty)}/Y_m}^*.
\]
These are in turn equivalent to 
\[
L_{i+1}^* \isom\ L_i^*, \quad i=1,\ldots,m-1, \quad \textrm{ and } \quad L_1^* \isom\ \tau_s^*L_m^*.
\]
Hence, the claim follows.
\end{proof}

For $m=1$ and $\dim Y=2$, the $d$-semistability has already been examined in~\cite[Proposition~2.1.4]{Mar24}.

%----------------------------------------------------
%----------------------------------------------------
\section{Normal bundles of normal crossing Lagrangians}\label{section lagrangian normal crossing deformations}
%----------------------------------------------------
%----------------------------------------------------

Our goal in this section is to prove the following result.

\begin{theorem}\label{theorem trivial normal bundle}
Let $Y$ be a Lagrangian subvariety in a holomorphic symplectic manifold $X$. Assume that there exist a finite étale cover $p\colon Y'\to Y$, a point $s\in A$, and a line bundle $L\in\pic^0(A)$ such that $Y'\cong Y_m(A,s,L)$ as in Construction~\ref{construction type i fiber}. Then the pullback $p^*N_{Y/X}$ of the normal bundle is trivial.
\end{theorem}

For the proof, we relate the normal bundle to the sheaf of Kähler differentials of $Y$ as follows. Let $I \subset \sO_X$ denote the ideal sheaf of $Y$.  Recall from~\cite[Lemma~3.5]{Leh16} that there is the following diagram with exact rows:

\begin{equation}\label{eq omega to tangent}
\xymatrix{
0 \ar[r]
& T_{Y} \ar[r] & T_{X}\otimes \sO_Y \ar[r] & N_{Y/X} \ar[r] & T^{1}_{Y} \ar[r] & 0 \\
& I/I^{2} \ar[u] \ar[r] & \Omega_{X}\otimes \sO_Y \ar[u]_\isom \ar[r] & \Omega_{Y} \ar[u] \ar[r] & 0
}
\end{equation}
From this diagram, we see that the sequence
\begin{equation}\label{eq omega to normal}
\Omega_Y \to N_{Y/X} \to T^1_Y \to 0
\end{equation}
is exact. In particular, we obtain the following situation
\[
\Omega_Y \to N_{Y/X} \subset \Omega_Y\tensor \sK_Y = N_{Y/X}\tensor \sK_Y,
\]
where $\sK_Y$ is the sheaf of total quotient rings. For special singularities, one might hope to identify $N_{Y/X}$ explicitly as a certain sheaf of meromorphic sections of $\Omega_Y$. Note that abstractly, this is always possible if $Y$ is lci, due to the following general observation.
\begin{lemma}\label{lemma extension meromorphic sections}
Let $Y$ be a scheme and $\phi\colon  E \to F$ a morphism of quasi-coherent sheaves on $Y$ such that the induced morphism $\phi_\sK\colon E \tensor \sK_Y \to F \tensor \sK_Y$ is an isomorphism. If $F$ is locally free, then there is a unique subsheaf $E_\mer \subset E \tensor \sK_Y$ such that $\phi$ factors through an isomorphism \mbox{$\phi_\mer\colon E_\mer \to F$}.
\end{lemma}
\begin{proof}
Since $F$ is locally free, the morphism $F \to F \tensor \sK_Y$ is injective. Define $E_\mer:=\phi_\sK^{-1}(F)$. Then clearly $\phi$ factors as $E\to E_\mer \to F$. As $\phi_\sK$ is an isomorphism, any subsheaf of $E \tensor \sK_Y$ through which $\phi$ factors, has to map injectively to $F$ under $\phi_\sK$. The requirement that $\phi_\mer$ is an isomorphism thus stipulates that $E_\mer$ must coincide with $\phi_\sK^{-1}(F)$. This proves the uniqueness claim.
\end{proof}

In order to identify the normal bundle of a normal crossing Lagrangian subvariety, we need a holomorphic version of Moser's theorem. The proof is essentially the same as in the $C^\infty$ case; see e.g.~\cite[Theorem~7.4]{CdS01}.

\begin{lemma}\label{lemma moser snc lagrangian}
Let $(U,\sigma)$ be a holomorphic symplectic manifold with $U\subset \C^{2n}$ open, and let $Y \subset U$ be a Lagrangian normal crossing subvariety. Then, possibly after shrinking $U$, there is an open set $V \subset \C^{2n}$ and biholomorphic map $\phi\colon V \to U$ such that
\begin{equation}\label{eq pullback is standard symplectic}
\phi^*\sigma=\sum_{i=1}^n dx_i\wedge dy_i
\end{equation}
and 
\[
\phi^{-1}\left(Y\right)=\begin{cases} \; \{y_1=\ldots=y_n=0\} & \textrm{ if } Y \textrm{ is smooth};\\[3pt] \;\{x_1y_1=0,\ y_2=\ldots=y_n=0\} & \textrm{ otherwise.}\\
\end{cases}
\]
\end{lemma}
\begin{proof}
A Lagrangian subvariety in $X$ with normal crossing singularities has no triple points; see~\cite[Lemma~3.2]{Leh16} or~\cite[Lemma~5.3]{GLR14}. So by Lemma~\ref{lemma snc sub is sncsub}, we can choose holomorphic coordinates $x_1,\ldots,x_n,y_1,\ldots, y_n$ such that the defining equations for $Y$ are as required. It remains to find a coordinate change that preserves $Y$ and turns $\sigma$ into the standard symplectic form $\sigma_\std$, i.e. the right hand side of \eqref{eq pullback is standard symplectic}. We will treat the case where $Y$ has irreducible components 
\[
Y_1=\{y_1=\dots=y_n=0\},
\qquad
Y_2=\{x_1=0,\ y_2=\dots=y_n=0\};
\]
the case of a smooth Lagrangian is similar. By a linear change of coordinates, the form $\sigma$ can be brought to standard form on a given tangent space, that is, we may assume that $\sigma(0)=\sigma_\std(0)$, where $0$ has to be understood in the chosen coordinates.
We set
\[
\eta:=\sigma-\sigma_\std.
\]
Then $d\eta=0$ and $\eta(0)=0$. Moreover, since $Y_1$ and $Y_2$ are Lagrangian for both $\sigma$ and $\sigma_\std$, we have
\begin{equation}\label{eq moser restriction to components}
\eta|_{Y_1}=0,
\qquad
\eta|_{Y_2}=0.
\end{equation}
We will show that, up to shrinking $U$, around any given point $p\in Y_1\cap Y_2$ we can find a holomorphic $1$-form $\alpha$ on $U$ such that $d\alpha=\eta$ and
\[
\alpha|_{Y_1}=0,
\qquad
\alpha|_{Y_2}=0.
\]
First, by the holomorphic Poincaré lemma, we find a holomorphic $1$-form $\alpha'$ on $U$, possibly after shrinking the latter around $p$, such that $d\alpha'=\eta$. Then by \eqref{eq moser restriction to components}, we have $d\alpha'|_{Y_i}=0$ so that after shrinking $U$ further around $p$ there are holomorphic functions $g_i$ on $Y_i$ with $dg_i = \alpha'|_{Y_i}$ for $i=1,2$. In particular, $dg_1|_{Y_1 \cap Y_2} = \alpha'|_{Y_1 \cap Y_2} = dg_2|_{Y_1 \cap Y_2}$ so that $g_1|_{Y_1 \cap Y_2}-g_2|_{Y_1 \cap Y_2}$ is constant. Adding this constant to $g_2$, we may assume that both functions coincide on $Y_1\cap Y_2$. Thus, after shrinking $U$ once more around $p$, we find a holomorphic function $g$ on $U$ such that $dg|_{Y_i} = \alpha'|_{Y_i}$ for $i=1,2$. Then the holomorphic $1$-form $\alpha:=\alpha' - dg$ has all the required properties. 

Then we consider the family of holomorphic $2$-forms
\[
\sigma_t:=\sigma_\std+t\eta
=
(1-t)\sigma_\std+t\sigma
\]
for a complex parameter $t$. Thus, $\sigma_0=\sigma_\std$ and $\sigma_1=\sigma$. After shrinking $U$, we may assume that $\sigma_t$ is nondegenerate for every $t\in [0,1]$. This follows because nondegeneracy is an open condition and $\sigma_t(0)=\sigma_\std(0)$ for all $t\in\C$. Moreover, $Y$ remains Lagrangian for every $\sigma_t$ since both $\sigma_\std$ and $\eta$ vanish on either $Y_i$.

We define a parameter-dependent holomorphic vector field $\mathfrak X_t$ by
\[
\iota_{\mathfrak X_t}\sigma_t=-\alpha.
\]
This equation has a unique solution because $\sigma_t$ is nondegenerate. Since $Y$ is Lagrangian for each $\sigma_t$, the vector field $\mathfrak X_t$ is tangent to $Y$ for all $t$. Consequently, the local flow $\Phi_t$ of $\mathfrak X_t$ preserves $Y$.

We now compute $\Phi_t^*\sigma_t$. From by~\cite[Theorem~6.4]{CdS01}, we get
\[
\frac{d}{dt}\Phi_t^*\sigma_t
= \Phi_t^*\left(\frac{d}{dt}\sigma_t+\scrL_{\mathfrak X_t}\sigma_t\right).
\]
Using Cartan's formula and $d\sigma_t=0$, we have
\[
\scrL_{\mathfrak X_t}\sigma_t
= d(\iota_{\mathfrak X_t}\sigma_t)+\iota_{\mathfrak X_t}d\sigma_t
= d(\iota_{\mathfrak X_t}\sigma_t)
= -d\alpha 
= -\eta
\]
so that
\[
\frac{d}{dt}\sigma_t+\scrL_{\mathfrak X_t}\sigma_t = \eta-\eta = 0.
\]
It follows that
\[
\frac{d}{dt}\Phi_t^*\sigma_t=0.
\]
Hence, $\Phi_t^*\sigma_t$ is independent of $t$. Since $\Phi_0=\operatorname{id}$, we obtain
\[
\Phi_1^*\sigma = \Phi_1^*\sigma_1 = \Phi_0^*\sigma_0 = \sigma_\std.
\]
Finally, because $\phi:=\Phi_1$ preserves $Y$ as a set, the claim follows. 
\end{proof}

\begin{theorem}\label{theorem lagrangian d semistable}
Let $Y \subset X$ be a normal crossing Lagrangian subvariety in a holomorphic symplectic manifold $(X,\sigma)$. Then $Y$ is $d$-semistable.
\end{theorem}

In the proof, we will also derive the following description of logarithmic differential forms in local coordinates. 

\begin{proposition}\label{proposition moser coordinates for log differentials}
The sheaf $\Omega_Y(\log)$ from Definition~\ref{definition relative omega log} can be described locally in terms of the coordinates given by Lemma~\ref{lemma moser snc lagrangian}. More precisely, in these coordinates on an open set $U\subset X$ around a singular point of $Y$, we have
\[
\Omega_Y(\log)|_{Y\cap U} = \sO_{Y\cap U}\left\langle \frac{dy_1}{y_1}, \frac{dx_1}{x_1}, dx_2, \ldots, dx_n\right\rangle \mathrel{\bigg/\! \left\langle \frac{dy_1}{y_1} + \frac{dx_1}{x_1}\right\rangle}.
\]
\end{proposition}

\begin{proof}[Proof of Theorem~\ref{theorem lagrangian d semistable} and Proposition~\ref{proposition moser coordinates for log differentials}]
\leavevmode

By Proposition~\ref{proposition friedman kawamata namikawa}, the claim of the theorem is equivalent to $Y$ admitting a log structure of semistable type. We will construct a log atlas for~$Y$; see Definition~\ref{definition log atlas kawamata namikawa} and Lemma~\ref{lemma kawamata namikawa is semistable type}. For this, we use Lemma~\ref{lemma moser snc lagrangian} to obtain a collection $\{U_i\}_i$ of open subsets of $Y$ covering $Y^\sing$ and, for each $i$, coordinates $$x_1^{(i)}, \ldots, x_n^{(i)}, y_1^{(i)}, \ldots, y_n^{(i)}$$ on $U_i$ such that $Y \cap U_i$ is cut out in $U_i$ by the equations
\[
x^{(i)}_1\cdot y^{(i)}_1=0,\; y^{(i)}_2=\ldots=y^{(i)}_n=0
\]
and $\sigma$ takes the standard form $\sum_i dx^{(i)}_i \wedge dy^{(i)}_i$ in these coordinates. The last $n-1$ equations define a smooth variety of dimension $n+1$ where $Y \cap U_i$ is embedded as an snc divisor, as required for a log atlas.

We will show next that coordinate changes satisfy the unit condition~\eqref{eq kawamata namikawa unit condition} on $Y^\sing$. On an overlap $U_{ij}:=U_i \cap U_j$, the defining equations of the two branches are related by units. Up to renaming, we may thus assume that
\[
x_1^{(i)} = u_x^{(ij)} x_1^{(j)} \quad \textrm{ and } \quad y_1^{(i)} = u_y^{(ij)} y_1^{(j)} 
\]
for some invertible holomorphic functions $u_x^{(ij)}, u_y^{(ij)}$ on $U_{ij}$. From the conditions
 \[
\begin{aligned}
0&=d\log x_1^{(i)}+ d\log y_1^{(i)}\\
0&=d\log x_1^{(j)}+ d\log y_1^{(j)} = d\log x_1^{(i)}+ d\log y_1^{(i)} + d\log \bigl(u_x^{(ij)}\cdot u_y^{(ij)}\bigr)
\end{aligned}
\]
on $Y \cap U_{ij}$, we infer that the product $u_x^{(ij)}\cdot u_y^{(ij)}$ is constant there. On the Zariski tangent space of a point in $Y^\sing \cap U_{ij}$, we have $dx_1^{(i)}\wedge dy_1^{(i)}=dx_1^{(j)}\wedge dy_1^{(j)}$ since both represent the restriction of the symplectic form from $X$ and at the same time $dx_1^{(i)}\wedge dy_1^{(i)}=u_x^{(ij)}\cdot u_y^{(ij)} dx_1^{(j)}\wedge dy_1^{(j)}$. Clearly, $dx_1^{(i)}\wedge dy_1^{(i)} \neq 0$ at such a point so that we must have $u_x^{(ij)}\cdot u_y^{(ij)} = 1$ on~$Y^\sing$. Since we have already shown that the product is constant, it has to be equal to one everywhere on $Y\cap U_{ij}$, possibly after shrinking the $U_i$. Thus, we have proven the theorem. 

The proposition follows, since we have just proven that the coordinates from Lemma~\ref{lemma moser snc lagrangian} form a log atlas for $Y$ and hence a log structure of semistable type.
\end{proof}

\begin{proposition}\label{proposition normal sheaf}
Let $X$ be a holomorphic symplectic manifold and let $Y\subset X$ be a Lagrangian normal crossing subvariety. Then the morphism $\Omega_Y \to N_{Y/X}$ extends to an isomorphism $\Omega_Y(\log) \to N_{Y/X}$. 
\end{proposition}
\begin{proof}
By Lemma~\ref{lemma extension meromorphic sections}, there is a unique subsheaf $\Omega_{Y,\mer} \subset \Omega_Y \tensor \sK_Y$ such that $\phi\colon\Omega_Y \to N_{Y/X}$ induces an isomorphism $\phi_\mer\colon\Omega_{Y,\mer} \to N_{Y/X}$. By Proposition~\ref{proposition relative log differentials into meromorphic differentials}, we have $\Omega_Y(\log) \subset \Omega_Y \tensor \sK_Y$, so it suffices to prove that the image of $\Omega_Y(\log)$ under $\phi_\sK = \phi \tensor\sK_Y$ is exactly $N_{Y/X}$.

This can be verified locally. By Proposition~\ref{proposition moser coordinates for log differentials}, a local frame for $\Omega_{Y}(\log)$ is given by 
\[
\frac{dy_1}{y_1}=-\frac{dx_1}{x_1}, dx_2, \ldots, dx_n.
\]
The symplectic form \eqref{eq pullback is standard symplectic} maps these forms to the vector fields
\[
\frac{\partial_{x_1}}{y_1}= -\frac{\partial_{y_1}}{x_1}, \partial_{y_2}, \ldots, \partial_{y_n}.
\]
From the defining equations of $Y$ in these coordinates, one deduces that the $\sO_Y$-module generated by those vector fields surjects onto $N_{Y/X}$ on the given coordinate chart, as claimed.
\end{proof}

\begin{corollary}\label{corollary normal sheaf}
Let $X$ be a holomorphic symplectic manifold and let $Y\subset X$ be a Lagrangian normal crossing subvariety. Let $S_n:=\Spec(\C[t]/(t^{n+1})$ and let $\sY_n \subset  X \times S_n$ be a deformation of $Y$ inside $X$ over $S_n$. Then the morphism $\Omega_{\sY_n/S_n} \to N_{\sY/X\times S_n}$ extends to an isomorphism $\Omega_{\sY_n/S_n}(\log) \to N_{\sY_n/X\times S_n}$. 
\end{corollary}
\begin{proof}
The argument is local around connected components of the double locus $Y^\sing$. If such a component is smoothed in the deformation $\sY_n\to S_n$, the log structure is just the divisorial log strucutre from Example~\ref{example divisorial logarithmic structure infinitesimal}. 
If the deformation is locally trivial around such a component, a computation as in Proposition~\ref{proposition normal sheaf} yields the result.
\end{proof}

It remains to prove the triviality of the sheaf of semistable relative log differentials of $Y$. We start with a weaker statement, namely the triviality of the sheaf of log differentials of the irreducible components $Y_i$ together with the canonical boundary divisor coming from the remaining components.

\begin{proposition}\label{proposition omega log trivial for snc}
Let $A$ be an abelian variety and $P:=\PP(\sO_A \oplus L)$. Let $\Ysing_0$ and $\Ysing_\infty$ be the sections corresponding to the two summands. Then the sheaf 
\[
\Omega_P(\log (\Ysing_0+\Ysing_\infty))
\]
of log $1$-forms is trivial if and only if $L\in {\pic}^0(A)$. 
\end{proposition}
\begin{proof}
Let us consider the 
residue sequence
\begin{equation}\label{eq residue sequence}
0 \to \Omega_P \to \Omega_P(\log(\Ysing_0 + \Ysing_\infty)) \stackrel{\mathrm{res}}{\longrightarrow} \sO_{\Ysing_0}\oplus \sO_{\Ysing_\infty} \to 0.
\end{equation}
Then the connecting homomorphism 
\begin{equation}\label{eq connecting homomorphism}
H^0(\sO_{\Ysing_0}\oplus \sO_{\Ysing_\infty}) \to H^1(P, \Omega_P)\isom H^{1,1}(P)
\end{equation}
is the morphism that to $1 \in H^0(\sO_{\Ysing_0})$ associates the cohomology class $[\Ysing_0]$ and analogously for $\Ysing_\infty$,
cp.~\cite[\S 2]{EV92}. 
It follows from Lemma~\ref{lemma class of section of p1 bundle in pic} that $[\Ysing_0]$ and $[\Ysing_\infty]$ are proportional (equivalently, the connecting homomorphism \eqref{eq connecting homomorphism} has rank one)  if and only if $L\in {\pic}^0(A)$. We are thus left to show that triviality of $\Omega_P(\log(\Ysing_0 + \Ysing_\infty))$ is equivalent to this sheaf having sections that do not come from $\Omega_P$.

All global sections of $\Omega_P$ come from $A$ via pullback. Assume there is a global section $s$ of $\Omega_P(\log(\Ysing_0 + \Ysing_\infty))$ that is not trivial along a fibre of $\pi:P \to A$. Then its image in $$H^0(P,\Omega_P(\log(\Ysing_0 + \Ysing_\infty))/\pi^*\Omega_A) = H^0(P,\Omega_{P/A}(\Ysing_0 + \Ysing_\infty))$$
is not identically zero. The sheaf $\Omega_{P/A}(\Ysing_0 + \Ysing_\infty)$ restricts to the trivial line bundle on any fibre of $\pi$, so if $s$ does not vanish along a fibre of $\pi$, it is already nowhere vanishing along that fibre. In particular, it maps to a nonvanishing section of $\Omega_{\Ysing_0}$ under the residue map
\[
\mathrm{res_0}: \Omega_P(\log(\Ysing_0 + \Ysing_\infty)) \to \Omega_{\Ysing_0}.
\]
But every such section is already nowhere vanishing along $\Ysing_0 \isom A$ and thus nowhere vanishing along every fibre by the previous argument. This shows that the composition $\sO_P \xrightarrow{s} \Omega_P(\log(\Ysing_0 + \Ysing_\infty))\to \Omega_{P/A}(\Ysing_0 + \Ysing_\infty)$ is an isomorphism and the claim follows.
\end{proof}

Recall the sheaf $\tOm_Y$ of torsion free differentials from Definition~\ref{definition torsion free differentials}. By~\cite[Proposition~3.6]{Leh16}, taking the image of the canonical morphism $\Omega_Y \to N_{Y/X}$ from \eqref{eq omega to tangent} we obtain a short exact sequence
\begin{equation}\label{eq omega mod torsion sequence}
0 \to \tOm_Y \to N_{Y/X} \to T^1_Y \to 0.
\end{equation}

Recall from Proposition~\ref{proposition friedman kawamata namikawa} that a $d$-semistable normal crossing variety admits a log structure of semistable type. This is not necessarily unique, but by Corollary~\ref{corollary independence of log differentials} its associated sheaf $\Omega_{Y}(\log)$ of semistable relative log differentials are independent of the choice.

\begin{theorem}\label{theorem trivial log differentials}
Let $Y=Y(A,s,L)$ be as in Construction~\ref{construction type i fiber}. Then the sheaf $\Omega_{Y}(\log)$ is trivial.
\end{theorem}
\begin{proof}
The proof is analogous to the one of Proposition~\ref{proposition omega log trivial for snc}. Let us consider the exact sequence
\[
0\to \tOm_Y \to \Omega_Y(\log) \to Q \to 0, 
\]
where $Q$ is defined as the cokernel of the first inclusion. Let $\Ysing:=Y^\sing$ denote the double locus. Taking groups of fractions, from \eqref{eq sequence divisorial log structure} and the fact that $Y$ has no triple points by~\cite[Lemma~3.2]{Leh16}, we obtain a residue map 
\[
\res\colon \Omega_Y(\log) \to \sO_\Ysing \tensor \bigl( \nu_*\Z_{Y^\nu}/\Z_Y\bigr) \cong \sO_\Ysing.
\]
Note that the last isomorphim is only canonical up to a sign since there are two canonical choices for a generator. As sections in $\tOm_Y$ have trivial residue, we obtain a morphism $Q \to \sO_\Ysing$. One verifies in local coordinates that it is an isomorphism. We consider the induced long exact sequence in cohomology
\begin{equation}\label{eq long exact sequence omega log trivial}
\scalebox{0.95}[1]{$0 \to H^0(Y,\tOm_Y) \to H^0(Y,\Omega_Y(\log)) \to H^0(\Ysing,\sO_\Ysing) \xrightarrow{\delta} H^1(Y,\tOm_Y)  \to \ldots$}
\end{equation}
In order to identify $\delta$, recall that Friedman has shown in~\cite[Proposition~1.5]{Fri83} that the complex $\tOm^\bullet_Y$ is a resolution of the constant sheaf $\C_Y$ and together with the stupid filtration computes Deligne's mixed Hodge structure on the cohomology groups $H^i(Y,\C)$ of $Y$. In particular, there is a canonical identification $H^1(Y,\tOm_Y) \isom \Gr_F^1H^2(Y,\C)$.

Let $\Ysing=\cup_{i=1}^m \Ysing_i$ be a decomposition into connected components and let $\Sigma_i^{(\infty)}\subset Y_i$ and $\Sigma_{i+1}^{(0)} \subset Y_{i+1}$ be the preimages of $\Ysing_i$ in the normalisation.

\emph{Claim:} Let $e_i \in H^0(\Ysing,\sO_\Ysing)$ be the function that is $1$ on $\Ysing_i$ and $0$ on all other connected components. Then
\begin{equation}\label{eq description connecting homomorphism}
\delta(e_i)= c_1(M_i),
\end{equation}
where $M_i$ is the unique line bundle, up to isomorphism, satisfying 
\[
M_i|_{Y_i}\isom\sO_{Y_i}(\Sigma_i^{(\infty)}), \qquad M_i|_{Y_{i+1}} \isom \sO_{Y_{i+1}}(-\Sigma_{i+1}^{(0)}) \quad \textrm{ and}
\]
\[
M_i|_{Y_k}\isom\sO_{Y_k} \qquad \textrm{ for all } k\neq i,i+1.
\]
Note that $\sO_{Y_i}(\Sigma_i^{(\infty)}) \isom \sO_{Y_{i+1}}(-\Sigma_{i+1}^{(0)})$ by $d$-semistability, which in turn is shown in Proposition~\ref{proposition construction is d semistable}. Thus, the line bundle that was just defined on the normalisation $Y^\nu=\coprod_i Y_i$ indeed descends to a line bundle $M_i$ on $Y$. Further, by Lemma~\ref{lemma picard pure weight two}, the first Chern class of a line bundle on $Y$ lies in $F^1H^2(Y,\C)$, so that we can take its class modulo $F^2H^2(Y,\C)$, which is how the right hand side of \eqref{eq description connecting homomorphism} has to be interpreted.

Admitting the claim for the moment, the proof of the theorem is completed as follows. We will first show that $\rk \delta \leq m-1$. Let us start by recalling that by the degeneration of the weight spectral sequence on $E_2$~\cite[p.71]{GS75}, there is an exact sequence
\[
H^2(Y,\Q) \xrightarrow{\nu^*} H^2(Y^\nu,\Q) \to H^2(\Ysing,\Q)
\]
and the kernel of the first morphism is exactly the part $W_{1}H^2(Y,\Q)$ of the weight filtration. The map $\nu^*$ sends $c_1(M_i) \in H^2(Y,\Q)$ to $c_1\bigl(\sO_{Y_i}(\Sigma_i^{(\infty)})\bigr) - c_1\bigl(\sO_{Y_{i+1}}(\Sigma_{i+1}^{(0)})\bigr) \in H^2(Y^\nu,\Q)$. 
By Lemma~\ref{lemma class of section of p1 bundle in pic}, we have
\[
\sO_{Y_i}\bigl(\Sigma_i^{(\infty)}-\Sigma_i^{(0)}\bigr)
\cong
\pi_i^*L.
\]
Hence, we deduce $c_1\bigl(\sO_{Y_i}(\Sigma_i^{(\infty)})\bigr) = c_1\bigl(\sO_{Y_i}(\Sigma_i^{(0)})\bigr)$ and obtain that
\[
e_1+\cdots+e_m\in \ker(\nu^*\circ \delta).
\] Since $\nu^*$ is injective on the image of the Picard group by Lemma~\ref{lemma picard pure weight two}, we obtain the desired inequality $\rk(\delta)\leq m-1$. Since $\dim H^0(\Ysing,\sO_\Ysing)=m$, we conclude from \eqref{eq long exact sequence omega log trivial} that $H^0(Y,\Omega_Y(\log))$ has a section not contained in $H^0(Y,\tOm_Y)$. One concludes as in the proof of Proposition~\ref{proposition omega log trivial for snc} that such a section is everywhere linearly independent from the sections in $H^0(Y,\tOm_Y)$. By Lemma~\ref{lemma gluing invariant is translation}, we have $\dim H^0(Y,\tOm_Y) = \dim A$ so that $\Omega_Y(\log)$ is generated by $\dim Y$ many global sections. Therefore, it has to be trivial. It remains to verify the claim.

\emph{Proof of Claim:} 
Let us first prove that $\delta$ indeed takes values in the image of the Picard group in $H^1(Y,\tOm_Y)$. For this, consider the the following morphism of short exact sequences.
\[
\xymatrix{
1 \ar[r] & \sO_Y^\times \ar[r] \ar[d]_{d\log} & M_{\mathrm{rel}}^{\gp} \ar[r] \ar[d]^{d\log} & \Z_{\Ysing} \ar[r] \ar[d] & 0 \\
0 \ar[r] & \widetilde\Omega_Y \ar[r] & \Omega_Y(\log) \ar[r] & \sO_\Ysing \ar[r] & 0.
}
\]
Here, $M^\gp_{\mathrm{rel}}$ denotes the group of fractions of the quotient $M/L$, where $M$ is a chosen log structure of semistable type and $L$ is the pullback of the log structure from the standard log point. Note that $\Z_\Ysing$ is generated by the sections $e_i$ and the rightmost vertical map is just the inclusion. This diagram induces the commutative diagram
\[
\xymatrix{
H^0(\Ysing,\Z_{\Ysing}) \ar[r] \ar[d] & H^1(Y,\sO_Y^\times) \ar[d] \\
H^0(\Ysing,\sO_\Ysing) \ar[r]^\delta & H^1(Y,\widetilde\Omega_Y),\\
}
\]
where the horizontal maps are connecting homomorphisms. As observed before, the equality $H^0(\Ysing,\sO_\Ysing) = H^0(\Ysing,\Z_{\Ysing})\tensor_\Z \C$ holds so that $\delta$ indeed takes values in the image of the Picard group.

Next we describe the composition $\nu^*\circ \delta$. 
By Corollary~\ref{corollary relative log differentials on normalization component}, the residue sequence 
\[
0\to \widetilde\Omega_Y
\to \Omega_Y(\log)
\to \sO_\Ysing
\to 0
\]
pulls back to the exact sequence
\[
\nu_j^*\tOm_Y \to \Omega_{Y_j}\bigl(\log(\Sigma_j^{(0)}+\Sigma_j^{(\infty)})\bigr) \to \sO_{\Sigma_j^{(0)}}\oplus \sO_{\Sigma_j^{(\infty)}} \to 0,
\]
where again the signs for the last map are determined by the identification of $\nu_*\Z_{Y^\nu}/\Z_Y$ with $\Z_{\Ysing_i}$ locally on each $\Ysing_i$. Replacing the first morphism by its image, we obtain the usual residue sequence

\[
0\to \Omega_{Y_j} \to \Omega_{Y_j}\bigl(\log(\Sigma_j^{(0)}+\Sigma_j^{(\infty)})\bigr) \to \sO_{\Sigma_j^{(0)}}\oplus \sO_{\Sigma_j^{(\infty)}} \to 0, 
\]
as a direct computation shows. With the aforementioned choice of identifications, the section $e_i$ restricts on the normalisation as
\[
\begin{aligned}
e_i&\mapsto \phantom{-}1_{\Sigma_i^{(\infty)}}\quad 
&&\text{on }Y_i,\\
e_i&\mapsto -1_{\Sigma_{i+1}^{(0)}}
&&\text{on }Y_{i+1},\\
\end{aligned}
\]
and to $0$ on all other components. 
By Proposition~\ref{proposition omega log trivial for snc} applied to $Y_i$, the connecting homomorphism for \eqref{eq residue sequence} maps
\[
1_{\Sigma_i^{(\infty)}} \mapsto c_1\bigl(\sO_{Y_i}(\Sigma_i^{(\infty)})\bigr).
\]
Similarly, on $Y_{i+1}$ it maps 
\[
1_{\Sigma_{i+1}^{(0)}} \mapsto c_1\bigl(\sO_{Y_{i+1}}(\Sigma_{i+1}^{(0)})\bigr)
\]
so that 
\[
\left.\nu^*\delta(e_i)\right|_{Y_i} = c_1\bigl(\sO_{Y_i}(\Sigma_i^{(\infty)})\bigr),
\]
\[
\left.\nu^*\delta(e_i)\right|_{Y_{i+1}} = - c_1\bigl(\sO_{Y_{i+1}}(\Sigma_{i+1}^{(0)})\bigr),
\]
and
\[
\left.\nu^*\delta(e_i)\right|_{Y_j}=0 \qquad\text{for }j\neq i,i+1,
\]
as needed. Again, the morphisms of short exact sequences give rise to the commutative diagram
\[
\xymatrix{
H^0(\Ysing,\sO_\Ysing) \ar[r]^\delta \ar@{^(->}[0,0]!<0pt,-2.5ex>;[d]+<0pt,+2.8ex> & H^1(Y,\widetilde\Omega_Y)\ar[d]+<0pt,+2.8ex>\\
\displaystyle\bigoplus_{i=1}^m H^0\bigl(\sO_{\Sigma_i^{(0)}}\oplus \sO_{\Sigma_i^{(\infty)}} \bigr) \ar[r] & \displaystyle\bigoplus_{i=1}^m H^1(Y_i,\Omega_{Y_i}).\\
}
\]
Since we have already seen that the image of $\delta$ lies in the image of the Picard group, the right vertical map is injective on the image of $\delta$. Hence, the description of $\nu^*\circ \delta$ given above is sufficient to determine $\delta$. The claim follows.
\end{proof}

\begin{lemma}\label{lemma rational sub in f1 pure weight two}
Let $H$ be a rational mixed Hodge structure with weights in the interval $[0,2]$, and assume that $\Gr^W_0H$ is of type
$(0,0)$. Let $P\subset H$ be a rational subspace such that
$P_{\C}\subset F^1H_{\C}$. Then the natural map
\[
P \to \Gr^W_2 H
\]
is injective. 
\end{lemma}
\begin{proof}
We show that $P\cap W_1H=0$. Let $v\in P\cap W_1H$. Since $P_{\C}\subset F^1H_{\C}$, the class of $v$ in
$\Gr^W_1H$ lies in $F^1\Gr^W_1H_{\C}$. Recall that the Hodge filtration on the weight graded pieces is defined by
\[
F^p\Gr^W_kH_{\C}
= \frac{(F^pH_{\C}\cap W_kH_{\C})+W_{k-1}H_{\C}} {W_{k-1}H_{\C}}.
\]
But $\Gr^W_1H$ is pure of weight $1$, so
\[
\Gr^W_1H_{\C}
= H^{1,0}\oplus H^{0,1}, \qquad F^1\Gr^W_1H_{\C}=H^{1,0}.
\]
A rational class lying in $H^{1,0}$ must be zero, since it is fixed by complex conjugation. Hence, $v\in W_0H$. But since 
$\Gr^W_0H$ is of type $(0,0)$, we have
\[
F^1\Gr^W_0H_{\C}=0 
\]
so that $v=0$. The claim follows.
\end{proof}

Recall that a reduced compact Kähler space admits a resolution of singularities by a compact Kähler manifold. From this, one may construct a cubical hyperresolution consisting only of compact Kähler manifolds; see \cite[Theorem~I.2.6]{GNPP88} or~\cite[Theorem~5.25]{PS08}. Therefore, the cohomology groups of a compact Kähler space possess a functorial mixed Hodge structure. 

\begin{lemma}\label{lemma picard pure weight two}
Let $X$ be a compact Kähler space. Then the image of the map $c_1\colon \pic(X) \to H^2(X,\Q)$ is contained in $F^1H^2(X,\C)$ and injects into $\Gr_2^WH^2(X,\Q)$. 
\end{lemma}
\begin{proof}
Since $c_1\colon \pic(X) \to H^2(X,\Q)$ factors through $\pic(X_\red)$ by the exponential sequence, we may assume that $X$ is reduced. Since $X$ is compact Kähler, its singular cohomology groups carry a functorial mixed Hodge structure. As in~\cite[Theorem~5.39]{PS08} one shows that the weights of $H^2(X,\Q)$ lie in the interval $[0,2]$. Thus, the second statement follows from the first by Lemma~\ref{lemma rational sub in f1 pure weight two}. Therefore, it remains to show that $c_1\left(\pic(X)\right)_\C \subset F^1H^2(X,\C)$. This in turn is deduced from the exponential sequence and the canonical morphism $\sO_X \to \underline{\Omega}_X^0$ where $\underline{\Omega}_X^0$ denotes the $0$-th graded piece of the du Bois complex. Recall that $\mathbb H^2(X,\underline{\Omega}_X^0) \isom \Gr_F^0H^2(X,\C)$.
\end{proof}

We may now assemble the pieces and obtain the proof of the section's main result.

\begin{proof}[Proof of Theorem~\ref{theorem trivial normal bundle}]
By Proposition~\ref{proposition normal sheaf}, the normal sheaf of $Y$ in $X$ is isomorphic to the sheaf of log differentials $\Omega_Y(\log)$. From Proposition~\ref{proposition construction is d semistable}, the variety $Y'$ is $d$-semistable. Using Lemma~\ref{lemma log differentials pull back}, we have
\[
p^*N_{Y/X}\isom p^*\Omega_Y(\log) \isom \Omega_{Y'}(\log)
\]
Then $\Omega_{Y'}(\log)$ is trivial by Theorem~\ref{theorem trivial log differentials}, and this completes the proof.
\end{proof}

%----------------------------------------------------
%----------------------------------------------------
\section{Unobstructedness Results}\label{section unobstructedness}
%----------------------------------------------------
%----------------------------------------------------

The following result generalises results from~\cite{Voi92,Leh16}. It is the main ingredient in the proof of our main theorem in Kodaira type I. 
Let $f\colon Y' \to X$ be a morphism of compact analytic spaces.  Recall from Notation~\ref{notation deformation functors} the functor $D_f^\ft$ for deformations with fixed target associated to such a morphism. 

\begin{theorem}\label{theorem normal crossing cover unobstructedness}
Let $X$ be a holomorphic symplectic manifold, let $Y \subset X$ be a Lagrangian normal crossing subvariety, let $p\colon Y'\to Y$ be a finite étale cover, and denote $f\colon Y' \to X$ the composition of $p$ with the inclusion. Then the functor $D_f^\ft$ is unobstructed. The tangent space is given by
\[
T^1_{D_f^\ft}= H^0(Y',p^*N_{Y/X}) = H^0(Y',\Omega_{Y'}(\log)).
\]
\end{theorem}
\begin{proof}
Let $K_f^\bullet$ be the complex $\bigl[f^*\Omega_X\to \Omega_{Y'}\bigr]$ concentrated in degrees $-1,0$.  Recall from Corollary~\ref{corollary tangent spaces} that the tangent space our deformation functor~is
\[
T^1_{D_f^\ft} = \mathbb E xt^1_{Y'}(K_f^\bullet,\sO_{Y'}).
\]
Since $f$ is a composition $\xymatrix{Y' \ar[r]^p & \,Y\, \ar@{^(->}[r]^\iota & X}$, where $p$ is étale and $\iota$ is a closed embedding, the nonzero differential in $K_f^\bullet$ is actually surjective and we obtain a quasi-isomorphism $p^*I/I^2[1] \to K_f^\bullet$, where $I\subset \sO_X$ is the ideal sheaf of $Y$. It follows that
\[
T^1_{D_f^\ft} = \mathbb E xt^1_{Y'}(p^*I/I^2[1],\sO_{Y'}) = H^0(Y',p^*N_{Y/X}).
\]
Here, we used that $\iota$ is a regular embedding so that $I/I^2$ is locally free. In addition, $Y$ is Lagrangian by assumption, so that $N_{Y/X} \isom \Omega_{Y}(\log)$ by Proposition~\ref{proposition normal sheaf}. Since $p$ is étale, $p^*\Omega_{Y}(\log) = \Omega_{Y'}(\log)$ by Lemma~\ref{lemma log differentials pull back}. This proves the claim about the tangent space. 

For the unobstructedness, we would naively like to use a relative version of the above argument. The problem is that $I/I^2$ is no longer the right object, since the image of a deformation is not necessarily a deformation of the image. One can circumvent this by either gluing local ideal sheaves or, as we chose, lift the étale cover $p$ to an open neighbourhood of $Y$ inside $X$. Let us first set up some notation.

For $n\in \N$, we denote $A_n:=\C[t]/(t^{n+1})$ and $S_n:=\Spec(A_n)$. 
To prove smoothness for a deformation functor $D\colon \Art_\C \to \Set$, by the $T^1$-lifting principle (see Section~\ref{section t1 lifting}), it suffices to show that for a deformation $\xi_n \in D(A_n)$ with image $\xi_{n-1}\in D(A_{n-1})$, the map between Kawamata's $T^1$-spaces $T^1(\xi_n/S_n) \to T^1(\xi_{n-1}/S_{n-1})$ is surjective. 

Let $F_n\colon \sY'_n \to X\times S_n$ be a deformation of $f$ over $S_n$, let $\sY'_{n-1}:=\sY'_n\times_{S_n}S_{n-1}$ and denote by $F_{n-1}\colon \sY'_{n-1}\to X \times S_{n-1}$ the base change of $F_n$.

\emph{Claim:} There is a canonical isomorphism $T^1(F_n/S_n)\xrightarrow{\;\isom\;}\Omega_{\sY'_n/S_n}(\log)$ and under this isomorphism, the map becomes the ordinary restriction of differential forms.

Admitting the claim for the moment, we infer from~\cite[Theorem~1.10]{FFR21} that the $H^q(\sY'_n,\Omega^p_{\sY'_n/S_n}(\log))$ are locally free and compatible with base change for all $p$ and $q$. The sought-for surjectivity follows from the commutativity of the following diagram of $A_n$-modules.
\[
\xymatrix@C=1cm{
T^1(F_n/S_n) \ar[r]\ar[d]_\isom & T^1(F_{n-1}/S_{n-1}) \ar[d]^\isom\\
H^1(\sY'_n,\Omega_{\sY'_n/S_n}(\log)) \ar@{->>}@(dr,l)[rd] & H^1(\sY'_{n-1},\Omega_{\sY'_{n-1}/S_{n-1}}(\log))\\
&  H^1(\sY'_n,\Omega_{\sY'_n/S_n}(\log))\tensor_{A_n} A_{n-1}\ar@{=}[u]\\
}
\]
To complete the argument, it remains to verify the claim.

\emph{Proof of the Claim.} Let us denote by $K_{F_n}^\bullet$ the complex $\bigl[F_n^*\Omega_{X\times S_n/S_n}\to \Omega_{\sY'_n/S_n}\bigr]$ concentrated in degrees $-1,0$. By Corollary~\ref{corollary kawamata t1 for fixed target}, there is an isomorphism
\[
T^1(F_n/S_n)\xrightarrow{\;\isom\;}\mathbb E xt^1_{\sY_n'}(K_{F_n}^\bullet,\sO_{\sY_n'}).
\]
We would like to identify the right hand side with $H^0(\sY_n',\Omega_{\sY'_n/S_n}(\log))$. As explained above, a technical obstacle is that the ideal sheaf of the image is no longer the correct thing to look at. We solve this problem by embedding $\sY'$ into an étale cover of an open neighbourhood of $Y$ in $X$. By~\cite[I, Theorem~8.8]{BHPV04}, there exists an open neighbourhood $U \subset X$ of $Y$ such that the inclusion $Y \subset U$ is a deformation retract. In particular, the induced morphism $\pi_1(Y,y) \to \pi_1(U,y)$ on fundamental groups is an isomorphism for any choice of basepoint $y\in Y$. Hence, there is a diagram
\[
\xymatrix{
Y' \ar@{^(->}[d]\ar[r]^p & Y \ar@{^(->}[d]\ar[r] & X \\
U' \ar[r]^\varpi & U \ar@{^(->}[ru] &
}
\]
where $\varpi$ is a finite étale cover and pushforward along $Y'\to U'$ is an isomorphism on fundamental groups. Since $F_n\colon \sY_n \to X\times S_n$ factors through $U\times S_n$ and $\varpi$ is a topological cover, there is a unique lift in the following diagram.
\[
\xymatrix{
Y' \ar[d]\ar[r] & U' \ar[r] &U'\times S_n \ar[d]\\
\sY_n' \ar[rr]^{F_n}\ar@{-->}[rru]^{\exists ! {F'_n}} && U\times S_n\\
}
\]
In the algebraic category, this follows from the lifting property for formally étale morphisms; see~\cite[Corollaire~17.6.2]{EGAIV4}. In the analytic category, we did not find a suitable reference, but the claim about existence and uniqueness of lifts follows easily from the fact that an étale morphism has local inverses. With $Y' \to U'$ also $F'_n\colon \sY_n \to U'\times S_n$ is a closed immersion. Since $U'\to X$ is étale, we identify
\[
K_{F_n}^\bullet \isom \bigl[{F'_n}^*\Omega_{U'\times S_n/S_n}\to \Omega_{\sY'_n/S_n}\bigr] \isom J_n/J_n^2[1],
\]
where $J_n \subset \sO_{U'\times S_n}$ is the ideal sheaf of $\sY_n'$. As in the absolute case, we deduce
\[
T^1(F_n/S_n) \isom H^0(\sY_n',N_{\sY_n'/U'\times S_n}).
\]
The claim now follows from Corollary~\ref{corollary normal sheaf}. We conclude the proof.
\end{proof}

The following theorem is the relevant application for our main result. Recall that deformations of morphisms which are étale over closed subvarieties possess a universal deformation by Proposition~\ref{proposition universal deformation}.

\begin{theorem}\label{theorem normal crossing cover corank one smoothing}
Let $X$ be a compact hyperkähler manifold of dimension~$2n$, and let $Y \subset X$ be a normal crossing Lagrangian closed subvariety with torsion dualizing sheaf $\omega_Y$. Let $f\colon Y' \to X$ be the composition of the index one cover $p\colon Y'\to Y$ with the inclusion $Y\subset X$. Assume that $h^0(\tOm_Y) = n-1$ and that $-K_{Y^\nu}$ is nef where $Y^\nu$ denotes the normalisation. Let 
\begin{equation}\label{eq universal deformation}
\xymatrix{
\mathfrak U \ar[r]^{F} \ar[d]_\pi & X\\
\mathfrak D &\\
}
\end{equation}
be a representative of the universal deformation of $f$. Then $\dim \mathfrak D = n$, the map $F$ is generically étale (in fact, étale in a neighbourhood of ${Y'}^\reg\subset \mathfrak U$), and the general fibre of $\pi$ maps isomorphically to a Lagrangian subvariety of $X$ that is an abelian variety. 
\end{theorem}
\begin{proof}
By Theorem~\ref{theorem normal crossing cover unobstructedness}, the functor $D^\ft_f$ is unobstructed with tangent space
\[
T^1:=T^1_{D^\ft_f} = H^0(Y',\Omega_{Y'}(\log)).
\]
From~\cite[Theorem~1.1]{Leh16}, we deduce that $Y$, and hence also $Y'$, is projective. Thus, by Theorem~\ref{theorem fibers nc case}, either $Y'$ is abelian or there is an abelian variety $A$, a point $s\in A$, and a line bundle $L\in \pic^0(A)$ such that $Y'\isom Y_m(A,s,L)$ as in Construction~\ref{construction type i fiber}. 
In the first case, $\Omega_{Y'}(\log)=\Omega_{Y'}$ as $Y'$ is smooth and $\Omega_{Y'}$ is trivial for abelian varieties. In the second case, $\Omega_{Y'}(\log)$ is trivial by Theorem~\ref{theorem trivial log differentials}. In either case, $T^1$ has dimension $n$, hence the first claim follows.

Let us denote by $[f]\in \mathfrak D$ the distinguished point and note that the fibre of $\pi$ over it is $\mathfrak U_{[f]} = Y'$ with $F|_{Y'}=f$. For the second claim, we will compute the rank of the differential of $F_*:T_{\mathfrak U} \to F^*T_X$ at a general point of $Y'$. For this purpose, fix a smooth point $y'\in Y'$ and consider the commutative diagram 
\[
\xymatrix{
0 \ar[r] & T_{Y',y'} \ar[r] \ar@{=}[d] & T_{\mathfrak U,y'} \ar[r] \ar[d]^{F_*}  & T_{\mathfrak D,[f]} \ar[d]^{\ev} \ar[r] & 0 \\
0 \ar[r] & T_{Y',y'} \ar[r] & T_{X,f(y')} \ar[r] & N_{Y/X,f(y')} \ar[r] & 0 \\
}
\]
between the tangent spaces. Since the deformation problem is unobstructed and $p$ is étale, $y'$ is a smooth point of $\mathfrak U$ and $f(y')$ is a smooth point of $Y$. This shows that the rows in the diagram are exact. It follows from deformation theory\footnote{This follows directly from the construction of first order deformations out of global sections of the normal sheaf. It can be found in~\cite[Proposition~4.8]{Deb01} or~\cite[Proposition~3.4]{Kol96} when the source is a rational curve. The geometry of the source is however irrelevant to the argument.} that $\ev$ is the natural evaluation map from $T_{\mathfrak D,[f]} = H^0(Y',p^*N_{Y/X})$ to $N_{Y/X,f(y')}$, that is, $\ev(s) = s(f(y'))$. But $p^*N_{Y/X} \isom \Omega_{Y'}(\log)$ is trivial and thus in particular generated by global sections. This implies that $\ev$ is surjective (and even an isomorphism for dimension reasons). The same is true for $F_*$ by the five lemma. Therefore,  $F$ is étale at a general point of $\mathfrak U$. 

For the last claim, we will show first that the general fibre of $\pi$ is an abelian variety and the pullback of the symplectic form on $X$ to it is identically zero. If $Y'$ is an abelian variety, it is clear that the general fibre of $\pi$ is a complex torus of dimension $n$. If instead $Y'\isom Y_m(A,s,L)$, the sheaf $\Omega_{Y'}(\log)$ has a nowhere vanishing section which is everywhere linearly independent from any element in $H^0(Y',\tOm_{Y'})$. Denote $\eta \in H^0(Y',\Omega_{Y'}(\log))$ such a section. By~\eqref{eq residue sequence}, it has nonvanishing residues at all points of the double locus ${Y'}^\sing$. Then a computation as in Proposition~\ref{proposition normal sheaf} shows that the image of $\eta$ under 
\[
H^0(Y',\Omega_{Y'}(\log)) \isom H^0(Y',\pi^*N_{Y/X}) \xrightarrow{\gamma} H^0(Y',T^1_{Y'})
\]
generates $T^1_{Y'}$ at every point. The induced deformation of the germ $(Y',y')$ is given by a the germ $\gamma(\eta)_{y'}\in T^1_{Y',y'}$. Recall that $T^1_{Y'}$ is (the push forward of) a line bundle on ${Y'}^\sing$. Hence, locally around $y'$, we can trivialise $T^1_{Y',y'}\isom \sO_{{Y'}^\sing,y'}$ and choose a representative $u \in \sO_{Y',y'}^\times$ of the function corresponding to $\gamma(\eta)_{y'}$. 
If $Y'$ is given by an equation $xy=0$ in local coordinates around $y'$, then the first order deformation (i.e. over the ring of dual numbers $\C[\varepsilon]$) induced by $\eta$ of the germ $(Y',y')$ can be written as $xy-\varepsilon u=0$; see e.g. the proof of~\cite[Corollary~2.1.25]{GLS25} or apply~\cite[Corollary~2.1.17]{GLS25} to a slice transversal to ${Y'}^\sing$. We conclude that any $1$-parameter deformation of $f$ with tangent direction $\eta$ is a smoothing of $Y'$. 

Take a disk $\Delta \subset \mathfrak D$ through $[f]$ with tangent direction $\eta$ and denote $\mathfrak U_\Delta:=\mathfrak U\times_{\mathfrak D} \Delta \xrightarrow{\pi_\Delta} \Delta$. Then $\pi_{\Delta *}\Omega_{\mathfrak U_\Delta/\Delta}(\log Y')$ is locally free of rank $n$ and compatible with arbitrary base change; see e.g.~\cite[Corollary~11.18]{PS08}. 
But the fibre of $\pi_{\Delta *}\Omega_{\mathfrak U_\Delta/\Delta}(\log Y')$ at zero is precisely $H^0(Y',\Omega_{Y'}(\log))$, hence $H^0(\mathfrak U_t,\Omega_{\mathfrak U_t})$ is of dimension $n$ for all $t\in \Delta\setminus\{0\}$, again by Theorem~\ref{theorem trivial log differentials}. Moreover, we find $n$ linearly independent sections of $\pi_{\Delta *}\Omega_{\mathfrak U_\Delta/\Delta}(\log Y')$ over $\Delta$ whose restrictions generate $\Omega_{\mathfrak U_t}$ at any point $x \in \mathfrak U_t$. Thus, $\Omega_{\mathfrak U_t}$ is trivial and $\mathfrak U_t$ is a complex torus. 

Let us prove that $\mathfrak U_t$ is Lagrangian for general $t\in \mathfrak D$. As in~\cite[Lemma~5.17]{BL22}, one shows that the fundamental class $[\mathfrak U_t]\in H_{2n}(X,\Z)$ does not depend on $t\in\mathfrak D$. In particular, $F_*[\mathfrak U_t]=\deg(p) [Y]$ for all $t$. Let $\kappa$ be a Kähler form on $X$. Since $Y$ is Lagrangian, $\int_Y \sigma\wedge \bar\sigma\wedge\kappa^{\wedge n-2} = 0$. Put together,
\[
\int_{\mathfrak U_t} F^*\sigma\wedge F^*\bar\sigma\wedge \kappa^{\wedge n-2} = \deg(p) \int_Y \sigma\wedge \bar\sigma\wedge\kappa^{\wedge n-2} =0
\]
for all $t\in \mathfrak D$. For smooth fibres of $\pi$, this implies that they are Lagrangian with respect to $F^*\sigma$. Note that this form is symplectic along those fibres since $F$ is étale there. So their images are Lagrangian with respect to $\sigma$. 

Again by~\cite[Theorem~1.1]{Leh16}, the general fibre of $\pi$ is projective. 
Put together, we conclude that in both cases, the general fibre $\pi$ is an abelian variety and the pullback of the symplectic form to it vanishes.

It remains to prove that $F$ is an embedding when restricted to the general fibre of $\pi$. By construction, the index one cover $p$ is an étale Galois cover with Galois group $G=\Z/\ell\Z$, where $\ell=\deg p$. Since $f\colon Y' \to X$ is invariant for the action of $G$, this group acts (after potentially shrinking the representative $\mathfrak D$) on both $\mathfrak D$ and $\mathfrak U$ such that $\pi$ is equivariant, $[f]\in\mathfrak D$ is a fixed point, and the action on $\mathfrak U$ restricts to the the Galois action on $Y'=\mathfrak U_{[f]}$; see Proposition~\ref{proposition group action}. Next observe that (again after possibly shrinking $\mathfrak D$) we have $\deg F = \deg p$, where the left hand side is meant to denote the degree of $F$ as a morphism onto its image. Since $G$ acts without fixed points on $Y=\mathfrak U_{[f]}$, it does so on the whole of $\mathfrak U$. We obtain a factorisation of \eqref{eq universal deformation} as follows.
\[
\xymatrix{
\mathfrak U \ar[r] \ar@/^1.5pc/[rr]^{F} \ar[d]_\pi & \mathfrak V:=\mathfrak U/G \ar[r]_(0.6){\bar F} \ar[d]_\rho & X\\
\mathfrak D \ar[r]^(0.4)h & \mathfrak E:=\mathfrak D/G&\\
}
\]
Observe that the following hold. 
\begin{enumerate}
	\item The morphism $\bar F$ has degree one onto its image and hence, since $X$ is normal, is an open embedding. 
	\item The general fibres of $\rho$ are étale quotients of abelian varieties and Lagrangian subvarieties of $X$. 
\end{enumerate}
If the general fibres of $\rho$ were themselves abelian, the claim would follow. That in fact they are can be seen as follows. Let $V:=\mathfrak V_t$ be a general fibre of $\rho$. Since the image of $\bar F$ is open, in particular, of dimension $2n$, the Douady space of $X$ has dimension at least $n$ at $[V]$. The same holds for its tangent space $H^0(V,N_{V/X}) \isom H^0(V,\Omega_V)$. But since $V$ is an étale quotient of an abelian variety, the inequality $h^0(V,\Omega_V)\geq n$ implies that $V$ is itself abelian.
\end{proof}

%----------------------------------------------------
%----------------------------------------------------
\section{Proof of the Main Result}\label{section proof main result}
%----------------------------------------------------
%----------------------------------------------------

In this section, we assemble the results from the previous sections to prove our main theorems. We start with Theorem~\ref{theorem main fibration}, since it will be used in the proof of Theorem~\ref{theorem main semiampleness}.

\begin{proof}[Proof of Theorem~\ref{theorem main fibration}]
With the notation of \eqref{eq universal deformation}, we consider the universal deformation of $f$. By Theorem~\ref{theorem normal crossing cover corank one smoothing}, the general fibre of $\pi$ maps onto a Lagrangian subvariety of $X$ that is an abelian variety. If $X$ is nonprojective, the claim follows from~\cite[Theorem~4.1]{GLR13}. For projective $X$, one uses~\cite[Theorem~1.2]{HW13} and~\cite[Corollary~1]{Mat14}.
\end{proof}

\begin{remark}\label{remark hulek spandaw}
The hypothesis that $-K_{Y^\nu}$ is nef in Theorem~\ref{theorem main fibration} is natural from the point of view of degenerations of abelian varieties. It is used there through Theorem~\ref{theorem normal crossing cover corank one smoothing} to ensure that $Y'\isom Y_m(A,\tau,L)$ as in Construction~\ref{construction type i fiber}, where $Y'\to Y$ is the index one cover. At this point, the nefness of $-K_{Y^\nu}$ equivalent to $L\in\pic^0(A)$ by Corollary~\ref{corollary nef iff not big}; cp. e.g.~\cite[Proposition~2]{HS96}.
\end{remark}

\begin{lemma}\label{lemma singular fiber vs divisor}
Let $X$ be a compact hyperkähler manifold of dimension $2n$. Suppose that there is a prime divisor $D \subset X$ with an almost holomorphic fibration $\phi\colon D\ratl U$ over a variety $U$ of dimension $n-1$ such that the general fibre $Y$ of $\phi$ is Lagrangian of Kodaira type I. Then $h^0(Y,\tOm_Y)\geq n-1$.
\end{lemma}
\begin{proof}
Since the general fibre of $\phi$ has normal crossing singularities, there is a dense open set $U^\circ \subset U$, over which the restriction $\phi|_{U^\circ}$ is a locally trivial deformation (in the sense of~\cite[Definition~4.1]{BL22}) of any of its fibres: to see this, one chooses $U^\circ$ such that $\phi|_{U\circ}$ and its restriction to the singular locus of $\phi$ are flat, and the fibres have no worse than normal crossing singularities. By properness, this forces the number of connected components of the singular locus of $\phi$ to be constant. For normal crossing singularities without triple points, any deformation is either locally trivial or a smoothing. The latter would contradict the flatness assumption on the singular locus of $\phi$. 

Since $\dim U = n-1$ and the deformation $\phi$ is locally trivial, we have that 
\[
n-1\leq \dim T_{U,t} \leq \dim T^1_{D_{Y\subset X}^\lt} = h^0(Y,N'_{Y/X}),
\]
where $N'_{Y/X}:=\ker\left(N_{Y/X}\to T^1_Y\right)$. Since $$H^0(Y,N'_{Y/X}) \isom H^0(Y,\tOm_Y)$$ by~\cite[Proposition~3.6]{Leh16}, the claim follows.
\end{proof}

We come to the proof of our main theorem. It is essentially a consequence of Theorem~\ref{thm:only-case-left} and the special case $n=2$ of Theorem~\ref{theorem main fibration}.

\begin{proof}[Proof of Theorem~\ref{theorem main semiampleness}]
If $\kappa(X,L)>0$, the result follows from~\cite[Theorem~4.1]{HLL25}. 
Arguing by contradiction let us assume that $\kappa(X,L)=0$. Then by Theorem~\ref{thm:only-case-left}, there is a unique prime divisor $D\subset X$, and the canonical divisor $K_D$ is semiample of numerical dimension one such that the corresponding fibration $\phi\colon D \to U$ has general fibres of Kodaira type I. More preceisely, let $Y:=\phi^{-1}(t) \subset D$ be a fibre over a general point $t\in U$. 
Then $Y$ has normal crossing singularities without triple points, and the dualizing sheaf $\omega_Y$ is a torsion line bundle by adjunction.

Moreover, $Y$ is uniruled by Lemma~\ref{lemma:lagrangian2}. In particular, $Y\subset X$ is Lagrangian. Hence, we infer $h^0(Y,\tOm_Y) \geq~1$ from Lemma~\ref{lemma singular fiber vs divisor}. Again by Lemma~\ref{lemma:lagrangian2}, the divisor $K_{Y^\nu}$ is not big. But then $-K_{Y^\nu}$ is nef by Corollary~\ref{corollary nef iff not big}. 
Then we may deduce from Corollary~\ref{corollary p1 bundles from locally trivial deformations} applied to the index one cover of~$Y$ that this is in fact an equality. 
From Theorem~\ref{theorem main fibration}, we conclude that there is a Lagrangian fibration $f\colon X\to B$. The pull back of a very ample line bundle along $f$ thus has positive Kodaira dimension, contradicting the assumptions $\pic(X) = \Z L$ and $\kappa(X,L)=0$.
\end{proof}

\bibliographystyle{amsalpha}
\bibliography{biblio}

\end{document}